\documentclass[11pt]{amsart}

\usepackage{amscd}
\usepackage{amsmath, amssymb, verbatim}
\usepackage{amsfonts}
\newcommand{\de}{\partial}
\newcommand{\db}{\overline{\partial}}

\newcommand{\ddbar}{\sqrt{-1} \partial \overline{\partial}}
\newcommand{\Ric}{\mathrm{Ric}}
\newcommand{\ov}[1]{\overline{#1}}
\newcommand{\mn}{\sqrt{-1}}

\newcommand{\tr}[2]{\mathrm{tr}_{#1}{#2}}
\newcommand{\ti}[1]{\tilde{#1}}
\newcommand{\vp}{\varphi}
\newcommand{\vol}{\mathrm{Vol}}
\newcommand{\diam}{\mathrm{diam}}
\newcommand{\Null}{\mathrm{Null}}
\newcommand{\ve}{\varepsilon}

\renewcommand{\leq}{\leqslant}
\renewcommand{\geq}{\geqslant}
\renewcommand{\le}{\leqslant}
\renewcommand{\ge}{\geqslant}

\numberwithin{equation}{section}

\begin{document}
\newtheorem{claim}{Claim}
\newtheorem{theorem}{Theorem}[section]
\newtheorem{lemma}[theorem]{Lemma}
\newtheorem{corollary}[theorem]{Corollary}
\newtheorem{proposition}[theorem]{Proposition}
\newtheorem{question}[theorem]{Question}
\newtheorem{conjecture}[theorem]{Conjecture}
\newtheorem{problem}[theorem]{Problem}

\theoremstyle{definition}
\newtheorem{remark}[theorem]{Remark}
\newtheorem{example}[theorem]{Example}

\title[KAWA lecture notes on the K\"ahler-Ricci flow]{KAWA lecture notes on the K\"ahler-Ricci flow}
\author{Valentino Tosatti}
\address{Department of Mathematics, Northwestern University, 2033 Sheridan Road, Evanston, IL 60208}
\email{tosatti@math.northwestern.edu}
\thanks{The author is supported in part by NSF grant DMS-1308988 and by a Sloan Research Fellowship.}

\begin{abstract}
These lecture notes provide an introduction to the study of the K\"ahler-Ricci flow on compact K\"ahler manifolds, and a detailed exposition of some recent developments.
\end{abstract}
\maketitle
\tableofcontents

\section{Introduction}

The Ricci flow on a compact K\"ahler manifold $X$, starting at a K\"ahler metric $\omega_0$, preserves the K\"ahler condition in the sense that the evolved metrics are still K\"ahler. It is then customary to call this flow the K\"ahler-Ricci flow, and to write it as an evolution equation of K\"ahler forms as
\begin{equation}\label{krf}
\left\{
                \begin{aligned}
                  &\frac{\de}{\de t}\omega(t)=-\Ric(\omega(t))\\
                  &\omega(0)=\omega_0.
                \end{aligned}
              \right.
\end{equation}
The theory of the K\"ahler-Ricci flow is rather well-developed, and the key feature is that the behavior of the flow deeply reflects the complex structure of the manifold $X$.

In particular, there is a conjectural picture of the behavior of the K\"ahler-Ricci flow for {\em any} initial data $(X,\omega_0)$. Furthermore, as advocated by the work of Song-Tian \cite{SoT, SoT2, SoT3, So2}, in the case when $X$ is projective and the class $[\omega_0]$ is rational, the behavior is intimately related to the Minimal Model Program in algebraic geometry \cite{KMM}. This is in stark contrast with the general Ricci flow on compact Riemannian manifolds, where formulating such a conjectural picture seems completely hopeless in (real) dimensions larger than $3$.

In these lecture notes we will explain this conjectural picture in detail, and prove several results which go some way towards achieving this picture. After reviewing some preliminary notions and setting up notation in Section \ref{sectprel}, the first result that we consider is a cohomological characterization of the maximal existence time of the flow from \cite{Ca, Ts, Ts2, TiZ}, which we prove in Section \ref{sectmax}. Next, in Section \ref{sectfin} we discuss finite time singularities, both volume noncollapsed and volume collapsed, in particular giving a characterization of the singularity formation set, due to Collins and the author \cite{CT}. In Section \ref{sectinf} we study the case when the flow exists for all positive time, and we investigate the convergence properties at infinity, giving a detailed exposition of the collapsing results proved in \cite{SoT, SoT2, FZ, TWY, HT, TZ}. Lastly, in Section \ref{sectopen} we collect some well-known open problems on the K\"ahler-Ricci flow.

There are already two excellent set of lecture notes on the K\"ahler-Ricci flow, by Song-Weinkove \cite{SWL} and Weinkove \cite{We}. While preparing these notes, I have benefitted greatly from these references, and in fact the exposition in Section \ref{sectmax} follows \cite{SWL,We} rather closely (I decided to keep this material here because many similar arguments are used in later sections). On the other hand, in Sections \ref{sectfin} and \ref{sectinf}, which form the bulk of these notes, I have decided to focus on rather recent results which are not contained in \cite{SWL,We}.

It is not possible to cover the complete theory of the K\"ahler-Ricci flow in a short set of lecture notes, so I had to make a selection of which material to present, based on my own limited knowledge, and many important results on the K\"ahler-Ricci flow are not covered here. In particular, nothing is said about the convergence properties of the normalized K\"ahler-Ricci flow on Fano manifolds, which is a vast research area by itself (see e.g. \cite{CW,PS} and references therein). I also do not mention weak solutions of the K\"ahler-Ricci flow \cite{SoT3,EGZ3,DL}, the K\"ahler-Ricci flow for conical metrics \cite{CW2,Ed,Sh}, the K\"ahler-Ricci flow on noncompact K\"ahler manifolds \cite{Ch, Shi}, or the Chern-Ricci flow \cite{Gi0, TW,TW2,TWY0} (a generalization of the K\"ahler-Ricci flow to possibly non-K\"ahler complex manifolds).
Still, my hope is that these notes will somewhat complement \cite{SWL,We} by providing a view of some more recent developments in this field.\\

{\bf Acknowledgments. }These lecture notes are an expanded version of the mini-course ``The K\"ahler-Ricci flow'', given by the author at the $6^{\rm th}$ KAWA Winter School on March 23-26, 2015 at the Centro De Giorgi of Scuola Normale Superiore in Pisa.
The author is very grateful to M. Abate, J. Marzo, J. Raissy, P. Thomas and A. Zeriahi for the kind invitation to give a mini-course at KAWA, and to prepare these lecture notes. Many thanks also to M. Alexis, G. Edwards, Y. Li, B. Weinkove, X. Yang, Y. Zhang and to the referees for useful comments on a preliminary version. These notes were mostly written while the author was visiting the Yau Mathematical Sciences Center of Tsinghua University in Beijing, which he would like to thank for the hospitality.

\section{Preliminaries}\label{sectprel}
In these notes we assume that the reader is already familiar with the basic theory of compact K\"ahler manifolds, and we will not review all the necessary basic material. The reader can consult \cite{GH,Hu} for comprehensive introductions, or \cite{SWL,We} for a quick introduction which is tailored to the K\"ahler-Ricci flow.

\subsection{$(1,1)$ classes and the K\"ahler cone}
Let $X^n$ be a compact complex manifold, of complex dimension $n$. A closed real $(1,1)$ form $\omega$ on $X$ is called a K\"ahler metric if it is positive definite, in the sense that if we write $$\omega=\mn \sum_{i,j=1}^ng_{i\ov{j}}dz_i \wedge d\ov{z}_j,$$
in local holomorphic coordinates $\{z_i\}$ on $X$, then for each point $x\in X$ the $n\times n$ Hermitian matrix
$$g_{i\ov{j}}(x),$$
is positive definite. We will write $\omega>0$, and we will say that $X$ (or $(X,\omega)$) is a K\"ahler manifold.

In this case, $\omega$ defines a cohomology class $[\omega]$ inside
$$H^{1,1}(X,\mathbb{R})=\frac{\{\textrm{closed real }(1,1) \textrm{ forms on }X\}}{\ddbar C^\infty(X,\mathbb{R})}.$$
If $\alpha$ is a closed real $(1,1)$ form on $X$ we will write $[\alpha]$ for its class in $H^{1,1}(X,\mathbb{R})$.

Recall that when $X$ admits a K\"ahler metric then the following $\de\db$-Lemma holds (see e.g. \cite{GH}):
\begin{lemma}\label{dd}
Let $X$ be a compact K\"ahler manifold, and $\alpha$ an exact real $(1,1)$ form on $X$. Then there exists $\vp\in C^\infty(X,\mathbb{R})$, unique up to addition of a constant, such that
$$\alpha=\ddbar\vp.$$
\end{lemma}

Thanks to the $\de\db$-Lemma, we can identify $H^{1,1}(X,\mathbb{R})$ with the subspace of $H^2(X,\mathbb{R})$ of de Rham classes which have a representative which is a real $(1,1)$ form. In particular, $H^{1,1}(X,\mathbb{R})$ is a finite-dimensional real vector space.

Then we define the K\"ahler cone of $X$ to be
$$\mathcal{C}_X=\{[\alpha]\in H^{1,1}(X,\mathbb{R})\ |\ \textrm{there exists }\omega\textrm{ K\"ahler metric on }X\textrm{ with }[\omega]=[\alpha]\}.$$

This is an open, convex cone inside $H^{1,1}(X,\mathbb{R})$. Indeed $\mathcal{C}_X$ being a cone means that if we are given $[\alpha]\in\mathcal{C}_X$ and $\lambda\in\mathbb{R}_{>0}$ then $\lambda[\alpha]\in\mathcal{C}_X$, which is obvious. The convexity of $\mathcal{C}_X$ follows immediately from the fact that if $\omega_1$ and $\omega_2$ are K\"ahler metrics on $X$ and $0\leq\lambda\leq 1$, then $\lambda\omega_1+(1-\lambda)\omega_2$ is also a K\"ahler metric. To show that $\mathcal{C}_X$ is open, we fix closed real $(1,1)$ forms $\{\alpha_1,\dots,\alpha_k\}$ on $X$ such that $\{[\alpha_1],\dots,[\alpha_k]\}$ is a basis of $H^{1,1}(X,\mathbb{R})$. Given a K\"ahler class $[\alpha]\in\mathcal{C}_X$ we can write
$[\alpha]=\sum_{i=1}^k\lambda_i[\alpha_i],$ for some $\lambda_i\in\mathbb{R}$. Since $[\alpha]\in\mathcal{C}_X$, there exists a function $\vp$ such that
$$\sum_{i=1}^k\lambda_i\alpha_i+\ddbar\vp>0.$$
Since $X$ is compact, it follows that
$$\sum_{i=1}^k\ti{\lambda}_i\alpha_i+\ddbar\vp>0,$$
for all $\ti{\lambda}_i$ sufficiently close to $\lambda_i$ ($1\leq i\leq k$), and so all $(1,1)$ classes in a neighborhood of $[\alpha]$ contain a K\"ahler metric.

Furthermore we have that $\mathcal{C}_X\cap(-\mathcal{C}_X)=\emptyset$. Indeed if $\omega$ is a K\"ahler metric on $X$ and the class $-[\omega]$ is also K\"ahler, then there is a K\"ahler metric $\ti{\omega}=-\omega+\ddbar\vp$ for some function $\vp$, and so $\ddbar\vp=\omega+\ti{\omega}>0$, everywhere on $X$. This is impossible, since $\ddbar\vp\leq 0$ at a maximum point of $\vp$.

A class $[\alpha]\in\ov{\mathcal{C}_X}$ is called nef. In other words, a nef class is a limit of K\"ahler classes.

\begin{lemma}\label{nef2}
Let $(X,\omega)$ be a compact K\"ahler manifold. Then a class $[\alpha]\in H^{1,1}(X,\mathbb{R})$ is nef if and only if for every $\ve>0$ there exists $\vp_\ve\in C^\infty(X,\mathbb{R})$ such that
\begin{equation}\label{nef1}
\alpha+\ddbar\vp_\ve>-\ve\omega.
\end{equation}
\end{lemma}
\begin{proof}
Condition \eqref{nef1} is equivalent to $[\alpha+\ve\omega]\in\mathcal{C}_X$, for all $\ve>0$, which certainly implies that $[\alpha]\in \ov{\mathcal{C}_X}$. Conversely, if
$[\alpha]$ is nef then there is a sequence $\{\beta_i\}$ of closed real $(1,1)$ forms such that $\alpha+\beta_i>0$ for all $i$, and $[\beta_i]\to 0$ in $H^{1,1}(X,\mathbb{R})$.
As before we fix closed real $(1,1)$ forms $\{\alpha_1,\dots,\alpha_k\}$ on $X$ such that $\{[\alpha_1],\dots,[\alpha_k]\}$ is a basis of $H^{1,1}(X,\mathbb{R})$, and for each $i$ write
$$[\beta_i]=\sum_{j=1}^k \lambda_{ij}[\alpha_j],$$
with $\lambda_{ij}\in\mathbb{R}$. Since $[\beta_i]\to 0$, and $\{[\alpha_1],\dots,[\alpha_k]\}$ is a basis, we conclude that $\lambda_{ij}\to 0$ as $i\to\infty$, for each fixed $j$.
If we let
$$\ti{\beta}_i=\sum_{j=1}^k \lambda_{ij}\alpha_j,$$
then the forms $\ti{\beta}_i$ converge smoothly to zero, as $i\to\infty$, and we can find functions $\vp_i$ such that
$\beta_i=\ti{\beta}_i+\ddbar\vp_i$. For every $\ve>0$ we choose $i$ sufficiently large so that $\ti{\beta}_i<\ve\omega$ on $X$, and so
$$\alpha+\ve\omega+\ddbar\vp_i>\alpha+\ti{\beta}_i+\ddbar\vp_i=\alpha+\beta_i>0,$$
which proves \eqref{nef1}.
\end{proof}

\begin{corollary}\label{sum}
Let $X$ be a compact K\"ahler manifold and two real $(1,1)$ classes $[\alpha]\in \overline{\mathcal{C}_X}$ and $[\beta]\in\mathcal{C}_X$. Then $[\alpha]+[\beta]\in\mathcal{C}_X$.
\end{corollary}
\begin{proof}
We may suppose that $\beta>0$ is a K\"ahler metric, and so $\beta>2\ve\omega$ for some $\ve$ small enough. Since $[\alpha]$ is nef, Lemma \ref{nef2} gives us a function $\vp_\ve$ such that $\alpha+\ddbar\vp_\ve>-\ve\omega$, and so
$$\alpha+\beta+\ddbar\vp_\ve>\ve\omega>0.$$
\end{proof}
A nef class $[\alpha]$ is called nef and big if
$$\int_X\alpha^n>0.$$

\subsection{Ricci curvature and first Chern class}
Given a K\"ahler metric
$\omega=\mn \sum_{i,j=1}^ng_{i\ov{j}}dz_i \wedge d\ov{z}_j$ on $X$, we define the Christoffel symbols of the Chern connection of $\omega$ to be
$$\Gamma_{ij}^k=g^{k\ov{\ell}}\de_i g_{j\ov{\ell}},$$
which satisfy that  $\Gamma_{ij}^k=\Gamma_{ji}^k$ because $\omega$ is closed. Using these, we can define the covariant derivative $\nabla$ with the usual formulae (see e.g. \cite{SWL}).
The Riemann curvature tensor $\mathrm{Rm}$ of $\omega$ is the tensor with components
$$R^j_{ik\ov{\ell}}=-\de_{\ov{\ell}}\Gamma_{ki}^j,$$
and we will also consider the tensor where we lower one index
$$R_{i\ov{j}k\ov{\ell}}=g_{p\ov{j}} R^p_{ik\ov{\ell}},$$
and a direct calculation gives
$$R_{i\ov{j}k\ov{\ell}}=-\de_k\de_{\ov{\ell}}g_{i\ov{j}}+g^{p\ov{q}}\de_k g_{i\ov{q}}\de_{\ov{\ell}}g_{p\ov{j}}.$$
If $\xi,\eta\in T^{1,0}X$ are $(1,0)$ tangent vectors, we define the bisectional curvature in the direction of $\xi,\eta$ to be
$$\mathrm{Rm}(\xi,\ov{\xi},\eta,\ov{\eta})=R_{i\ov{j}k\ov{\ell}}\xi^i\ov{\xi^j}\eta^k\ov{\eta^\ell}\in\mathbb{R}.$$
The Ricci curvature tensor is defined to be
$$R_{i\ov{j}}=R_{k\ov{\ell}i\ov{j}}g^{k\ov{\ell}},$$
and another direct calculation gives the crucial formula
\begin{equation}\label{ric}
R_{i\ov{j}}=-\de_i\de_{\ov{j}}\log\det(g_{p\ov{q}}).
\end{equation}
The scalar curvature $R$ is then defined to be
$$R=g^{i\ov{j}}R_{i\ov{j}}.$$
We define the Ricci form of $\omega$ to be
$$\Ric(\omega)=\mn \sum_{i,j=1}^nR_{i\ov{j}}dz_i \wedge d\ov{z}_j,$$
which by \eqref{ric} is locally equal to $-\ddbar\log\det(g_{p\ov{q}}).$ Therefore $\Ric(\omega)$ is a closed real $(1,1)$ form. If $\ti{\omega}$ is another K\"ahler metric then
$$\Ric(\omega)-\Ric(\ti{\omega})=\ddbar\log\frac{\det\ti{g}}{\det g},$$
where $\log\frac{\det\ti{g}}{\det g}$ is the globally defined smooth function which equals
$$\log\frac{\det(\ti{g}_{p\ov{q}})}{\det(g_{p\ov{q}})}$$
in any local holomorphic coordinate chart.
If we use the K\"ahler volume element $\omega^n$, then we also have that
$$\log\frac{\det\ti{g}}{\det g}=\log\frac{\ti{\omega}^n}{\omega^n}.$$
Therefore the cohomology class
$$[\Ric(\omega)]\in H^{1,1}(X,\mathbb{R}),$$
is independent of $\omega$, and we set
$$c_1(X)=\frac{1}{2\pi}[\Ric(\omega)],$$
the first Chern class of $X$.
Also, if we denote by
$$K_X=\Lambda^n (T^{1,0}X)^*,$$
the canonical bundle of $X$, then the first Chern class of $K_X$ satisfies $c_1(K_X)=-c_1(X)$.

If $\Omega$ is a smooth positive volume form on $X$, then in local holomorphic coordinates we can write
$$\Omega=f (\mn)^ndz_1\wedge d\ov{z}_1\wedge\dots\wedge dz_n\wedge d\ov{z}_n,$$
where $f$ is a smooth positive locally defined function, and we let
$$\Ric(\Omega)=-\ddbar\log f.$$
It is easy to see that $\Ric(\Omega)$ gives a well-defined global closed real $(1,1)$ form, and that its cohomology class in $H^{1,1}(X,\mathbb{R})$ does not depend on the choice of $\Omega$. Taking $\Omega=\omega^n$ for some K\"ahler metric $\omega$, we immediately see that
$$\Ric(\omega^n)=\Ric(\omega),$$
and so $[\Ric(\Omega)]=2\pi c_1(X)$ for any smooth positive volume form $\Omega$. Sometimes we may also write $\Ric(\Omega)=-\ddbar\log\Omega$.

\subsection{Some more notation}
If $\alpha$ is a real $(1,1)$ form on $X$, and $\omega$ a K\"ahler metric, we will write
$$\tr{\omega}{\alpha}=g^{i\ov{j}}\alpha_{i\ov{j}},$$
and it is easy to see that
$$n\alpha\wedge\omega^{n-1}=(\tr{\omega}{\alpha})\omega^n.$$
In particular, if $f\in C^\infty(X,\mathbb{R})$,
$$\tr{\omega}{(\ddbar f)}=g^{i\ov{j}}\de_i\de_{\ov{j}}f=\Delta f,$$
where $\Delta$ is the complex Laplacian of the metric $\omega$ (if we want to emphasize the metric, we will also write $\Delta_{\omega}$). At a maximum point of $f$, we have that $\ddbar f\leq 0$, and so also $\Delta f\leq 0$.

We also have
$$\tr{\omega}{(\mn \de f\wedge\db f)}=g^{i\ov{j}}\de_i f\de_{\ov{j}}f=|\de f|^2_g,$$
where $g$ denotes the Hermitian metric defined by the K\"ahler metric $\omega$.

Next, we define the $C^k$ norms on smooth functions $(k\geq 0)$, with respect to $\omega$, by
$$\|f\|_{C^k(X,g)}=\sum_{p+q\leq k, 0\leq p\leq q}\sup_X|\nabla{}^p\ov{\nabla}{}^q f|_g,$$
where
$$|\nabla{}^p\ov{\nabla}{}^q f|^2_g=g^{i_1\ov{k_1}}\cdots g^{\ell_q \ov{j_q}}  \nabla_{i_1}\cdots\nabla_{i_p}\nabla_{\ov{j_1}}\dots \nabla_{\ov{j_q}}f\nabla_{\ov{k_1}}\cdots\nabla_{\ov{k_p}}\nabla_{\ell_1}\dots \nabla_{\ell_q}f.$$
We only sum on $p\leq q$ to avoid repetition of terms (since $|\nabla{}^q\ov{\nabla}{}^p f|_g=|\nabla{}^p\ov{\nabla}{}^q f|_g$ because $f$ is real-valued).
We will also abbreviate
\begin{equation}\label{abbrev}
\sum_{p+q\leq k, 0\leq p\leq q}|\nabla{}^p\ov{\nabla}{}^q f|_g=\sum_{0\leq j\leq k}|\nabla^j_{\mathbb{R}}f|_g.
\end{equation}
Similarly we can define the $C^k$ norms on tensors (if the tensor is not real, we sum over all $p,q\geq 0, p+q\leq k$).

We will also briefly use H\"older space $C^{k,\alpha}(X,g)$, where $k\in\mathbb{N}$ and $0<\alpha<1$. This is composed of functions $f:X\to\mathbb{R}$
such that the norm
\[\begin{split}
\|f\|_{C^{k,\alpha}(X,g)}&=\sum_{i\leq k}\|\nabla^i_{\mathbb{R}} f\|_{C^0(X,g)}+\sup_{x\neq y \in X}\frac{|\nabla^k_{\mathbb{R}} f(x)-\nabla^k_{\mathbb{R}} f(y)|_{g}}{d(x,y)^\alpha}
\end{split}\]
is finite (we assume of course that $f$ is sufficiently differentiable so that these derivatives make sense),
where $\nabla_{\mathbb{R}}$ is the real covariant derivative of $g$, $d(x,y)$ is the $g$-distance between $x,y\in X$, and in the expression $|\nabla^k_{\mathbb{R}} f(x)-\nabla^k_{\mathbb{R}} f(y)|_{g}$ we
are using parallel transport with respect to $g$ to compare the values of these two tensors, which are at different points in $X$.

\subsection{Analytic subvarieties}
We now quickly cover the basics about analytic subvarieties of a compact complex manifold, see \cite[p.12-14]{GH} for more details.
A closed subset $V\subset X$ is called an analytic subvariety of $X$ if for every point $x\in V$ we can find an open neighborhood $x\in U\subset X$ and holomorphic functions
$\{f_1,\dots,f_N\}$ on $U$ such that
$$V\cap U=\{y\in U\ |\ f_1(y)=\dots=f_N(y)=0\}.$$
A point $x\in V$ is called regular, or smooth, if near $x$ the subvariety $V$ is a complex submanifold of $X$. A point which is not regular is called singular.
The set of regular points is denoted by $V_{reg}$ and its complement by
$V_{sing}=V\backslash V_{reg}$. The singular locus $V_{sing}$ is itself an analytic subvariety of $X$, and it is properly contained in $V$.
A subvariety $V$ is called irreducible if we cannot write $V=V_1\cup V_2$ where $V_1,V_2$ are analytic subvarieties which are not equal to $V$. In this case, $V_{reg}$ is connected, and so it is is a complex submanifold of $X$ of a well-defined dimension, which we call $\dim V$.

If $V$ is not irreducible, then we can write $V=V_1\cup\dots \cup V_N$ where the $V_j$ are irreducible analytic subvarieties of $X$, called the irreducible components of $V$. In this case, we set $\dim V$ to be the maximum of $\dim V_j$. With these definitions, we have that $\dim V=0$ if and only if $V$ is a finite set of point.

A fundamental result of Lelong (see \cite[p.32]{GH}) shows that if $V$ is an irreducible analytic subvariety of $X$ of dimension $k>0$, and $\alpha$ is a smooth real $(k,k)$ form on $X$, then the integral
$$\int_V\alpha:=\int_{V_{reg}}\alpha,$$
is finite. Furthermore, for any smooth real $(k-1,k-1)$ form $\beta$ on $X$ we have
$$\int_V\ddbar \beta=0,$$
see \cite[p.33]{GH}.
Therefore if $[\alpha]$ is a real $(1,1)$ class on $X$, we may unambiguously write
$$\int_V\alpha^k.$$
Furthermore, if $[\alpha]\in \mathcal{C}_X$, and we fix a K\"ahler metric $\omega\in[\alpha]$, then
$$\int_V\alpha^k=\int_V\omega^k=k!\vol(V,\omega)>0,$$
see \cite[p.31]{GH}, where $\vol(V,\omega)$ denotes the real $2k$-dimensional volume of $V_{reg}$ with respect to $\omega$ (which is finite). Passing to the limit, we obtain that if $[\alpha]\in \ov{\mathcal{C}_X}$, and $V\subset X$ is any irreducible positive-dimensional analytic subvariety, then
$$\int_V\alpha^{\dim V}\geq 0.$$
For a nef $(1,1)$ class $[\alpha]\in \ov{\mathcal{C}_X}$ we then define its null locus to be
\begin{equation}\label{null}
\Null(\alpha)=\bigcup_{\int_V\alpha^{\dim V}=0} V,
\end{equation}
where the union is over all irreducible positive-dimensional analytic subvarieties $V\subset X$ with $\int_V\alpha^{\dim V}=0$.
The set $\Null(\alpha)$ is in fact an analytic subvariety of $X$ (in general not irreducible), as follows for example from \cite[Theorem 1.1]{CT}. We have that
$\Null(\alpha)=X$ if and only if $\int_X\alpha^n=0$, and otherwise $\Null(\alpha)$ is a proper analytic subvariety of $X$.

\subsection{Kodaira dimension}
Let $X$ be a compact complex manifold. We consider the space of global pluricanonical forms, namely
$$H^0(X,K_X^{\otimes\ell}),$$
where $\ell\geq 1$. If $H^0(X,K_X^{\otimes\ell})=0$ for all $\ell\geq 1$, then we say that the Kodaira dimension of $X$ is $-\infty$, and we write $\kappa(X)=-\infty$. If this is not the case, then we let
$$\kappa(X)=\limsup_{\ell\to\infty} \frac{\log \dim H^0(X,K_X^{\otimes\ell})}{\log\ell}.$$
It can be proved that either $\kappa(X)=-\infty$ or otherwise $0\leq\kappa(X)\leq n$, and in fact we have
$$C^{-1}\ell^{\kappa(X)}\leq \dim H^0(X,K_X^{\otimes\ell})\leq C\ell^{\kappa(X)},$$
for some constant $C>0$ and all $\ell$ such that $H^0(X,K_X^{\otimes\ell})\neq 0$ (see \cite[Corollary 2.1.38]{Laz}).
Furthermore, we have that $\kappa(X)=0$ if and only if $\dim H^0(X,K_X^{\otimes\ell})\leq 1$ for all $\ell\geq 1$, and it equals $1$ for at least one value of $\ell$.

Two compact complex manifolds $X,Y$ are called bimeromorphic if we can find proper analytic subvarieties $V_1\subset X, V_2\subset Y$ and a biholomorphism $\Phi:X\backslash V_1\to Y\backslash V_2$.
If two compact complex manifolds are bimeromorphic, then they have the same Kodaira dimension.

A compact K\"ahler manifold is called uniruled if for every point $x\in X$ there exists a rational curve $x\in C\subset X$, i.e. a non-constant holomorphic map $f:\mathbb{CP}^1\to X$ with image $C$ containing $x$. Uniruled manifolds have $\kappa(X)=-\infty$, and the converse is also conjectured to be true.

\subsection{Gromov-Hausdorff convergence}
Let $(X,d_X), (Y,d_Y)$ be compact metric spaces. Given $\ve>0$ we say that their Gromov-Hausdorff distance is less than or equal to $\ve$ if there are two maps
$F:X\to Y$ and $G:Y\to X$ (not necessarily continuous) such that
\begin{equation}\label{un1}
|d_X(x_1,x_2)-d_Y(F(x_1),F(x_2))|\leq \ve,
\end{equation}
for all $x_1,x_2\in X$,
\begin{equation}\label{un2}
|d_Y(y_1,y_2)-d_X(G(y_1),G(y_2))|\leq \ve,
\end{equation}
for all $y_1,y_2\in Y$,
\begin{equation}\label{un3}
d_X(x,G(F(x)))\leq \ve,
\end{equation}
for all $x\in X$, and
\begin{equation}\label{un4}
d_Y(y,F(G(y)))\leq \ve,
\end{equation}
for all $y\in Y$.
We then say that a family $(X_t,d_t), t\in[0,\infty),$ of compact metric spaces converge to a compact metric space $(Y,d_Y)$ in the Gromov-Hausdorff topology if for all $\ve>0$ there is $T\geq 0$ such that the Gromov-Hausdorff distance between $(X_t,d_t)$ and $(Y,d_Y)$ is at most $\ve$ for all $t\geq T$. We refer the reader to \cite{Ro} for more about this notion.

\section{Maximal existence time}\label{sectmax}

\subsection{The maximal existence time of the K\"ahler-Ricci flow}
Let $\omega(t)$ be a solution of the K\"ahler-Ricci flow \eqref{krf} on a compact K\"ahler manifold $X$, with $t\in [0,T), 0<T\leq \infty$. Taking the cohomology class of \eqref{krf} we see that
$$\frac{\de}{\de t}[\omega(t)]=-[\Ric(\omega(t))]=-2\pi c_1(X),$$
where the right-hand side is independent of $t$. It follows that
$$[\omega(t)]=[\omega_0]-2\pi tc_1(X),$$
and so
$$[\omega_0]-2\pi tc_1(X)\in\mathcal{C}_X,$$
for $t\in [0,T)$.
The converse is the content of the following theorem proved in \cite{Ca, Ts, Ts2, TiZ}.
\begin{theorem}\label{tz}
Let $(X^n,\omega_0)$ be a compact K\"ahler manifold. Then the K\"ahler-Ricci flow \eqref{krf} has a unique smooth solution $\omega(t)$ defined on the maximal time interval $[0,T), 0<T\leq\infty$, where $T$ is given by
\begin{equation}\label{coho}
T=\sup\{ t>0\ |\ [\omega_0]-2\pi tc_1(X)\in\mathcal{C}_X\}.
\end{equation}
\end{theorem}

Here and in the rest of these notes, when we say that $[0,T)$ is maximal we really mean forward maximal. It may be possible that the flow \eqref{krf} has a solution also for some negative time, but this is in general not the case, and we will not discuss backwards solvability in these notes.

This theorem has the following useful corollary.

\begin{corollary}\label{nefff}
Under the same assumptions as in Theorem \ref{tz}, we have that $T=\infty$ if and only if $-c_1(X)\in\ov{\mathcal{C}_X}$.
\end{corollary}
Note that the condition $-c_1(X)\in\ov{\mathcal{C}_X}$ is independent of the initial metric $\omega_0$. It is equivalent to the fact that $K_X$ is nef, and this is also sometimes stated by saying that $X$ is a smooth minimal model.
\begin{proof}
If $-c_1(X)\in\ov{\mathcal{C}_X}$ then $-2\pi tc_1(X)\in\ov{\mathcal{C}_X}$ too, for all $t\geq 0$. Since $[\omega_0]\in\mathcal{C}_X$, we conclude from Corollary \ref{sum} that
$[\omega_0]-2\pi tc_1(X)\in\mathcal{C}_X$, and so $T=\infty$ thanks to Theorem \ref{tz}.

If conversely $T=\infty,$ then for all $t>0$ we have
$$\frac{1}{2\pi t}[\omega_0]-c_1(X)=\frac{1}{2\pi t}[\omega(t)]\in\mathcal{C}_X,$$
and letting $t\to\infty$ we immediately obtain that $-c_1(X)\in\ov{\mathcal{C}_X}$.
\end{proof}

\subsection{Reduction to a parabolic complex Monge-Amp\`ere equation}
We now start the proof of Theorem \ref{tz}. We set
$T=\sup\{ t>0\ |\ [\omega_0]-2\pi tc_1(X)\in\mathcal{C}_X\}.$
As we discussed earlier, it is clear that no solution of \eqref{krf} can exist for $t\geq T$, and so it is enough to show that \eqref{krf} has a unique smooth solution defined on $[0,T)$.
Fix any $0<T'<T$ (so in particular $T'<\infty$). By definition we have that $[\omega_0]-2\pi T'c_1(X)\in\mathcal{C}_X$, so we can choose a K\"ahler metric $\eta$ in this class.
We define
\begin{equation}\label{1}
\chi=\frac{1}{T'}(\eta-\omega_0),
\end{equation}
so $\chi$ is a closed real $(1,1)$ form cohomologous to $-2\pi c_1(X)$, and
\begin{equation}\label{2}
\hat{\omega}_t=\omega_0+t\chi=\frac{1}{T'}((T'-t)\omega_0+t\eta),
\end{equation}
which is a K\"ahler metric for all $t\in [0,T']$. Fix $\Omega'$ any smooth positive volume form on $X$. Then $\Ric(\Omega')$ is a closed real $(1,1)$ form cohomologous to $2\pi c_1(X)$, and so there is a smooth function $F$ such that $\chi=-\Ric(\Omega')+\ddbar F$. We then define
$$\Omega=e^F\Omega',$$
which is a smooth positive volume form with
\begin{equation}\label{3}
\Ric(\Omega)=-\chi.
\end{equation}
\begin{lemma}\label{krflow}
A smooth family $\omega(t)$ of K\"ahler metrics on $[0,T')$ solves the K\"ahler-Ricci flow \eqref{krf} if and only if there is a smooth family of smooth functions $\vp(t), t\in [0,T')$ such that $\omega(t)=\hat{\omega}_t+\ddbar\vp(t)$ and we have
\begin{equation}\label{ma}
\left\{
                \begin{aligned}
                  &\frac{\de}{\de t}\vp(t)=\log\frac{(\hat{\omega}_t+\ddbar\vp(t))^n}{\Omega}\\
                  &\vp(0)=0\\
                  &\hat{\omega}_t+\ddbar\vp(t)>0.
                \end{aligned}
              \right.
\end{equation}
\end{lemma}
Equation \eqref{ma} is called a parabolic complex Monge-Amp\`ere equation.
\begin{proof}
For the ``if'' direction, we set $\omega(t)=\hat{\omega}_t+\ddbar\vp(t)$ and compute
$$\frac{\de}{\de t}\omega(t)=\chi+\ddbar\log \frac{\omega(t)^n}{\Omega}=\chi+\Ric(\Omega)-\Ric(\omega(t))=-\Ric(\omega(t)),$$
and since clearly $\omega(0)=\hat{\omega}_0=\omega_0$, we conclude that $\omega(t)$ solves \eqref{krf}.

For the ``only if'' direction, given a solution $\omega(t)$ of \eqref{krf} on $[0,T')$, we define
$$\vp(t)=\int_0^t\log \frac{\omega(s)^n}{\Omega}ds,$$
for $t\in[0,T')$. We clearly have that
$$\frac{\de}{\de t}\vp(t)=\log\frac{\omega(t)^n}{\Omega}, \quad \vp(0)=0.$$
We compute
$$\frac{\de}{\de t}\left(\omega(t)-\hat{\omega}_t-\ddbar\vp(t)\right)=-\Ric(\omega(t))-\chi+\Ric(\omega(t))-\Ric(\Omega)=0,$$
and so $\omega(t)-\hat{\omega}_t-\ddbar\vp(t)$ is a smooth family of real $(1,1)$ forms which satisfy
$$\frac{\de}{\de t}\left(\omega(t)-\hat{\omega}_t-\ddbar\vp(t)\right)=0,\quad \left(\omega(t)-\hat{\omega}_t-\ddbar\vp(t)\right)|_{t=0}=0,$$
and so we must have $\omega(t)-\hat{\omega}_t-\ddbar\vp(t)\equiv 0$ on $X\times [0,T')$, and so \eqref{ma} holds.
\end{proof}

We can now prove the uniqueness in Theorem \ref{tz}.
\begin{theorem}\label{uniq}
Suppose $\omega_1(t)$ and $\omega_2(t)$ are two solutions of \eqref{krf} on the same time interval $[0,T')$. Then $\omega_1(t)=\omega_2(t)$ for all $t\in [0,T')$.
\end{theorem}
\begin{proof}
Thanks to Lemma \ref{krflow} we can write
$$\omega_1(t)=\hat{\omega}_t+\ddbar\vp_1(t), \quad \omega_2(t)=\hat{\omega}_t+\ddbar\vp_2(t),$$
where $\vp_1(t),\vp_2(t)$ both solve \eqref{ma} for $t\in[0,T')$. Our goal is to show that $\vp_1(t)=\vp_2(t)$ for all $t\in [0,T')$.

If we write $\psi(t)=\vp_2(t)-\vp_1(t)$ then we have
$$(\omega_1(t)+\ddbar\psi(t))^n=\omega_2(t)^n=e^{\dot{\vp}_2(t)}\Omega=e^{\dot{\psi}(t)}\omega_1(t)^n,$$
using \eqref{ma}.
Here and in the following we write
$$\dot{\psi}(t)=\frac{\de}{\de t}\psi(t).$$
 In other words, the function $\psi(t)$ satisfies
\[\left\{
                \begin{aligned}
                  &\frac{\de}{\de t}\psi(t)=\log\frac{(\omega_1(t)+\ddbar\psi(t))^n}{\omega_1(t)^n}\\
                  &\psi(0)=0\\
                  &\omega_1(t)+\ddbar\psi(t)>0.
                \end{aligned}
              \right.
\]
Then, for every $\ve>0$, the function $\ti{\psi}(t)=\psi(t)-\ve t$ satisfies
$$\frac{\de}{\de t}\ti{\psi}(t)=\log\frac{(\omega_1(t)+\ddbar\ti{\psi}(t))^n}{\omega_1(t)^n}-\ve,$$
and we can now apply the maximum principle. Fix any $0<T''<T'$, and let the maximum of $\ti{\psi}(t)$ on $X\times [0,T'']$ be achieved at
$(x,t)$. If $t>0$ then at $(x,t)$ we have
$$0<\omega_1(t)+\ddbar\ti{\psi}(t)\leq \omega_1(t),$$
and so
$$ (\omega_1(t)+\ddbar\ti{\psi}(t))^n\leq \omega_1(t)^n,$$
and
$$0\leq \frac{\de}{\de t}\ti{\psi}(t)=\log\frac{(\omega_1(t)+\ddbar\ti{\psi}(t))^n}{\omega_1(t)^n}-\ve\leq -\ve,$$
a contradiction. Therefore we must have $t=0$, and so $\ti{\psi}(x,t)=\psi(x,0)=0$. Since $(x,t)$ was a maximum point, we conclude that
$$\ti{\psi}(t)\leq 0$$
on $X\times[0,T'']$,
or in other words
$$\psi(t)\leq \ve t$$
on $X\times[0,T'']$, and since $T''<T'$ was arbitrary, the same holds on $[0,T')$. Letting $\ve\to 0$ we conclude that
$$\psi(t)\leq 0,$$
on $X\times[0,T').$ Applying the same argument to $\psi(t)+\ve t$, and looking at its minimum point, we conclude that $\psi(t)$ is identically zero.
\end{proof}

\subsection{Existence for a short positive time}
We are now ready to prove a short-time existence theorem, originally due to Hamilton \cite{Ha} for the Ricci flow on general compact Riemannian manifolds. The K\"ahler setting allows for a much simpler proof.

\begin{theorem}\label{short}
Let $(X^n,\omega_0)$ be a compact K\"ahler manifold. Then there exists $\ve>0$ and a unique smooth solution $\omega(t)$ of the K\"ahler-Ricci flow \eqref{krf} defined on $[0,\ve)$.
\end{theorem}
\begin{proof}
Let $T>0$ be defined as in \eqref{coho}, fix any $0<T'<T$, fix a K\"ahler metric $\eta$ in $[\omega_0]-2\pi T'c_1(X)$, and define $\chi, \hat{\omega}_t$ and $\Omega$ as in \eqref{1}, \eqref{2} and \eqref{3}. Since we have already proved uniqueness in Theorem \ref{uniq}, our goal is to produce a solution $\vp(t)$ of \eqref{ma} defined on $[0,\ve)$ for some $\ve>0$ (thanks to Lemma \ref{krflow}). Up to rescaling the time parameter, we may assume that $T'\geq 1$.

Fix an integer $k\geq 2$ and a real number $0<\alpha<1$, and let
$U_t\subset C^{k,\alpha}(X,g_0)$ be the open set given by all functions $\psi\in C^{k,\alpha}(X,g_0)$ such that $\hat{\omega}_t+\ddbar\psi>0$ everywhere on $X$.
This is an open set which contains the origin, and for every $t\in [0,T']$ we can define an operator $E_t:U_t\to C^{k-2,\alpha}(X,g_0)$ by
$$E_t(\psi)=\log\frac{(\hat{\omega}_t+\ddbar\psi)^n}{\Omega}.$$
To take care of the dependence on $t$ (which we will just restrict to $[0,1]$) we consider the parabolic H\"older space $C^{k,\alpha}(X\times[0,1],g_0)$ of functions $u:X\times[0,1]\to\mathbb{R}$
such that the norm
\[\begin{split}
\|u\|_{C^{k,\alpha}(X\times[0,1],g_0)}&=\sum_{i+2j\leq k}\|\nabla^i_{\mathbb{R}} \de_t^j u\|_{C^0(X\times[0,1],g_0)}\\
&+\sum_{i+2j=k}\sup_{x\neq y \in X, t\neq s\in [0,1]}\frac{|\nabla^i_{\mathbb{R}} \de_t^j u(x,t)-\nabla^i_{\mathbb{R}} \de_t^j u(y,s)|_{g_0}}{(d(x,y)^2+|t-s|)^{\frac{\alpha}{2}}}
\end{split}\]
is finite (we assume of course that $u$ is sufficiently differentiable in $X$ and $t$ so that these derivatives make sense),
where $\nabla_{\mathbb{R}}$ is the real covariant derivative of $g_0$ (see \eqref{abbrev}), $d(x,y)$ is the $g_0$-distance between $x,y\in X$, and in the expression $|\nabla^i_{\mathbb{R}} \de_t^j u(x,t)-\nabla^i_{\mathbb{R}} \de_t^j u(y,s)|_{g_0}$ we
are using parallel transport with respect to $g_0$ to compare the values of these two tensors, which are at different points in $X$ (see e.g. \cite{Kr, Li} for more on these spaces).

These are Banach spaces, and we let $U\subset  C^{k,\alpha}(X\times[0,1],g_0)$ be the subset
of all functions $\psi\in C^{k,\alpha}(X\times[0,1],g_0)$ such that $\hat{\omega}_t+\ddbar\psi(t)>0$ on $X\times[0,1]$, which is again an open set containing the origin.
We then define an operator $E:U\to C^{k-2,\alpha}(X\times[0,1],g_0)$ by
$$E(\psi)(t)=\log\frac{(\hat{\omega}_t+\ddbar\psi(t))^n}{\Omega}.$$
If we can find $\ve>0$ and a function $\vp\in U\subset C^{k,\alpha}(X\times[0,1],g_0)$ such that
\begin{equation}\label{want}
\left\{
                \begin{aligned}
                  &\frac{\de}{\de t}\vp(t)=E(\vp)(t)\\
                  &\vp(0)=0,
                \end{aligned}
              \right.
\end{equation}
on $X\times[0,\ve)$ then standard parabolic PDE theory (differentiating \eqref{want} and applying e.g. \cite[Chapter 8]{Kr}) implies that $\vp$ is smooth on $X\times[0,\ve)$, and so is our desired solution of \eqref{ma}.

To achieve this, we first note that if we have such a solution $\vp(t)$ (suppose that it is smooth) then its time derivatives
$$\frac{\de^\ell}{\de t^\ell}\vp(0),$$
for all $\ell\geq 0$ are equal to certain smooth functions $F_\ell$ which are expressible purely in terms of the given data $\omega_0,\chi,\Omega.$ For example
$$F_0=0,\quad F_1=\log\frac{\omega_0^n}{\Omega}, \quad F_2=-\tr{\omega_0}{\Ric(\omega_0)}=-R(\omega_0),$$
and so on. The case of general $\ell$ follows easily by differentiating the flow equation \eqref{ma}, noting that all time derivatives of $\hat{\omega}_t$ and $\ddbar\vp(t)$ are so expressible. We choose a function $\hat{\vp}\in C^{k+1}(X\times[0,1],g_0)$ (so in particular in $C^{k,\alpha}$) such that
$$\frac{\de^\ell}{\de t^\ell}\hat{\vp}(0)=F_\ell,$$
for all $0\leq \ell\leq \left \lfloor{\frac{k}{2}}\right \rfloor+1$, and such that $\hat{\vp}$ lies inside $U$. In other words,
 the Taylor series of $\hat{\vp}$ in $t$ at $t=0$ matches the one of a solution $\vp$ (if it exists) up to order $\left \lfloor{\frac{k}{2}}\right \rfloor+1$.
Let $\hat{h}=\frac{\de}{\de t}\hat{\vp}-E(\hat{\vp})$, for $t\in [0,1]$, so that $\hat{h}$ is by construction a function in $C^{k-2,\alpha}(X\times[0,1],g_0)$, whose Taylor series in $t$ at $t=0$ vanishes up to order $\left \lfloor{\frac{k}{2}}\right \rfloor$. For a given $\ve>0$ let
$h_\ve(t)$ be equal to $0$ for $0\leq t\leq \ve$ and equal to $\hat{h}(t-\ve)$ for $\ve\leq t\leq 1$.
Then by construction we have that $h_\ve\in C^{k-2,\alpha}(X\times[0,1],g_0)$ and
\begin{equation}\label{norm}
\|h_\ve-\hat{h}\|_{C^{k-2,\alpha}(X\times[0,1],g_0)}\to 0, \quad {\mathrm as\ }\ve\to 0,
\end{equation}
because $\hat{h}\in C^{k-2,\alpha}(X\times[0,1],g_0)$.
We then wish to perturb $\hat{\vp}$ to another function $\vp\in U\subset C^{k,\alpha}(X\times[0,1],g_0)$ which solves
\begin{equation}\label{have}
\left\{
                \begin{aligned}
                  &\frac{\de}{\de t}\vp(t)=E(\vp)(t)+h_\ve(t)\\
                  &\vp(0)=0,
                \end{aligned}
              \right.
\end{equation}
on $X\times[0,1]$, for some small $\ve>0$,
because if we can do this then $\vp$ solves \eqref{want} on $X\times[0,\ve)$ since $h_\ve(t)=0$ for $0\leq t\leq \ve$. This is a standard application of the Inverse Function Theorem in Banach spaces together with the theory of linear parabolic PDEs. Indeed consider the operator
$$\mathcal{E}:U\to C^{k-2,\alpha}(X\times[0,1],g_0)\times C^{k,\alpha}(X,g_0),$$
$$\mathcal{E}(\psi)=\left(\frac{\de}{\de t}\psi - E(\psi), \psi(0)\right).$$
Then $\mathcal{E}$ defines a Fr\'echet differentiable map between Banach spaces, and its Gateaux derivative at $\psi\in U$ in the direction $\eta\in C^{k,\alpha}(X\times[0,1],g_0)=T_{\psi}U$ is given by
\begin{equation}\label{linear}
D_{\psi}\mathcal{E}(\eta)=\left(\frac{\de}{\de t}\eta-D_\psi E(\eta),\eta(0)\right),
\end{equation}
where $D_\psi E(\eta)$ is given by
$$D_\psi E(\eta)=\frac{\de}{\de s}\bigg|_{s=0} E(\psi+s\eta)=\tr{\hat{\omega}_t+\ddbar\psi(t)}{(\ddbar\eta(t))}=\Delta_{\hat{\omega}_t+\ddbar\psi(t)}\eta(t),$$
for all $t\in [0,1]$.
Given any point $(h,\eta_0)\in C^{k-2,\alpha}(X\times[0,1],g_0)\times C^{k,\alpha}(X,g_0)$, the condition that
$D_\psi \mathcal{E}(\eta)=(h,\eta_0)$ is equivalent to the linear parabolic PDE
\begin{equation}\label{linear2}
\left\{
                \begin{aligned}
                  &\frac{\de}{\de t}\eta(t)=\Delta_{\hat{\omega}_t+\ddbar\psi(t)}\eta(t)+h(t)\\
                  &\eta(0)=\eta_0,
                \end{aligned}
              \right.
\end{equation}
for $t\in [0,1]$.
It follows that the map
$$D_{\psi}\mathcal{E}:C^{k,\alpha}(X\times[0,1],g_0)\to C^{k-2,\alpha}(X\times[0,1],g_0)\times C^{k,\alpha}(X,g_0),$$
is an isomorphism of Banach spaces thanks to the existence, uniqueness and continuous dependence on the initial data for the linear parabolic PDE \eqref{linear2} (see e.g. \cite[Chapter 8]{Kr}).
The Inverse Function Theorem in Banach spaces then implies that $\mathcal{E}$ is a local isomorphism, near any point in $U$. Since our function $\hat{\vp}$ solves
\begin{equation}\label{have2}
\left\{
                \begin{aligned}
                  &\frac{\de}{\de t}\hat{\vp}(t)=E(\hat{\vp})(t)+\hat{h}(t)\\
                  &\hat{\vp}(0)=0,
                \end{aligned}
              \right.
\end{equation}
on $X\times[0,1]$, and recalling \eqref{norm}, we see that there exists $\ve>0$ small enough and $\vp\in U$ solving \eqref{want}, as desired.
\end{proof}

\subsection{A priori estimates and completion of proof of Theorem \ref{tz}}
Thanks to Theorem \ref{short} we now have a solution $\omega(t)$ of \eqref{krf} for some short time $[0,\ve), \ve>0$. We may take then the largest possible $\ve$, and call it $T_{{\rm max}}$, which satisfies $0<T_{{\rm max}}\leq\infty$, and depends only on $\omega_0$. Recall that to prove Theorem \ref{tz} we have to show that in fact we have a solution on $[0,T)$ where $T$ is given by \eqref{coho}, and that earlier we have fixed $0<T'<T$. If we have that $T_{{\rm max}}\geq T'$ then we are done, since $T'<T$ is arbitrary, so the goal is to show that if $T_{{\rm max}}<T'$ (in particular, $T_{{\rm max}}<\infty$) then we obtain a contradiction.

The key to deriving the contradiction are the following {\rm a priori} estimates.
\begin{theorem}\label{estimates}
For every $k\geq 0$ there is a constant $C_k$, which depends only on $k,\omega_0$, such that
\begin{equation}\label{est1}
\|\vp(t)\|_{C^k(X,g_0)}\leq C_k,
\end{equation}
\begin{equation}\label{est2}
\omega(t)\geq C_0^{-1}\omega_0,
\end{equation}
for all $t\in [0,T_{{\rm max}})$.
\end{theorem}

Indeed, assuming Theorem \ref{estimates} we can now complete the proof of Theorem \ref{tz}.

\begin{proof}[Proof of Theorem \ref{tz}]
Observe that the flow equation \eqref{ma} together with \eqref{est1}, \eqref{est2} implies that
\begin{equation}\label{est3}
\left\|\frac{\de^\ell}{\de t^\ell}\vp(t)\right\|_{C^k(X,g_0)}\leq C_{k,\ell},
\end{equation}
for all $k,\ell\geq 0$ and for some uniform constants $C_{k,\ell}$.

The Ascoli-Arzel\`a Theorem implies that for every $k\geq 0$ the embedding $C^{k+1}(X,g_0)\hookrightarrow C^k(X,g_0)$ is compact. Therefore the bounds \eqref{est1}, together with a diagonal argument, show that given any sequence $t_j\to T_{{\rm max}}$ there exists a subsequence $t_{j_k}$ and a smooth function $\vp_{T_{{\rm max}}}$ such that $\vp(t_{j_k})$ converges to $\vp_{T_{{\rm max}}}$ in $C^\ell(X,g_0)$ for all $\ell\geq 0$ (at this point the function $\vp_{T_{{\rm max}}}$ may depend on the chosen sequence). Now \eqref{est3} in particular implies that $\sup_X |\dot{\vp}(t)|\leq C$ for all $t\in [0,T_{{\rm max}})$, for some constant $C$ which depends only on the initial data, and so
\begin{equation}\label{time}
\frac{\de}{\de t}(\vp(t)-Ct)\leq 0,
\end{equation}
on $X\times [0,T_{{\rm max}})$. The functions $\vp(t)-Ct$ are therefore nonincreasing in $t$ and uniformly bounded below (by \eqref{est1} and the fact that $T_{{\rm max}}<\infty$), and so they have a unique pointwise limit as $t\to T_{{\rm max}}$, which is necessarily equal to $\vp_{T_{{\rm max}}}$ since this is the $C^\ell$ (in particular uniform) limit of the sequence $\vp(t_{j_k})$. Therefore the limit $\vp_{T_{{\rm max}}}$ is unique, and an elementary argument implies that $\vp(t)\to \vp_{T_{{\rm max}}}$ as $t\to T_{{\rm max}}$ in $C^\ell(X,g_0)$ for all $\ell\geq 0$.
Indeed, if this was not the case then we could find a sequence $t_j\to T_{{\rm max}}$ and an $\ell\geq 0$ such that the functions $\vp(t_j)$ do not converge to $\vp_{T_{{\rm max}}}$ in $C^\ell(X,g_0)$, but we have shown that we can then extract a subsequence $t_{j_k}$ so that $\vp(t_{j_k})$ does converge to $\vp_{T_{{\rm max}}}$ in $C^\ell(X,g_0)$, a contradiction.

Therefore the metrics $\omega(t)=\hat{\omega}_t+\ddbar\vp(t)$ converge smoothly to the $(1,1)$ form $\omega(T_{{\rm max}})=\hat{\omega}_{T_{{\rm max}}}+\ddbar \vp_{T_{{\rm max}}}$, which is positive definite (i.e. a K\"ahler metric) thanks to \eqref{est2}.

We can then use Theorem \ref{short} to solve the parabolic complex Monge-Amp\`ere equation
\begin{equation}\label{maafter}
\left\{
                \begin{aligned}
                  &\frac{\de}{\de t}\vp(t)=\log\frac{(\hat{\omega}_t+\ddbar\vp(t))^n}{\Omega}\\
                  &\vp(T_{{\rm max}})=\vp_{T_{{\rm max}}}\\
                  &\hat{\omega}_t+\ddbar\vp(t)>0,
                \end{aligned}
              \right.
\end{equation}
for $t\in [T_{{\rm max}},T_{{\rm max}}+\ve)$, and for some $\ve>0$ (note that in that proof we had the initial value of $\vp$ equal to zero, while now it is $\vp_{T_{{\rm max}}}$, but the proof there works for this case as well). Therefore $\omega(t):=\hat{\omega}_t+\ddbar\vp(t)$ for $t\in [T_{{\rm max}},T_{{\rm max}}+\ve)$ defines a solution of \eqref{krf} on this time interval, with initial metric equal to $\omega(T_{{\rm max}})$

Lastly, we remark that \eqref{est3} together with a similar argument as before (using Ascoli-Arzel\`a, a diagonal argument, and the analog of \eqref{time} to show uniqueness of the limit) shows that for every $\ell\geq 0$ we have that as $t\to T_{{\rm max}}$ the function
$\frac{\de^\ell}{\de t^\ell}\vp(t)$ converges smoothly to the same function that one gets from differentiating \eqref{maafter} and setting $t=T_{{\rm max}}$. This means that if we define $\vp(t)$ for all $t\in [0,T_{{\rm max}}+\ve)$ by piecing together the flow \eqref{ma} on $[0,T_{{\rm max}})$ together with the flow \eqref{maafter} for $t\in [T_{{\rm max}},T_{{\rm max}}+\ve)$, then the resulting function $\vp(t)$ is smooth in all variables, and gives a solution of the K\"ahler-Ricci flow \eqref{ma} on $[0,T_{{\rm max}}+\ve)$. This is a contradiction to the maximality of $T_{{\rm max}}$.
\end{proof}

We now start the proof of the {\rm a priori} estimates in Theorem \ref{estimates}. First, we prove \eqref{est1} for $k=0$.

Here and in the following, we denote by $C$ a generic positive constant which is allowed to depend only on the initial metric $\omega_0$, and may change from line to line. All such constants $C$ can in principle be made completely explicit.

\begin{lemma}\label{c01}
There is a constant $C>0$, which depends only on $\omega_0$, such that
\begin{equation}\label{c00}
\sup_X |\vp(t)|\leq C,
\end{equation}
for all $t\in [0,T_{{\rm max}}).$
\end{lemma}
\begin{proof}
Let $\ti{\vp}(t)=\vp(t)-At$, for some constant $A>0$ to be determined. We have
$$\frac{\de}{\de t}\ti{\vp}(t)=\log\frac{(\hat{\omega}_t+\ddbar\ti{\vp}(t))^n}{\Omega}-A,$$
for $t\in[0,T_{{\rm max}})$. Fix any $0<\tau<T_{{\rm max}}$ and let the maximum of $\ti{\vp}(t)$ on $X\times[0,\tau]$ be achieved at $(x,t)$. If $t>0$ then at $(x,t)$ we have
$$0\leq \frac{\de}{\de t}\ti{\vp}(t)=\log\frac{(\hat{\omega}_t+\ddbar\ti{\vp}(t))^n}{\Omega}-A\leq \log\frac{\hat{\omega}_t^n}{\Omega}-A,$$
using that
$$0<\hat{\omega}_t+\ddbar\ti{\vp}(t)\leq\hat{\omega}_t,$$
at $(x,t)$. But recall that $\hat{\omega}_t$ are K\"ahler metrics for all $t\in[0,T_{{\rm max}}],$ which vary smoothly in $t$, and so
$$A=1+\sup_{X\times[0,T_{{\rm max}}]}\log\frac{\hat{\omega}_t^n}{\Omega},$$
is a finite, uniform constant, and with this choice of $A$ we obtain a contradiction. Therefore we must have that the maximum of $\ti{\vp}(t)$ is achieved at $t=0$, where this function is zero. This shows that
$$\sup_X \vp(t)\leq At\leq AT_{{\rm max}},$$
for all $t\in [0,T_{{\rm max}}),$ which gives half of the estimate \eqref{c00}.

For the other half, one looks at the function $\vp(t)+Bt$,
where
$$B=1-\inf_{X\times[0,T_{{\rm max}}]}\log\frac{\hat{\omega}_t^n}{\Omega},$$
and argues similarly.
\end{proof}

Having given all the details on how to apply the maximum principle in this case, from now on we will be more brief on this point (in particular, when applying the maximum principle
we will always restrict to a compact time subinterval without mention).

\begin{lemma}\label{c002}
There is a constant $C>0$, which depends only on $\omega_0$, such that
\begin{equation}\label{c000}
\sup_X |\dot{\vp}(t)|\leq C,
\end{equation}
for all $t\in [0,T_{{\rm max}}).$
\end{lemma}
\begin{proof}
We compute
$$\left(\frac{\de}{\de t}-\Delta\right)\vp(t)=\dot{\vp}(t)-\tr{\omega(t)}{(\omega(t)-\hat{\omega}_t)}=\dot{\vp}(t)-n+\tr{\omega(t)}{\hat{\omega}_t},$$
$$\left(\frac{\de}{\de t}-\Delta\right)\dot{\vp}(t)=\frac{n\omega(t)^{n-1}\wedge(\chi+\ddbar\dot{\vp}(t))}{\omega(t)^n}-\Delta\dot{\vp}(t)=\tr{\omega(t)}{\chi},$$
where here and from now on we will always write $\Delta=\Delta_{\omega(t)}$.
Combining these, we obtain the useful equations
\begin{equation}\label{get}
\left(\frac{\de}{\de t}-\Delta\right)(t\dot{\vp}(t)-\vp(t)-nt)=\tr{\omega(t)}{(t\chi-\hat{\omega}_t)}=-\tr{\omega(t)}{\omega_0}<0,
\end{equation}
\begin{equation}\label{get2}
\left(\frac{\de}{\de t}-\Delta\right)((T'-t)\dot{\vp}(t)+\vp(t)+nt)=\tr{\omega(t)}{((T'-t)\chi+\hat{\omega}_t)}=\tr{\omega(t)}{\hat{\omega}_{T'}}>0.
\end{equation}
We won't need \eqref{get} right now, but we record it here for later use.
The maximum principle applied to \eqref{get2} gives that the minimum of $(T'-t)\dot{\vp}(t)+\vp(t)+nt$ is achieved at $t=0$, and so
$$(T'-t)\dot{\vp}(t)+\vp(t)+nt\geq T'\dot{\vp}(0)\geq T'\inf_X \log\frac{\omega_0^n}{\Omega}\geq -C,$$
and since $T'-t\geq T'-T_{{\rm max}}>0$, this implies that
$$\inf_X \dot{\vp}(t)\geq -C,$$
for all $t\in[0,T_{{\rm max}})$, using Lemma \ref{c01}. For the upper bound on $\dot{\vp}(t)$, we observe that
$$\frac{\de}{\de t}\dot{\vp}(t)=\frac{n\omega(t)^{n-1}\wedge(\frac{\de}{\de t}\omega(t))}{\omega(t)^n}=\tr{\omega(t)}{(-\Ric(\omega(t)))}=-R(t),$$
and since locally
$$R(t)=g^{i\ov{j}}R_{i\ov{j}}=-g^{i\ov{j}}\de_i\de_{\ov{j}}\log\det (g_{k\ov{\ell}}),$$
we obtain
$$\frac{\de}{\de t}R(t)=g^{i\ov{q}}g^{p\ov{j}}R_{p\ov{q}}R_{i\ov{j}}-g^{i\ov{j}}\de_i\de_{\ov{j}}\left(g^{p\ov{q}}\frac{\de}{\de t}g_{p\ov{q}}\right)=
|\Ric(\omega(t))|^2_{\omega(t)}+\Delta R(t),$$
and so $\left(\frac{\de}{\de t}-\Delta\right)R(t)\geq 0$, and the minimum principle implies that
\begin{equation}\label{scal}
\inf_X R(t)\geq \inf_XR(0)\geq -C,
\end{equation}
for all $t\in[0,T_{{\rm max}})$. Since $T_{{\rm max}}<\infty$, we can integrate this bound in $t$ and obtain $\sup_X \dot{\vp}(t)\leq C$ for all $t\in[0,T_{{\rm max}})$.
\end{proof}

\begin{theorem}\label{c22}
There is a constant $C>0$, which depends only on $\omega_0$, such that
\begin{equation}\label{c202}
\sup_X \tr{\omega_0}{\omega(t)}\leq C,
\end{equation}
for all $t\in [0,T_{{\rm max}}).$
\end{theorem}
\begin{proof}
Calculate
$$\frac{\de}{\de t}\tr{\omega_0}{\omega(t)}=-\tr{\omega_0}{\Ric(\omega(t))},$$
and at a point with local holomorphic normal coordinates for $\omega_0$ where $\omega(t)$ is diagonal, we have
\[\begin{split}
\Delta \tr{\omega_0}{\omega(t)}&=g^{k\ov{\ell}}\de_k\de_{\ov{\ell}}(g_0^{i\ov{j}}g_{i\ov{j}})=g_0^{i\ov{q}}g_0^{p\ov{j}}R^0_{k\ov{\ell}p\ov{q}} g^{k\ov{\ell}}g_{i\ov{j}}+
g_0^{i\ov{j}}g^{k\ov{\ell}}\de_k\de_{\ov{\ell}}g_{i\ov{j}}\\
&=\sum_{i,k=1}^n R^0_{k\ov{k}i\ov{i}}g^{k\ov{k}}g_{i\ov{i}}-g_0^{i\ov{j}}R_{i\ov{j}}+g_0^{i\ov{j}}g^{k\ov{\ell}}g^{p\ov{q}}\de_kg_{i\ov{q}}\de_{\ov{\ell}}g_{p\ov{j}}\\
&=\sum_{i,k=1}^n R^0_{k\ov{k}i\ov{i}}g^{k\ov{k}}g_{i\ov{i}}-\tr{\omega_0}{\Ric(\omega(t))}+g_0^{i\ov{j}}g^{k\ov{\ell}}g^{p\ov{q}}\overset{0}{\nabla}_kg_{i\ov{q}}\overset{0}{\nabla}_{\ov{\ell}}g_{p\ov{j}},
\end{split}\]
where $\overset{0}{\nabla}$ is the covariant derivative of $\omega_0$. Note that at our point we have that
$$R^0_{k\ov{k}i\ov{i}}={\rm Rm}^0(\de_k,\ov{\de_k},\de_i,\ov{\de_i})\geq -C_0,$$
where $-C_0$ is a lower bound for the bisectional curvature of $\omega_0$ among all $\omega_0$-unit vectors (note the vectors $\de_i,\de_k$ are $\omega_0$-orthonormal at our point).
Therefore
\[\begin{split}
\sum_{i,k=1}^n R^0_{k\ov{k}i\ov{i}}g^{k\ov{k}}g_{i\ov{i}}&\geq-C_0\sum_{i,k=1}^n g^{k\ov{k}}g_{i\ov{i}}=-C_0\left(\sum_{k=1}^n  g^{k\ov{k}}\right)\left(\sum_{i=1}^n  g_{i\ov{i}}\right)\\
&=-C_0(\tr{\omega_0}{\omega(t)}) (\tr{\omega(t)}{\omega_0}),
\end{split}\]
and so
\begin{equation}\label{gettt}
\left(\frac{\de}{\de t}-\Delta\right)\tr{\omega_0}{\omega(t)}\leq C_0(\tr{\omega_0}{\omega(t)}) (\tr{\omega(t)}{\omega_0})-g_0^{i\ov{j}}g^{k\ov{\ell}}g^{p\ov{q}}\overset{0}{\nabla}_kg_{i\ov{q}}\overset{0}{\nabla}_{\ov{\ell}}g_{p\ov{j}}.
\end{equation}
It follows that
\[\begin{split}
\left(\frac{\de}{\de t}-\Delta\right)\log\tr{\omega_0}{\omega(t)}&\leq C_0\tr{\omega(t)}{\omega_0}\\
&-\frac{1}{\tr{\omega_0}{\omega(t)}}\left(g_0^{i\ov{j}}g^{k\ov{\ell}}g^{p\ov{q}}\overset{0}{\nabla}_kg_{i\ov{q}}\overset{0}{\nabla}_{\ov{\ell}}g_{p\ov{j}}
-\frac{|\de \tr{\omega_0}{\omega(t)}|^2_{\omega(t)}}{\tr{\omega_0}{\omega(t)}}\right).
\end{split}\]
Surprisingly, the term inside the big bracket is nonnegative,
$$\left(g_0^{i\ov{j}}g^{k\ov{\ell}}g^{p\ov{q}}\overset{0}{\nabla}_kg_{i\ov{q}}\overset{0}{\nabla}_{\ov{\ell}}g_{p\ov{j}}
-\frac{|\de \tr{\omega_0}{\omega(t)}|^2_{\omega(t)}}{\tr{\omega_0}{\omega(t)}}\right)\geq 0,$$
because it is readily verified that it equals the norm squared
$$g_0^{i\ov{j}}g^{k\ov{\ell}}g^{p\ov{q}}B_{ki\ov{q}}\ov{B_{\ell j\ov{p}}}\geq 0,$$
of the tensor $B$ with components
$$B_{ki\ov{q}}=\overset{0}{\nabla}_kg_{i\ov{q}}-\frac{\de_{k}\tr{\omega_0}{\omega(t)}}{\tr{\omega_0}{\omega(t)}}g_{i\ov{q}}.$$
Indeed,
\[\begin{split}
g_0^{i\ov{j}}g^{k\ov{\ell}}g^{p\ov{q}}B_{ki\ov{q}}\ov{B_{\ell j\ov{p}}}&=g_0^{i\ov{j}}g^{k\ov{\ell}}g^{p\ov{q}}\overset{0}{\nabla}_kg_{i\ov{q}}\overset{0}{\nabla}_{\ov{\ell}}g_{p\ov{j}}
+g_0^{i\ov{j}}g^{p\ov{q}}\frac{|\de\tr{\omega_0}{\omega(t)}|^2_{\omega(t)}}{(\tr{\omega_0}{\omega(t)})^2}g_{i\ov{q}}g_{p\ov{j}}\\
&-2\mathrm{Re}\left(g_0^{i\ov{j}}g^{k\ov{\ell}}g^{p\ov{q}}\frac{\de_{k}\tr{\omega_0}{\omega(t)}}{\tr{\omega_0}{\omega(t)}}g_{i\ov{q}}\overset{0}{\nabla}_{\ov{\ell}}g_{p\ov{j}}\right)\\
&=g_0^{i\ov{j}}g^{k\ov{\ell}}g^{p\ov{q}}\overset{0}{\nabla}_kg_{i\ov{q}}\overset{0}{\nabla}_{\ov{\ell}}g_{p\ov{j}}
+\frac{|\de\tr{\omega_0}{\omega(t)}|^2_{\omega(t)}}{\tr{\omega_0}{\omega(t)}}\\
&-2\mathrm{Re}\left(g_0^{i\ov{j}}g^{k\ov{\ell}}\frac{\de_{k}\tr{\omega_0}{\omega(t)}}{\tr{\omega_0}{\omega(t)}}\overset{0}{\nabla}_{\ov{\ell}}g_{i\ov{j}}\right)\\
&=g_0^{i\ov{j}}g^{k\ov{\ell}}g^{p\ov{q}}\overset{0}{\nabla}_kg_{i\ov{q}}\overset{0}{\nabla}_{\ov{\ell}}g_{p\ov{j}}
-\frac{|\de\tr{\omega_0}{\omega(t)}|^2_{\omega(t)}}{\tr{\omega_0}{\omega(t)}},
\end{split}\]
as claimed, using that $g_0^{i\ov{j}}\overset{0}{\nabla}_{\ov{\ell}}g_{i\ov{j}}=\de_{\ov{\ell}}\tr{\omega_0}{\omega(t)}.$
This gives
\begin{equation}\label{gettt2}
\left(\frac{\de}{\de t}-\Delta\right)\log\tr{\omega_0}{\omega(t)}\leq C_0\tr{\omega(t)}{\omega_0},
\end{equation}
and combining this with \eqref{get} we obtain
$$\left(\frac{\de}{\de t}-\Delta\right)\big(\log\tr{\omega_0}{\omega(t)}+C_0(t\dot{\vp}(t)-\vp(t)-nt)\big)\leq 0,$$
and so the maximum principle implies that this quantity achieves its maximum at $t=0$, and so
$$\log\tr{\omega_0}{\omega(t)}\leq C-C_0(t\dot{\vp}(t)-\vp(t)-nt)\leq C,$$
on $X\times[0,T_{{\rm max}}),$ using Lemmas \ref{c01}, \ref{c002} and the fact that $t\leq T_{{\rm max}}<\infty$.
Exponentiating we obtain \eqref{c202}.
\end{proof}
\begin{corollary}
There is a constant $C>0$, which depends only on $\omega_0$, such that
\begin{equation}\label{c0equiv}
C^{-1}\omega_0\leq \omega(t)\leq C\omega_0,
\end{equation}
for all $t\in [0,T_{{\rm max}}).$
\end{corollary}
\begin{proof}
The bound $\omega(t)\leq C\omega_0$ follows immediately from \eqref{c202}. For the lower bound, note that the flow equation \eqref{ma} together with Lemma \ref{c002} give
\begin{equation}\label{vol}
C^{-1}\omega_0^n\leq \omega(t)^n\leq C\omega_0^n,
\end{equation}
and if at a point we choose coordinates where $\omega_0$ is the identity and $\omega(t)$ is diagonal with eigenvalues $\lambda_j>0, 1\leq j\leq n$, then \eqref{c202} shows that
$$\lambda_j\leq C,$$ for all $j$, while \eqref{vol} implies
$$\prod_{j=1}^n \lambda_j\geq C^{-1},$$
and so for any $j$ we have
$$\lambda_j=\frac{\prod_{i=1}^n \lambda_i}{\prod_{k\neq j} \lambda_k}\geq C^{-1},$$
which exactly says that $\omega(t)\geq C^{-1}\omega_0$.
\end{proof}
Of course \eqref{c0equiv} implies \eqref{est2}.

While all the arguments so far used the maximum principle, the higher order estimates are in fact purely local. For a proof we refer to \cite{ShW}.
\begin{theorem}\label{higher}
Let $U\subset X$ be a nonempty open set, and $\omega(t)$ solve the K\"ahler-Ricci flow \eqref{krf} on $U\times[0,T)$, for $0<T\leq\infty$, with initial K\"ahler metric $\omega_0$. Assume that there exists a constant $C_0>0$ such that
\begin{equation}\label{estt1}
C_0^{-1}\omega\leq \omega(t)\leq C_0\omega,
\end{equation}
on $U\times [0,T)$, for some K\"ahler metric $\omega$ on $X$. Then given any $K\subset U$ compact, and any $k\geq 1$ there is a constant $C$ which depends only on $K,U,k,\omega_0,\omega$ and $C_0$ such that
\begin{equation}\label{higer}
\|\omega(t)\|_{C^k(K,\omega)}\leq C,
\end{equation}
for all $t\in[0,T)$.
Furthermore, for any given $0<\ve<T$, the estimates \eqref{higer} hold for $t\in [\ve,T)$ with a constant $C$ that depends also on $\ve$ but does not depend on $\omega_0$.
\end{theorem}

We can now complete the proof of Theorem \ref{estimates}.

\begin{proof}[Proof of Theorem \ref{estimates}]
We have already established \eqref{est2} and \eqref{est1} for $k=0$, so it remains to show \eqref{est1} for $k\geq 1.$ First note that by a simple covering argument, \eqref{c0equiv} together with Theorem \ref{higher} implies that
\begin{equation}\label{higer2}
\|\omega(t)\|_{C^k(X,\omega_0)}\leq C_k,
\end{equation}
for all $t\in[0,T_{{\rm max}})$, and all $k\geq 1$, where $C_k$ is a uniform constant. But we have
$$\ddbar\vp(t)=\omega(t)-\hat{\omega}_t,$$
and $\hat{\omega}_t$ is a smoothly varying family of K\"ahler metrics for all $t\in [0,T_{{\rm max}}]$,
and so
$$\Delta_{\omega_0} \vp(t)=\tr{\omega_0}{\omega(t)}-\tr{\omega_0}{\hat{\omega}_t},$$
where the function on the right-hand-side is uniformly bounded in $C^k(X,\omega_0)$ for all $k\geq 0$ thanks to \eqref{c0equiv} and \eqref{higer2}.
But for any fixed $0<\alpha<1$ we have the elliptic estimates (see e.g. \cite{Kr})
\[\begin{split}
\|\vp(t)\|_{C^k(X,g_0)}&\leq \|\vp(t)\|_{C^{k,\alpha}(X,g_0)}\leq C_k(\|\Delta_{\omega_0} \vp(t)\|_{C^{k-2,\alpha}(X,g_0)}+\|\vp(t)\|_{C^0(X)})\\
&\leq C_k(\|\Delta_{\omega_0} \vp(t)\|_{C^{k-1}(X,g_0)}+\|\vp(t)\|_{C^0(X)}),
\end{split}\]
for all $k\geq 2$, and so (using Lemma \ref{c01}) we obtain \eqref{est1}.
\end{proof}

\subsection{Examples of calculations of $T$}
First, we look at the case when $n=1$, so $X$ is a compact Riemann surface. It is well-known that $X$ is diffeomorphic to a surface $\Sigma_g$ of genus $g$, for some $g\geq 0$.
Since $H^2(X,\mathbb{R})=\mathbb{R}$, it follows that $H^{1,1}(X,\mathbb{R})=\mathbb{R}$ as well.

\begin{example}
If $g=0$, so $X$ is diffeomorphic to $S^2$, then the uniformization theorem implies that $X$ is in fact biholomorphic to $\mathbb{CP}^1$, so $\mathcal{C}_X$ is generated by $[\omega_{\rm FS}]$ where $\omega_{\rm FS}$ is the Fubini-Study metric, which in the standard coordinate system (writing $\mathbb{CP}^1=\mathbb{C}\cup\{\infty\}$) is locally given by $\omega_{\rm FS}=\ddbar\log (1+|z|^2)$.
Recall that $\omega_{\rm FS}$ satisfies
$$\int_{X}\omega_{\rm FS}=2\pi,$$
and
$$\Ric(\omega_{\rm FS})=2\omega_{\rm FS}.$$
Therefore $2\pi c_1(X)=2[\omega_{\rm FS}]\in\mathcal{C}_X.$ If $\omega_0$ is any K\"ahler metric on $X$, then $[\omega_0]=\lambda[\omega_{\rm FS}]$ for some $\lambda>0$, and the evolved class is
$$[\omega(t)]=[\omega_0]-2\pi tc_1(X)=(\lambda -2t)[\omega_{\rm FS}],$$
which is K\"ahler if and only if $\lambda-2t>0$. Therefore by Theorem \ref{tz} the maximal existence time of the K\"ahler-Ricci flow \eqref{krf} is
$T=\frac{\lambda}{2}$. The limiting class is
$$[\alpha]=[\omega_0]-2\pi Tc_1(X)=0,$$
so in particular $\vol(X,\omega(t))\to 0$ as $t\to T$.
\end{example}
\begin{example}
If $g=1$, so $X$ is diffeomorphic to the torus $T^2$, then the uniformization theorem implies that $X$ is biholomorphic to $\mathbb{C}/\Lambda$ for some lattice $\Lambda\subset\mathbb{C}$. In general different lattices give rise to non-biholomorphic complex tori. In any case, any given Euclidean metric $\omega_{\rm flat}$ on $\mathbb{C}$ is invariant under translations by $\Lambda$ and so it descends to a K\"ahler metric $\omega_{\rm flat}$ on $X$ with
$$\Ric(\omega_{\rm flat})=0.$$
Therefore $c_1(X)=0$, and the flow starting at any initial metric $\omega_0$ does not change the K\"ahler class $[\omega(t)]=[\omega_0]$, and so by Theorem \ref{tz} we get that $T=\infty$. Clearly, the volume of $(X,\omega(t))$ is constant.
\end{example}
\begin{example}
If $g\geq 2$, then the uniformization theorem implies that $X$ is biholomorphic to $B/\Gamma$ were $B=\{z\in \mathbb{C}\ |\ |z|<1\}$ is the unit disc and $\Gamma$ is some discrete group which acts on $B$ by isometries of the Poincar\'e metric
$$\omega_{\rm hyp}=-\ddbar\log(1-|z|^2),$$
on $B.$ Therefore $\omega_{\rm hyp}$ descends to a K\"ahler metric on $X$, which satisfies
$$\Ric(\omega_{\rm hyp})=-2\omega_{\rm hyp},$$
by direct calculation. Therefore, if $\omega_0$ is any K\"ahler metric on $X$, then
$[\omega_0]=\lambda[\omega_{\rm hyp}]$ for some $\lambda>0$, and the evolved class is
$$[\omega(t)]=[\omega_0]-2\pi tc_1(X)=(\lambda +2t)[\omega_{\rm hyp}],$$
which is K\"ahler for all $t\geq 0$. Therefore by Theorem \ref{tz} the maximal existence time of the K\"ahler-Ricci flow \eqref{krf} is
$T=\infty$. The volume of $\vol(X,\omega(t))$ grows like $t$ as $t\to \infty$, and the cohomology class of the rescaled metrics $\frac{\omega(t)}{t}$ converges to $-2\pi c_1(X)$.
\end{example}
\begin{example}
Let $X=\mathbb{CP}^1\times\mathbb{CP}^1$, with projections $\pi_1,\pi_2$ to the two factors. Then $H^{1,1}(X,\mathbb{R})=\mathbb{R}^2$, generated by $a=\pi_1^*[\omega_{\rm FS}]$ and
$b=\pi_2^*[\omega_{\rm FS}]$, and it is easy to see that a class $[\alpha]=\lambda_1 a+\lambda_2b$ is K\"ahler if and only if $\lambda_1>0$ and $\lambda_2>0$.
Also, the product metric $\omega_{\rm prod}=\pi_1^*\omega_{\rm FS}+\pi_2^*\omega_{\rm FS}$ satisfies
$$\Ric(\omega_{\rm prod})=2\omega_{\rm prod},$$
and so $2\pi c_1(X)=2(a+b)$. Therefore the evolved class is
$$[\omega(t)]=[\omega_0]-2\pi tc_1(X)=(\lambda_1 -2t)a+(\lambda_2 -2t)b,$$
and so by Theorem \ref{tz} the maximal existence time is
$$T=\min\left(\frac{\lambda_1}{2},\frac{\lambda_2}{2}\right).$$
The limiting class as $t\to T$ is either zero, or a multiple of $a$ or $b$, and so we always have that $\vol(X,\omega(t))\to 0$ as $t\to T$.
\end{example}
\begin{example}\label{blow}
Let $\pi:X\to\mathbb{CP}^2$ be the blowup of $\mathbb{CP}^2$ at a point $p$, with exceptional divisor $E=\pi^{-1}(p)\cong\mathbb{CP}^1$. Then we have that
$H^{1,1}(X,\mathbb{R})=\mathbb{R}^2$, generated by $a=\frac{\pi^*[\omega_{\rm FS}]}{2\pi}$ and $b$, the Poincar\'e dual of $E$, and also
$$c_1(X)=3a-b.$$
Consider a $(1,1)$ class $[\alpha]=\lambda_1a+\lambda_2b$. The Nakai-Moishezon criterion of Bunchdahl \cite[Corollary 15]{Bu} and Lamari \cite{La} (which was extended to all dimensions by Demailly-P\u{a}un \cite{DP2}) in this case says that $[\alpha]\in \mathcal{C}_X$ if and only if
\begin{equation}\label{nm}
\int_X\alpha^2>0,\quad \int_E\alpha>0,\quad \int_H\alpha>0,
\end{equation}
where $H=\pi^{-1}(L)$ and $L\cong\mathbb{CP}^1$ is a projective line in $\mathbb{CP}^2$ which does not pass through $p$. The Poincar\'e dual of $L$ inside $\mathbb{CP}^2$ is $\frac{[\omega_{\rm FS}]}{2\pi}$, and so the Poincar\'e dual of $H$ inside $X$ is $a$, and so \eqref{nm} is equivalent to
\begin{equation}\label{nm2}
\int_X\alpha^2>0,\quad \int_X\alpha\wedge a>0,\quad \int_X\alpha\wedge b>0.
\end{equation}
We also have that
\begin{equation}\label{fa1}
\int_X a^2=\frac{1}{4\pi^2}\int_{\mathbb{CP}^2}\omega_{\rm FS}^2=1,
\end{equation}
\begin{equation}\label{fa2}
\int_X b^2=\int_E b=-1,
\end{equation}
\begin{equation}\label{fa3}
\int_X a\wedge b=\int_Ea=0,
\end{equation}
where \eqref{fa2} is well-known and \eqref{fa3} holds because we can represent $a$ by a smooth form supported in an arbitrarily small neighborhood of $H$, and since $H$ is disjoint from $E$ we may choose a representative of $a$ which vanishes everywhere on $E$. Using these, we immediately see that \eqref{nm2} is equivalent to
\begin{equation}\label{nm3}
\lambda_1^2-\lambda_2^2>0,\quad \lambda_1>0,\quad -\lambda_2>0,
\end{equation}
or equivalently
\begin{equation}\label{nm4}
0<-\lambda_2<\lambda_1.
\end{equation}
So if $[\omega_0]=\lambda_1a+\lambda_2b$ is any K\"ahler class on $X$ (so \eqref{nm4} holds), then the evolved class is given by
$$[\omega(t)]=[\omega_0]-2\pi tc_1(X)=(\lambda_1 -6\pi t)a-(-\lambda_2-2\pi t)b.$$
This class remains K\"ahler as long as $-\lambda_2-2\pi t>0$ and $\lambda_1 -6\pi t>-\lambda_2-2\pi t$,
and so by Theorem \ref{tz} the maximal existence time is
$$T=\min\left(\frac{\lambda_1+\lambda_2}{4\pi},\frac{-\lambda_2}{2\pi}\right).$$
We have that
$$\vol(X,\omega(t))=(\lambda_1-6\pi t)^2-(-\lambda_2-2\pi t)^2.$$
If $\lambda_1\leq -3\lambda_2$, then $T=\frac{\lambda_1+\lambda_2}{4\pi}$ and so $\vol(X,\omega(t))\to 0$ as $t\to T$. If instead
$\lambda_1> -3\lambda_2$, then $T=\frac{-\lambda_2}{2\pi}$ and so
$$\vol(X,\omega(t))\to (\lambda_1+3\lambda_2)^2>0$$
as $t\to T$. This is the first example that we encounter of a finite time noncollapsed singularity. We will study these in more detail in the next section.
\end{example}

\section{Finite time singularities}\label{sectfin}

\subsection{Finite time singularities of the K\"ahler-Ricci flow}
In this section we assume that the K\"ahler-Ricci flow \eqref{krf} has a finite time singularity at time $T<\infty$. The limiting class of the flow is
$$[\alpha]=\lim_{t\to T}[\omega(t)]=[\omega_0]-2\pi Tc_1(X),$$
and it is a nef class, since it is a limit of K\"ahler classes. Not all nef classes arise in this way, and we have the following elementary observation:

\begin{proposition}\label{nef}
Let $X$ be a compact K\"ahler manifold and $[\alpha]\in\de\mathcal{C}_X$ a nef $(1,1)$ class, which is not K\"ahler. Then there exists a K\"ahler metric $\omega_0$ such that the K\"ahler-Ricci flow \eqref{krf} has a finite time singularity with limiting class $[\alpha]$ if and only if $[\alpha]+\lambda c_1(X)\in \mathcal{C}_X$ for some $\lambda>0$.
In this case the maximal existence time is $T=\frac{\lambda}{2\pi}.$
\end{proposition}
\begin{proof}
If there exists a metric $\omega_0$ such that the K\"ahler-Ricci flow \eqref{krf} has a finite time singularity at time $T$ with limiting class $[\alpha]$, then we know that
$$[\alpha]=[\omega_0]-2\pi Tc_1(X),$$
and so $[\alpha]+2\pi Tc_1(X)\in\mathcal{C}_X$.

Conversely, if $[\alpha]+\lambda c_1(X)\in \mathcal{C}_X$ for some $\lambda>0$, we choose a K\"ahler metric $\omega_0$ in this class, and evolve it by the K\"ahler-Ricci flow \eqref{krf}.
The class of the evolved metric is
$$[\omega(t)]=[\omega_0]-2\pi tc_1(X)=[\alpha]+(\lambda-2\pi t)c_1(X)=\left(1-\frac{2\pi t}{\lambda}\right)[\omega_0]+\frac{2\pi t}{\lambda}[\alpha].$$
For $0\leq t<\frac{\lambda}{2\pi}$ this is a sum of a K\"ahler class and a nef class, and so it is K\"ahler, while for $t=\frac{\lambda}{2\pi}$ this equals $[\alpha]$ which is nef but not K\"ahler. It follow from Theorem \ref{tz} that the maximal existence time is $T=\frac{\lambda}{2\pi}<\infty$ and the limiting class is $[\alpha]$.
\end{proof}

\subsection{Noncollapsed finite time singularities}
We will say that a finite time singularity at time $T<\infty$ is noncollapsed if $\vol(X,\omega(t))\geq C^{-1}$ for all $t\in [0,T)$. As we saw, this is equivalent to the cohomological property
$$\int_X (\omega_0-2\pi T\Ric(\omega_0))^n=\int_X\alpha^n>0.$$
In other words, it is equivalent to requiring that the limiting class $[\alpha]$ be nef and big. Recall that in this case the null locus $\Null(\alpha)$, defined in \eqref{null}, is a proper analytic subvariety of $X$.

\begin{example}
Going back to Example \ref{blow}, if we choose the initial class to be $[\omega_0]=4a-b$, then we have $T=\frac{1}{2\pi}$ and the limiting class is
$$[\alpha]=a=\frac{\pi^*[\omega_{\rm FS}]}{2\pi}.$$
As shown in \eqref{fa3}, we have that
$$\int_E a=0,$$
so certainly $E\subset\Null(a)$. Since $\int_X a^2>0$ (see \eqref{fa1}), we have that $\Null(a)$ is not equal to $X$.
If $C\subset X$ is an irreducible curve which is not equal to $E$, then $C$ cannot be contained in $E$ and so its image $\pi(C)$ is an irreducible curve in $\mathbb{CP}^2$. We then have
$$\int_C a=\frac{1}{2\pi}\int_{\pi(C)}\omega_{\rm FS}>0,$$
since $\int_{\pi(C)}\omega_{\rm FS}$ equals the volume of $\pi(C)$ with respect to the Fubini-Study metric. Therefore we have shown that $\Null(a)=E$.
\end{example}

The following is the main result of this section:

\begin{theorem}[Collins-T. \cite{CT}]\label{fint}
Let $(X,\omega_0)$ be a compact K\"ahler manifold such that the K\"ahler-Ricci flow \eqref{krf} starting at $\omega_0$ has a noncollapsed finite time singularity at $T<\infty$.
Let $\alpha=\omega_0-2\pi T\Ric(\omega_0)$. Then there is a K\"ahler metric $\omega_T$ on $X\backslash \Null(\alpha)$ such that
$$\omega(t)\to\omega_T,$$
in $C^\infty_{\mathrm{loc}}(X\backslash \Null(\alpha))$ as $t\to T$.
\end{theorem}

When $X$ is projective and $[\omega_0]\in H^2(X,\mathbb{Q})$ this was known earlier: indeed in this case the limiting class $[\alpha]$ is the first Chern class of a $\mathbb{Q}$-divisor $D$, and it follows from a trick of Tsuji \cite{Ts} (cf. \cite{TiZ}) that we have uniform $C^\infty_{\mathrm{loc}}$ estimates on compact sets away from the intersection of the supports of all effective $\mathbb{Q}$-divisors $E$ such that $D-E$ is ample (such divisors exist thanks to ``Kodaira's Lemma'' \cite[Proposition 2.2.6]{Laz}). But this intersection equals the ``augmented base locus'' of $D$, as shown in \cite[Remark 1.3]{ELMNP}, and this in turn equals $\Null(c_1(D))$ thanks to Nakamaye's Theorem \cite{Na}. Our work in \cite{CT} extends Nakamaye's Theorem to real $(1,1)$ classes on K\"ahler manifolds, and this is the key new ingredient.

Following \cite{EMT} we define the singularity formation set of the flow $\Sigma$ (which depends on the initial metric $\omega_0$) by
$$\Sigma=X\backslash \{x\in X\ |\ \exists U\ni x \textrm{ open, }\exists C>0, \textrm{ s.t. } |\textrm{Rm}(t)|_{\omega(t)}\leq C \textrm{ on } U\times [0,T)\},$$
where $\mathrm{Rm}(t)$ denotes the curvature tensor of $\omega(t)$.

We have the following conjecture:

\begin{conjecture}[Feldman-Ilmanen-Knopf \cite{FIK}, Campana (see \cite{Zh2})]\label{fik}
For every finite time singularity of the K\"ahler-Ricci flow the singularity formation set $\Sigma$ is an analytic subvariety.
\end{conjecture}

This conjecture was solved in \cite{CT}:

\begin{theorem}[Collins-T. \cite{CT}]\label{good}
Conjecture \ref{fik} is true, and we have $$\Sigma=\Null(\alpha),$$
where $[\alpha]=[\omega_0]-2\pi T c_1(X)$ is the limiting class. In other words, $\Sigma$ is the union of all irreducible analytic subvarieties whose
volume goes to zero as $t\to T$.
\end{theorem}
As we will see, this is a simple application of Theorem \ref{fint}.

First, we rewrite the K\"ahler-Ricci flow as a parabolic complex Monge-Amp\`ere equation. This is similar to the setup we had in section \ref{sectmax}, but there are some key differences. We define $\alpha=\omega_0-2\pi T\Ric(\omega_0)$, which is a closed real $(1,1)$ form with no positivity properties in general, and let
$$\hat{\omega}_t=\frac{1}{T}((T-t)\omega_0+t\alpha), \quad 0\leq t\leq T,$$
which are forms cohomologous to $\omega(t)$, again with no positivity in general. We also let $\chi=\frac{1}{T}(\alpha-\omega_0)$ so that we can write $\hat{\omega}_t=\omega_0+t\chi$, and we choose a smooth positive volume form $\Omega$ with $\Ric(\Omega)=-\chi$. Then, as in section \ref{sectmax}, the K\"ahler-Ricci flow \eqref{krf} is equivalent to
\begin{equation}\label{ma3}
\left\{
                \begin{aligned}
                  &\frac{\de}{\de t}\vp(t)=\log\frac{(\hat{\omega}_t+\ddbar\vp(t))^n}{\Omega}\\
                  &\vp(0)=0\\
                  &\hat{\omega}_t+\ddbar\vp(t)>0.
                \end{aligned}
              \right.
\end{equation}

\begin{lemma}
There is a constant $C>0$ such that
\begin{equation}\label{up}
\vp(t)\leq C,
\end{equation}
\begin{equation}\label{upd}
\dot{\vp}(t)\leq C,
\end{equation}
on $X\times [0,T)$.
\end{lemma}
\begin{proof}
Recall from \eqref{scal} that we have $R(t)\geq -C$ on $X\times[0,T)$. Since $$\frac{\de}{\de t}\dot{\vp}(t)=-R(t),$$
this gives $\frac{\de}{\de t}\dot{\vp}(t)\leq C$. Integrating in $t$ we obtain \eqref{upd}, and integrating again we get \eqref{up}.
\end{proof}

Next, we give two equivalent definitions of $\Sigma$, following Z. Zhang \cite{Zh}.
\begin{proposition}\label{equiv}
We have that
\[\begin{split}
\Sigma&=X\backslash \{x\in X\ |\ \exists\ U\ni x \textrm{ open, }\exists\ C>0, \textrm{ s.t. } R(t)\leq C \textrm{ on } U\times [0,T)\}\\
&=X\backslash \{x\in X\ |\ \exists\ U\ni x \textrm{ open, }\exists\ \omega_U \textrm{ K\"ahler metric on $U$},\\
&\quad\quad\quad\quad\quad\quad\quad\quad\quad\quad \textrm{s.t. } \omega(t)\to\omega_U \textrm{ in }C^\infty(U)\textrm{ as }t\to T\},
\end{split}\]
where $R(t)$ is the scalar curvature of $\omega(t)$.
\end{proposition}
\begin{proof}
It is clear that if the metric $\omega(t)$ converge smoothly to a limit K\"ahler metric on some open set $U$ then we have $|\textrm{Rm}(t)|_{\omega(t)}\leq C$ on $U$. It is also
clear that a uniform bound on the curvature tensor implies an upper bound on the scalar curvature. Therefore we are left to show that if $R(t)\leq C$ on $U\times [0,T)$, where $U$ is an open set which contains a given point $x$, then on a possibly smaller open neighborhood $U'$ of $x$ we have smooth convergence of the metrics to a limit K\"ahler metric on $U'$.

To see this, first recall from \eqref{scal} that the bound $R\geq -C$ always holds on $X\times [0,T)$. Therefore on $U\times [0,T)$ we have $|R|\leq C$, and differentiating \eqref{krf} we have
$$\frac{\de}{\de t}\dot{\vp}=-R.$$
We conclude that on $U\times [0,T)$ we have $|\ddot{\vp}|\leq C$, and integrating in time this gives $|\vp|+|\dot{\vp}|\leq C$ on this set. The quantity $t\dot{\vp}-\vp-nt$ is therefore uniformly bounded on $U\times [0,T)$ and satisfies (thanks to \eqref{get})
$$\left(\frac{\de}{\de t}-\Delta\right)(t\dot{\vp}-\vp-nt)=\tr{\omega}{(t\chi-\hat{\omega}_t)}=-\tr{\omega}{\omega_0}.$$
Recall that from \eqref{gettt2} we also have
$$\left(\frac{\de}{\de t}-\Delta\right)\log\tr{\omega_0}{\omega}\leq C\tr{\omega}{\omega_0},$$
and so
$$\left(\frac{\de}{\de t}-\Delta\right)(\log\tr{\omega_0}{\omega}+C(t\dot{\vp}-\vp-nt))\leq 0,$$
This implies that this quantity achieves its maximum at $t=0$, and so
$$\tr{\omega_0}{\omega}\leq Ce^{-Ct\dot{\vp}+C\vp+Cnt}\leq Ce^{-C\dot{\vp}},$$
holds on $X\times[0,T)$. In particular, on $U\times [0,T)$ we obtain
$$\tr{\omega_0}{\omega}\leq C.$$
From the flow equation $\omega(t)^n=e^{\dot{\vp}}\Omega$ we also have $\omega(t)^n\geq C^{-1}\omega_0^n$ on $U\times [0,T)$, and so we conclude that
$$C^{-1}\omega_0\leq \omega(t)\leq C\omega_0,$$
on $U\times [0,T)$.  The local estimates of \cite{ShW} then give uniform $C^\infty$ bounds for $\omega(t)$ on $U'\times[0,T)$, for a smaller neighborhood $U'$ of $x$, and from these we easily obtain smooth convergence to a limit K\"ahler metric on $U'$.
\end{proof}

As a corollary, we see that the scalar curvature blows up at a finite time singularity \cite{Zh}:
\begin{corollary}
For every finite time singularity of the K\"ahler-Ricci flow the singularity formation set $\Sigma$ is nonempty, and furthermore we have that $\limsup_{t\to T}\sup_X R(t)=+\infty$.
\end{corollary}
\begin{proof}
Thanks to Proposition \ref{equiv}, if we had $\Sigma=\emptyset$ then the metrics $\omega(t)$ would converge in $C^\infty(X)$ to a limiting K\"ahler metric in the class $[\alpha]$, contradicting the fact that $[\alpha]$ is not in the K\"ahler cone. The blow up of the supremum of the scalar curvature also follows directly from Proposition \ref{equiv}.
\end{proof}

Assuming Theorem \ref{fint} we can now prove Theorem \ref{good}.
\begin{proof}[Proof of Theorem \ref{good}]
If $x\not\in\Null(\alpha)$, then by Theorem \ref{fint} the metrics $\omega(t)$ converge smoothly in a neighborhood of $x$ to a limiting K\"ahler metric. In particular
the curvature of $\omega(t)$ remains uniformly bounded near $x$, and therefore $x\not\in\Sigma$.

On the other hand, given $x\in \Null(\alpha)$, suppose that there exist an open set $U$ containing $x$, and a K\"ahler metric $\omega_T$ on $U$ such that $\omega(t)$ converges to $\omega_T$ in $C^\infty(U)$ as $t\to T$. Then, by definition of $\Null(\alpha)$, there is a positive-dimensional irreducible analytic subvariety $V\subset X$ which contains $x$ and with
$$\int_V\alpha^k=0,$$
where $k=\dim V$, and as usual $\alpha=\omega_0-2\pi T\Ric(\omega_0)$. Then we have that as $t\to T$ the integral
$$\int_V \omega(t)^k$$
converges to zero, since $[\omega(t)]\to [\alpha]$. But we also have
$$\int_V \omega(t)^k\geq \int_{V\cap U} \omega(t)^k\overset{t\to T}{\xrightarrow{\hspace*{1cm}}}\int_{V\cap U}\omega_T^k>0,$$
which is a contradiction. Therefore, using Proposition \ref{equiv}, we see that $x\in\Sigma$.
\end{proof}

We now turn to the proof of Theorem \ref{fint}. The key ingredient is the following theorem, which provides a suitable barrier function, and which is a general statement independent of the K\"ahler-Ricci flow.

\begin{theorem}[Collins-T. \cite{CT}]\label{ct}
Let $(X,\omega_0)$ be a compact K\"ahler manifold and $\alpha$ a closed real $(1,1)$ form whose class $[\alpha]$ is nef, and with $\int_X\alpha^n>0$. Then there exists an upper semicontinuous $L^1$ function $\psi:X\to\mathbb{R}\cup\{-\infty\},$ which equals $-\infty$ on $\Null(\alpha)$, which is finite and smooth on $X\backslash \Null(\alpha)$, and such that
$$\alpha+\ddbar\psi\geq \ve\omega_0,$$
on $X\backslash \Null(\alpha)$, for some $\ve>0$.
\end{theorem}

Note that we have that $\psi$ is globally bounded above on $X$, and so up to subtracting a constant from it we may assume that $\psi\leq 0$ on $X$. The proof of Theorem \ref{ct} is quite technical and involves very different techniques from the ones in these notes. Therefore we will skip its proof, referring the interested reader to the original article \cite{CT} or to the survey \cite{To3}. For the reader who is familiar with these concepts (see e.g. \cite{CT}), Theorem \ref{ct} easily implies that the null locus of a nef and big $(1,1)$ class on a compact K\"ahler manifold equals its non-K\"ahler locus, which is also the complement of its ample locus.

On $X\backslash \Null(\alpha)$ we have
\begin{equation}\label{positive}
\begin{split}
\hat{\omega}_t+\ddbar\psi&=\frac{1}{T}((T-t)(\omega_0+\ddbar\psi)+t(\alpha+\ddbar\psi))\\
&\geq \frac{T-t}{T}(\omega_0-\alpha)+\frac{t}{T}\ve\omega_0\\
&\geq \frac{\ve}{2}\omega_0,
\end{split}
\end{equation}
if $t\in [T-\delta,T+\delta],$ for some $\delta>0$.

\begin{lemma}\label{uno}
There is a constant $C>0$ such that
$$\dot{\vp}\geq C\psi-C,$$
on $X\times [0,T)$. Equivalently, we have
$$\omega^n\geq C^{-1}e^{C\psi}\omega_0^n.$$
\end{lemma}
\begin{proof}
Let
$$Q=(T+\delta-t)\dot{\vp}+\vp-\psi+nt,$$
which is smooth on $(X\backslash \Null(\alpha))\times[0,T),$ equal to $+\infty$ on $\Null(\alpha)$, and is bounded below on $X$ for each fixed $t\in [0,T)$.
Therefore $Q\geq -C$ holds on $X\times [0,T-\delta]$, for some uniform constant $C$.

Our goal is to show that in fact $Q\geq -C$ on $X\times [0,T)$. Given $T'\in (T-\delta,T)$ suppose that the minimum of $Q$ on $X\times [T-\delta,T']$ is achieved
at a point $(x,t)$, with $t\in (T-\delta,T']$. We must have $x\not\in\Null(\alpha)$, and so at $(x,t)$ we have
\[\begin{split}
0\geq \left(\frac{\de}{\de t}-\Delta\right)Q&=\tr{\omega}{((T+\delta-t)\chi+\hat{\omega}_t+\ddbar\psi)}\\
&=\tr{\omega}{(\hat{\omega}_{T+\delta}+\ddbar\psi)}\\
&\geq \frac{\ve}{2}\tr{\omega}{\omega_0}>0,
\end{split}\]
using \eqref{positive}. This contradiction shows that the minimum of $Q$ on $X\times [T-\delta,T']$ is achieved at time $T-\delta$, where we have $Q\geq -C$. Since $T'<T$ was arbitrary, we conclude that $Q\geq -C$ on $X\times [0,T)$. This gives
$$(T+\delta-t)\dot{\vp}\geq -\vp+\psi-nt-C\geq \psi-C,$$
$$\dot{\vp}\geq \frac{\psi-C}{T+\delta-t}\geq C\psi-C,$$
since $T+\delta-t\geq \delta$ and $\psi\leq 0$.

The equivalent estimate for the volume form follows from the flow equation.
\end{proof}

\begin{lemma}\label{tre}
There is a constant $C>0$ such that
$$\tr{\omega_0}{\omega}\leq Ce^{-C\psi},$$
on $X\times [0,T)$.
\end{lemma}
\begin{proof}
From \eqref{gettt2} we have
$$\left(\frac{\de}{\de t}-\Delta\right)\log\tr{\omega_0}{\omega}\leq C\tr{\omega}{\omega_0},$$
and from \eqref{get}
$$\left(\frac{\de}{\de t}-\Delta\right)(t\dot{\vp}-\vp-nt)=\tr{\omega}{(t\chi-\hat{\omega}_t)}=-\tr{\omega}{\omega_0},$$
and so
$$\left(\frac{\de}{\de t}-\Delta\right)(\log\tr{\omega_0}{\omega}+C(t\dot{\vp}-\vp-nt))\leq 0,$$
and by the maximum principle, the maximum of this quantity on $X\times[0,T)$ is achieved at $t=0$. This gives
$$\log\tr{\omega_0}{\omega}\leq C(-t\dot{\vp}+\vp+nt)+C\leq C-C\dot{\vp}\leq C-C\psi,$$
on $X\times[0,T)$, where we used Lemma \ref{uno}. Exponentiating gives what we want.
\end{proof}

\begin{proof}[Proof of Theorem \ref{fint}]
Given a compact set $K\subset X\backslash \Null(\alpha)$ with nonempty interior, we have $\inf_K\psi\geq -C_K$ (here and in the following we denote by $C_K$ a constant which depends on the compact set), and so thanks to Lemmas \ref{uno} and \ref{tre} we see that
$$C_K^{-1}\omega_0\leq \omega(t)\leq C_K \omega_0,$$
on $K\times[0,T)$. The local estimates of \cite{ShW} then give uniform $C^k$ bounds for $\omega(t)$ on compact subsets of $X\backslash \Null(\alpha)$, and arguing as in the proof of Theorem \ref{tz}, we easily obtain a K\"ahler metric
$\omega_T$ on $X\backslash\Null(\alpha)$ such that $\omega(t)$ converge to $\omega_T$ in $C^\infty_{\mathrm{loc}}(X\backslash\Null(\alpha))$ as $t\to T$.
\end{proof}

\subsection{A conjectural uniform bound for the potential}
We now mention a conjecture raised explicitly by Zhang \cite[Conjecture 5.1]{Zh2}:

\begin{conjecture}\label{lowerbd}
For every finite time solution of \eqref{ma3}, there is a constant $C>0$ such that
$$\vp(t)\geq -C,$$
on $X\times[0,T)$.
\end{conjecture}

Note that we do not necessarily assume that the singularity is non-collapsed.
Consider now the following conjecture, which is not about the K\"ahler-Ricci flow.

\begin{conjecture}\label{bpf}
Let $X$ be a compact K\"ahler manifold and $[\alpha]$ a nef $(1,1)$ class such that $[\alpha]+\lambda c_1(X)$ is a K\"ahler class for some $\lambda>0$.
Then a closed positive current with minimal singularities in the class $[\alpha]$ has bounded potential.
\end{conjecture}

The condition that a closed positive current with minimal singularities in the class $[\alpha]$ has bounded potential, is equivalent to the following statement (which does not involve currents, and can be taken as the definition in these notes): there is a constant $C_0>0$ such that for every $\ve>0$ there exists $\eta_\ve\in C^\infty(X,\mathbb{R})$ such that $\alpha+\ddbar\eta_\ve\geq -\ve\omega_0$ and
$\sup_X |\eta_\ve|\leq C_0$. The equivalence follows immediately from Demailly's regularization theorem for closed positive $(1,1)$ currents \cite{De}. In particular this condition holds if the class $[\alpha]$ has a smooth semipositive representative.

Conjecture \ref{bpf} is a transcendental (weak) version of the base-point-free theorem \cite{KMM}, which implies that Conjecture \ref{bpf} is true when $X$ is projective and
$[\alpha]\in (H^{1,1}(X,\mathbb{R})\cap H^2(X,\mathbb{Q}))\otimes\mathbb{R}=:NS^1(X,\mathbb{R})$. In fact, in this case the class $[\alpha]$ even has a smooth semipositive representative, and Tian conjectures in \cite{Ti3} that this is the case also in the setting of Conjecture \ref{bpf}.

Interestingly, these two conjectures are equivalent:
\begin{proposition}
Conjectures \ref{lowerbd} and \ref{bpf} are equivalent.
\end{proposition}
\begin{proof}
Assume Conjecture \ref{lowerbd}. Given $[\alpha]$ a nef class such that $[\alpha]+\lambda c_1(X)$ is a K\"ahler class, fix a K\"ahler metric $\omega_0$ in this class.
Since Conjecture \ref{bpf} is trivial if $[\alpha]$ is K\"ahler, we may assume that $[\alpha]$ is on the boundary of the K\"ahler cone. Then the K\"ahler-Ricci flow \eqref{krf} starting
at $\omega_0$ has a solution defined on the maximal time interval $[0,T)$ where $T=\frac{\lambda}{2\pi}$. We choose the representative $\alpha=\omega_0-T\Ric(\omega_0)$ of the class $[\alpha]$, and as before we let $\hat{\omega}_t=\frac{1}{T}((T-t)\omega_0+t\alpha)$ and $\chi=\frac{1}{T}(\alpha-\omega_0)$.
Since we know that $\vp(t)\leq C$ on $X\times[0,T)$, we get a uniform $C^0$ bound for $\vp(t)$, independent of $t$. Then
$$\alpha+\ddbar\vp(t)=\hat{\omega}_t+\ddbar\vp(t)+(T-t)\chi=\omega(t)+(T-t)\chi>(T-t)\chi,$$
and $(T-t)\chi$ goes to zero smoothly as $t\to T$. This proves that a closed positive current with minimal singularities in the class $[\alpha]$ has bounded potential.

Conversely, assume a closed positive current with minimal singularities in the class $[\alpha]$ has bounded potential, and consider a solution of \eqref{krf} with a singularity at time $T<\infty$. After writing the flow as \eqref{ma3} as before, we compute for any $\ve>0$
\[\begin{split}
\left(\frac{\de}{\de t}-\Delta\right)&((\vp+(T-t)\dot{\vp}+nt)+\ve(\vp-t\dot{\vp}+nt)-\eta_\ve)\\
&=\tr{\omega(t)}{(\alpha+\ve\omega_0+\ddbar\eta_\ve)}\geq 0,
\end{split}\]
and so by the minimum principle (together with $\eta_\ve\leq C$, independent of $\ve$) we obtain
$$((\vp+(T-t)\dot{\vp}+nt)+\ve(\vp-t\dot{\vp}+nt)-\eta_\ve)\geq -C,$$
or in other words
$$(1+\ve)\vp+(T-t-\ve t)\dot{\vp}\geq\eta_\ve -C\geq-C,$$
using that $\eta_\ve\geq -C,$ independent of $\ve$. We can then let $\ve\to 0$, and recalling that $\dot{\vp}\leq C$, we finally obtain
$\vp\geq-C$ on $X\times [0,T)$.
\end{proof}

The following can be viewed as partial progress towards Conjecture \ref{lowerbd} (which would be the same statement with $\nu=0$, and without the noncollapsed hypothesis).
\begin{proposition}\label{due}
Suppose that the limiting class $[\alpha]$ satisfies the noncollapsed condition $\int_X\alpha^n>0$, and let $\psi$ be as in Theorem \ref{ct}. Then for every $\nu>0$ there is a constant $C_\nu>0$ such that
$$\vp\geq \nu\psi-C_\nu,$$
on $X\times [0,T)$.
\end{proposition}
\begin{proof}
Since the class $[\alpha]$ is nef, for every $\nu>0$ there is a smooth function $\rho_\nu$ such that $\alpha+\ddbar\rho_\nu\geq -\nu\ve \omega_0$, where $\ve$ is as in Theorem \ref{ct}. Then away from $\Null(\alpha)$ we have
$$\alpha+\ddbar(\nu\psi+(1-\nu)\rho_\nu)\ge \nu^2\ve\omega_0.$$
As in \eqref{positive} it follows that
\begin{equation}\label{positive2}
\hat{\omega}_t+\ddbar(\nu\psi+(1-\nu)\rho_\nu)\geq \frac{\nu^2\ve}{2}\omega_0,
\end{equation}
on $X\backslash\Null(\alpha)$, provided $t\in [T-\delta,T+\delta],$ for some $\delta>0$. For simplicity write $\psi_\nu=\nu\psi+(1-\nu)\rho_\nu$, and
let
$$Q=\vp-\psi_\nu+At,$$
where $A>0$ is a constant to be determined. The function $Q$
is smooth on $(X\backslash \Null(\alpha))\times[0,T),$ equal to $+\infty$ on $\Null(\alpha)$, and is bounded below on $X$ for each fixed $t\in [0,T)$.
Therefore $Q\geq -C$ holds on $X\times [0,T-\delta]$, for some uniform constant $C$.

Our goal is to show that in fact $Q\geq -C$ on $X\times [0,T)$. Given $T'\in (T-\delta,T)$ suppose that the minimum of $Q$ on $X\times [T-\delta,T']$ is achieved
at a point $(x,t)$, with $t\in (T-\delta,T']$. We must have $x\not\in\Null(\alpha)$, and so at $(x,t)$ we have, using \eqref{positive2},
\[\begin{split}
0\geq \frac{\de Q}{\de t}&=\log\frac{(\hat{\omega}_t+\ddbar\psi_\nu+\ddbar Q)^n}{\Omega}+A\\
&\geq \log\frac{(\hat{\omega}_t+\ddbar\psi_\nu)^n}{\Omega}+A\\
&\geq \log\frac{\left(\frac{\nu^2\ve}{2}\omega_0\right)^n}{\Omega}+A\geq -C+A>0,
\end{split}\]
provided we choose $A>C.$
This contradiction shows that the minimum of $Q$ on $X\times [T-\delta,T']$ is achieved at time $T-\delta$, where we have $Q\geq -C$. Since $T'<T$ was arbitrary, we conclude that $Q\geq -C$ on $X\times [0,T)$, which is what we wanted to prove.
\end{proof}

\subsection{Expected behavior at noncollapsed finite time singularities}
Next, we discuss what is expected to hold in the case of finite time non-collapsed singularities. Recall that in this case the limiting class
$[\alpha]=[\omega_0]-2\pi Tc_1(X)$ is nef and big (i.e. $\int_X\alpha^n>0$), and that singularities form precisely along the proper analytic subvariety $\Null(\alpha)\subset X$, by Theorem \ref{good}.

\begin{conjecture}
Let $X$ be a compact K\"ahler manifold and $[\alpha]$ a nef and big $(1,1)$ class which is not K\"ahler and such that $[\alpha]+\lambda c_1(X)$ is a K\"ahler class for some $\lambda>0$.
Then every irreducible component of $\Null(\alpha)$ is uniruled.
\end{conjecture}

If $X$ is projective and $[\alpha]\in NS^1(X,\mathbb{R})$ this follows from the base-point-free theorem \cite{KMM} together with \cite[Theorem 2]{Ka}. This conjecture is not hard to prove when $n=2$, see \cite[3.8.3]{SWL}.

An even stronger statement, which is true in the projective case, is this:
\begin{conjecture}
Let $X$ be a compact K\"ahler manifold and $[\alpha]$ a nef and big $(1,1)$ class which is not K\"ahler and such that $[\alpha]+\lambda c_1(X)$ is a K\"ahler class for some $\lambda>0$.
Then there is a bimeromorphic morphism $\pi:X\to Y$ onto a normal K\"ahler space $Y$ such that $\mathrm{Exc}(\pi)=\Null(\alpha)$ and $[\alpha]=\pi^*[\omega_Y]$ for some K\"ahler class $[\omega_Y]$ on $Y$.
\end{conjecture}

If this is the case, then $\pi^*\omega_Y$ is a smooth nonnegative representative of $[\alpha]$. This conjecture is easy to prove when $n=2$ (see again \cite[3.8.3]{SWL}), and when $n=3$ the recent results in \cite{HP} show that this holds in many cases, but it seems that more work is needed to prove this in general when $n=3$.

In general the singularities of $Y$ may be very bad, and it may not be possible to define a solution of the K\"ahler-Ricci flow on $Y$, even in a weak sense. In this case it is expected (see \cite{SoT3, So2, LT}) that there is another normal K\"ahler space $Y'$ bimeromorphic to $X$, with K\"ahler metric $\omega_{Y'}$ and with reasonable singularities, such that the metric completion of $(X\backslash\Null(\alpha),\omega_T)$ (the smooth limit provided by Theorem \ref{fint}) is isometric to the metric completion of $(Y'_{reg}, \omega_{Y'})$, and so that the K\"ahler-Ricci flow can be defined starting at $\omega_{Y'}$ (in a weak sense, cf. \cite{EGZ3, SoT3}), and that the whole process is continuous in the Gromov-Hausdorff sense. The only case when this has been established is when $n=2$, by Song-Weinkove \cite{SW, SW2}.

\subsection{Expected behavior at collapsed finite time singularities}
Lastly, we discuss what is expected to hold in the case of finite time collapsed singularities. In this case the limiting class
$[\alpha]=[\omega_0]-2\pi Tc_1(X)$ is nef but not big, i.e. $\int_X\alpha^n=0$, and we know that singularities form everywhere on $X$, by Theorem \ref{good}.

We will say that the manifold $X$ admits a Fano fibration if there is a surjective holomorphic map $f:X\to Y$ with connected fibers, where $Y$ is a compact normal K\"ahler space (the reader may wish to assume that $Y$ is a compact K\"ahler manifold) with $0\leq \dim Y<\dim X$ and such that for every fiber $F$ of $f$ we have that $-K_X|_F$ is ample. In this case the generic fiber of $f$ is a Fano manifold of dimension $\dim X-\dim Y$, but there may be some singular fibers. The simplest example of a Fano fibration is when $Y$ is a point, and $X$ is a Fano manifold. Other simple examples are obtained by taking $X=F\times Y$ where $F$ is a Fano manifold and $Y$ is any compact K\"ahler manifold.

\begin{conjecture}[\cite{TZ2}]\label{fincol}
Let $X^n$  be a compact K\"ahler manifold. Then there exists a K\"ahler metric $\omega_0$ such that the K\"ahler-Ricci flow \eqref{krf} develops a finite time collapsed singularity if and only if $X$ admits a Fano fibration $f:X\to Y$. In this case, we can write
\begin{equation}\label{fact}
[\omega_0]=\lambda c_1(X)+f^*[\omega_Y],
\end{equation}
for some K\"ahler metric $\omega_Y$ on $Y$ and some $\lambda>0$.
\end{conjecture}

The ``if'' direction is elementary, thanks to \eqref{fact}. Indeed, the evolving class along the flow is
\[\begin{split}
[\omega(t)]&=[\omega_0]-2\pi tc_1(X)=f^*[\omega_Y]+(\lambda-2\pi t)c_1(X)\\
&=\left(1-\frac{2\pi t}{\lambda}\right)[\omega_0]+\frac{2\pi t}{\lambda}f^*[\omega_Y].
\end{split}\]
For $0\leq t<\frac{\lambda}{2\pi}$ this is a sum of a K\"ahler class and a nef class, and so it is K\"ahler, while for $t=\frac{\lambda}{2\pi}$ this equals $f^*[\omega_Y]$ which is nef but not K\"ahler. It follow from Theorem \ref{tz} that the maximal existence time is $T=\frac{\lambda}{2\pi}<\infty$ and the limiting class is $f^*[\omega_Y]$. Since we have  $\int_X(f^*\omega_Y)^n=0$, it follows that the flow is collapsed at time $T$.

The ``only if'' direction is known if $X$ is projective and $[\omega_0]\in NS^1(X,\mathbb{R})$, thanks to the base-point-free theorem and the rationality theorem \cite{KMM}. It is also known when $n\leq 3$ thanks to \cite{TZ2} (which uses as a key ingredient \cite{HP2}).

Assuming Conjecture \ref{fincol}, it is then expected that the solution $\omega(t)$ of the K\"ahler-Ricci flow \eqref{krf} will converge to $f^*\omega_Y$ as $t\to T$, in a suitable sense, for some K\"ahler metric $\omega_Y$ on $Y$. This is proved in \cite{SSW} when $f:X\to Y$ is a submersion, with fibers projective spaces, but the convergence is rather weak. The difficulty in attacking this problem is that in general $\omega_Y$ will not be a ``canonical'' metric on $Y$ (e.g. K\"ahler-Einstein).

Lastly we mention a related conjecture, raised by Tian \cite[Conjecture 4.4]{Ti} (see also \cite{So}).

\begin{conjecture}\label{finext}
Let $(X^n,\omega_0)$ be a compact K\"ahler manifold, let $\omega(t)$ be the solution of the K\"ahler-Ricci \eqref{krf}, defined
on the maximal time interval $[0,T)$ with $T<\infty$. Then as $t\to 0$ we have
\begin{equation}\label{point}
\mathrm{diam}(X,\omega(t))\to 0,
\end{equation}
if and only if
\begin{equation}\label{fano}
[\omega_0]=\lambda c_1(X),
\end{equation}
for some $\lambda>0$.
\end{conjecture}
Condition \eqref{point} is equivalent to assuming that $(X,\omega(t))$ converges to a point in the Gromov-Hausdorff topology, and is called ``finite time extinction''.
Conjecture \ref{finext} predicts that finite time extinction happens if and only if the manifold is Fano and the initial class is a positive multiple of the first Chern class. The ``if'' direction follows from work of Perelman (see \cite{ST}), who proved the stronger result that $\diam(X,\omega(t))\leq C(T-t)^{\frac{1}{2}}$, assuming \eqref{fano}.
If $[\omega_0]\in H^2(X,\mathbb{Q})$ (so $X$ is projective), then Conjecture \ref{finext} was proved by Song \cite{So}, and when $n\leq 3$ it was proved in \cite{TZ2}.

Note that if \eqref{point} holds then the flow exhibits finite time collapsing at time $T$. Indeed, if this was not the case then the limiting class $[\alpha]$ would be nef with $\int_X\alpha^n>0$, and so Theorem \ref{fint} shows that on the open set $X\backslash \mathrm{Null}(\alpha)$ we have smooth convergence of $\omega(t)$ to a
limiting K\"ahler metric $\omega_T$, and so the diameter of $(X,\omega(t))$ cannot go to zero. In fact, it is proved in \cite{TZ2} that in general Conjecture \ref{fincol} implies Conjecture \ref{finext}.

\section{Long time behavior}\label{sectinf}
\subsection{K\"ahler-Ricci flows with long time existence}
Let $(X,\omega_0)$ be a compact K\"ahler manifold and let $\omega(t)$ be the solution of the K\"ahler-Ricci flow \eqref{krf} starting at $\omega_0$, defined on the maximal time interval $[0,T)$. As we saw in Corollary \ref{nefff}, we have $T=\infty$ if and only if $-c_1(X)$ is a nef class (i.e. $-c_1(X)\in\ov{\mathcal{C}_X}$). Since $c_1(K_X)=-c_1(X)$, in this case we also say that the canonical bundle $K_X$ is nef, or that $X$ is a (smooth) minimal model. In this section we will always assume that this is the case.

The goal of this section is to analyze the behavior of the flow as $t\to\infty$, and more specifically to investigate the convergence properties of the metrics $\omega(t)$, or of the rescaled metrics $\frac{\omega(t)}{t}$, as $t\to\infty$.

Chronologically, the first result along these lines is the following.

\begin{theorem}[Cao \cite{Ca}]\label{ena}
If $c_1(X)=0$ in $H^2(X,\mathbb{R})$ then as $t\to\infty$ the metrics $\omega(t)$ converge smoothly to the unique Ricci-flat K\"ahler metric $\omega_\infty$ in the class $[\omega_0]$.
\end{theorem}

For a detailed exposition of the proof of this result, see for example \cite[Theorem 3.4.4]{SWL}. In fact the convergence is exponentially fast in all $C^k$ norms, see e.g. \cite[Proof of Theorem 1.5]{TZ} and \cite{PS}.
Next, we have:
\begin{theorem}[Cao \cite{Ca}, Tsuji \cite{Ts}]\label{dva}
If $-c_1(X)\in\mathcal{C}_X$ then as $t\to\infty$ the rescaled metrics $\frac{\omega(t)}{t}$ converge smoothly to the unique K\"ahler-Einstein metric $\omega_\infty$ on $X$ which satisfies $\Ric(\omega_\infty)=-\omega_\infty$.
\end{theorem}

More generally, we have:

\begin{theorem}[Tsuji \cite{Ts}, Tian-Zhang \cite{TiZ}]\label{tri}
If $-c_1(X)\in\ov{\mathcal{C}_X}$ and $\int_X(-c_1(X))^n>0$, then there exists a K\"ahler-Einstein metric $\omega_\infty$ on $X\backslash\Null(-c_1(X))$ which satisfies $\Ric(\omega_\infty)=-\omega_\infty$, such that for any initial K\"ahler metric $\omega_0$, the rescaled metrics $\frac{\omega(t)}{t}$ converge smoothly on compact subsets of $X\backslash\Null(-c_1(X))$ to $\omega_\infty$ as $t\to\infty$.
\end{theorem}

Further properties, which we will not discuss, were established in \cite{TiZ, Zh3, SW2, GSW, TiZ2}.

We now give the proof of Theorem \ref{tri}, which will also give as a special case Theorem \ref{dva}, where we have that $\Null(-c_1(X))=\emptyset$. The uniqueness statement in Theorem \ref{dva} is stronger than the one in Theorem \ref{tri}, but its proof is much easier, and is left as an exercise.
\begin{proof}
Since the convergence is for the rescaled metrics $\frac{\omega(t)}{t}$, it is convenient to renormalize the flow as follows:
\begin{equation}\label{krf2a}
\left\{
                \begin{aligned}
                  &\frac{\de}{\de t}\omega(t)=-\Ric(\omega(t))-\omega(t)\\
                  &\omega(0)=\omega_0
                \end{aligned}
              \right.
\end{equation}
Note that if $\ti{\omega}(s)$ solves \eqref{krf} then $\omega(t)=\frac{\ti{\omega}(s)}{1+s}$ solves \eqref{krf2a} with the new time parameter $t=\log(1+s)$, and conversely
if $\omega(t)$ solves \eqref{krf2a} then $\ti{\omega}(s)=e^t\omega(t)$ solves \eqref{krf} with the new time parameter $s=e^t-1.$ It follows that \eqref{krf2a} is also solvable on $[0,\infty),$ and that the goal is now to show that the solution $\omega(t)$ of \eqref{krf2a} satisfies
\begin{equation}\label{conv1p}
\omega(t)\to \omega_\infty,
\end{equation}
in $C^\infty_{\mathrm{loc}}(X\backslash \Null(-c_1(X)))$ as $t\to\infty$, and that the limit $\omega_\infty$ is K\"ahler-Einstein and independent of the initial metric $\omega_0$.

The cohomology class of the solution $\omega(t)$ of \eqref{krf2a} is
$$[\omega(t)]=e^{-t}[\omega_0]-(1-e^{-t})2\pi c_1(X).$$
Fix now any closed real $(1,1)$ form $\eta$ cohomologous to $-2\pi c_1(X)$, a smooth positive volume form $\Omega$ with $\Ric(\Omega)=-\eta$, and let
$$\hat{\omega}_t=e^{-t}\omega_0+(1-e^{-t})\eta.$$
These are reference forms (not necessarily positive) cohomologous to $\omega(t)$.
We claim that \eqref{krf2a} is equivalent to
\begin{equation}\label{ma2a}
\left\{
                \begin{aligned}
                  &\frac{\de}{\de t}\vp(t)=\log\frac{(\hat{\omega}_t+\ddbar\vp(t))^n}{\Omega}-\vp(t)\\
                  &\vp(0)=0\\
                  &\hat{\omega}_t+\ddbar\vp(t)>0.
                \end{aligned}
              \right.
\end{equation}
Indeed, if $\vp(t)$ solves \eqref{ma2a} and we define $\omega(t)=\hat{\omega}_t+\ddbar\vp(t)$, then
\[\begin{split}
\frac{\de}{\de t}\omega(t)&=\frac{\de}{\de t}(\hat{\omega}_t+\ddbar\vp(t))\\
&=-\hat{\omega}_t+\eta-\Ric(\omega(t))+\Ric(\Omega)-\ddbar\vp(t)\\
&=-\Ric(\omega(t))-\omega(t),
\end{split}\]
and \eqref{krf2a} holds. Conversely, if $\omega(t)$ solves \eqref{krf2a}, we define $\vp(t)$ by solving the ODE
$$\frac{\de}{\de t}\vp(t)=\log\frac{\omega(t)^n}{\Omega}-\vp(t),\quad \vp(0)=0,$$
and compute
\[\begin{split}
\frac{\de}{\de t}(e^t(\omega(t)-\hat{\omega}_t-\ddbar\vp(t)))=e^t(-\Ric(\omega(t))+\Ric(\omega(t)))=0,
\end{split}\]
and since $(e^t(\omega(t)-\hat{\omega}_t-\ddbar\vp(t)))|_{t=0}=0$, we conclude that $\omega(t)=\hat{\omega}_t+\ddbar\vp(t)$ for all $t$, and \eqref{ma2a} holds.

We now apply Theorem \ref{ct} and obtain an upper semicontinuous $L^1$ function $\psi:X\to\mathbb{R}\cup\{-\infty\},$ with $\sup_X\psi=0$, which equals $-\infty$ on $\Null(-c_1(X))$, which is finite and smooth on $X\backslash \Null(-c_1(X))$, and such that
$$\eta+\ddbar\psi\geq \ve\omega_0,$$
on $X\backslash \Null(-c_1(X))$, for some $\ve>0$.

We remark that in fact in this case (since $[\eta]=2\pi c_1(K_X)$) the result of Theorem \ref{ct} was already known before, thanks to \cite{Na, Ts} (this is known as ``Tsuji's trick''). Also, in the setting of Theorem \ref{dva}, since $-c_1(X)\in\mathcal{C}_X$, we can choose $\eta$ to be a K\"ahler form, and $\psi$ identically equal to $0$, and in this case the forms $\hat{\omega}_t$ are all K\"ahler.

First, we show that
\begin{equation}\label{tt1}
\vp(t)\leq C,
\end{equation}
on $X\times [0,\infty)$. This is a simple consequence of the maximum principle since at any maximum point of $\vp$ (for $t>0$) we have
$$0\leq \frac{\de}{\de t}\vp=\log\frac{(\hat{\omega}_t+\ddbar\vp(t))^n}{\Omega}-\vp(t)\leq
\log\frac{\hat{\omega}_t^n}{\Omega}-\vp(t)\leq C-\vp(t),$$
using that at a maximum point $\hat{\omega}_t\geq \hat{\omega}_t+\ddbar\vp(t)>0$,
and we are done. Next, we show that
\begin{equation}\label{tt}
\dot{\vp}(t)\leq C(1+t)e^{-t},
\end{equation}
on $X\times [0,\infty)$. Indeed we compute
$$\left(\frac{\de}{\de t}-\Delta\right)\vp(t)=\dot{\vp}(t)-n+\tr{\omega(t)}{\hat{\omega}_t},$$
$$\left(\frac{\de}{\de t}-\Delta\right)\dot{\vp}(t)=-\dot{\vp}(t)-e^{-t}\tr{\omega(t)}{\omega_0}+e^{-t}\tr{\omega(t)}{\eta},$$
$$\left(\frac{\de}{\de t}-\Delta\right)((e^t-1)\dot{\vp}(t)-\vp(t)-nt)=-\tr{\omega(t)}{\omega_0}<0,$$
and so the maximum principle gives
$$(e^t-1)\dot{\vp}(t)-\vp(t)-nt\leq 0,$$
which together with \eqref{tt1} gives \eqref{tt} for $t\geq 1$ (and it is clear that \eqref{tt} holds for $0\leq t\leq 1$).

Next, we show that there is a constant $C>0$ such that
\begin{equation}\label{tt2}
\vp(t)+\dot{\vp}(t)\geq \psi-C,
\end{equation}
on $X\times [0,\infty)$.
Consider the quantity
$$Q=\vp(t)+\dot{\vp}(t)-\psi.$$
The function $Q$ is lower semicontinuous (hence bounded below) and it approaches $+\infty$ as we approach $\Null(-c_1(X))$, and so it achieves a minimum at $(x,t)$, for some $t>0$ and $x\not\in\Null(-c_1(X))$, and at this point we have
\[\begin{split}
0&\geq\left(\frac{\de}{\de t}-\Delta\right)Q=\tr{\omega(t)}{(\eta+\ddbar\psi)}-n\geq \ve\tr{\omega(t)}{\omega_0}-n\\
&\geq n\ve\left(\frac{\omega_0^n}{\omega(t)^n}\right)^{\frac{1}{n}}-n\geq C^{-1}e^{-\frac{\vp(t)+\dot{\vp}(t)}{n}}-n,
\end{split}\]
and so $\vp(t)+\dot{\vp}(t)\geq -C$, which implies that $Q\geq -C$ since $\psi\leq 0$, and this shows \eqref{tt2}.

Next we show that
\begin{equation}\label{tt3}
\tr{\omega_0}{\omega(t)}\leq Ce^{-C\psi},
\end{equation}
on $X\times [0,\infty)$. For this, we compute using \eqref{gettt2}
\[\begin{split}
\left(\frac{\de}{\de t}-\Delta\right)&(\log \tr{\omega_0}{\omega(t)}-A(\vp(t)+\dot{\vp}(t)-\psi))\\
&\leq C\tr{\omega(t)}{\omega_0}+An-A\tr{\omega(t)}{(\eta+\ddbar\psi)}\\
&\leq -\tr{\omega(t)}{\omega_0}+C,
\end{split}\]
on $X\backslash \Null(-c_1(X)),$ provided we choose $A>0$ large enough. Therefore at a maximum of this quantity (achieved at $(x,t)$ with $t>0$, and necessarily with $x\not\in\Null(-c_1(X))$), we have
$$\tr{\omega(t)}{\omega_0}\leq C.$$
We now use the elementary inequality
$$\tr{\omega_0}{\omega(t)}\leq \frac{(\tr{\omega(t)}{\omega_0})^{n-1}}{(n-1)!}\cdot \frac{\omega(t)^n}{\omega_0^n},$$
which can be proved by choosing coordinates so that at a point $\omega_0$ is the identity and $\omega(t)$ is diagonal with eigenvalues $\lambda_j>0$, so that it reduces to
$$\sum_j \lambda_j \leq \frac{1}{(n-1)!}\left(\sum_j \frac{1}{\lambda_j}\right)\left(\prod_k\lambda_k\right),$$
which is obvious since each term on the LHS appears in the RHS, and all other terms on the RHS are positive. We conclude that at our point of maximum we have
$$\tr{\omega_0}{\omega(t)}\leq C\frac{\omega(t)^n}{\omega_0^n}=Ce^{\vp(t)+\dot{\vp}(t)}\frac{\Omega}{\omega_0^n}\leq C,$$
using \eqref{tt1} and \eqref{tt}. Combining this with \eqref{tt2} it follows that
$$\log \tr{\omega_0}{\omega(t)}-A(\vp(t)+\dot{\vp}(t)-\psi)\leq C,$$
at the maximum and hence everywhere, and this (using \eqref{tt1}, \eqref{tt} again) implies \eqref{tt3}.
But note that
$$\frac{\omega(t)^n}{\omega_0^n}\geq C^{-1}e^{\vp(t)+\dot{\vp}(t)}\geq C^{-1}e^\psi,$$
using \eqref{tt2}, and so given any compact subset $K\subset X\backslash \Null(-c_1(X))$ there is a constant $C_K$ such that
\begin{equation}\label{tt4}
C_K^{-1}\omega_0\leq \omega(t)\leq C_K\omega_0,
\end{equation}
holds on $K\times[0,\infty)$. The higher order estimates in Theorem \ref{higher} give that
$$\|\omega(t)\|_{C^k(K,g_0)}\leq C_{K,k},$$
for all $t\geq 0, k\geq 0$, up to shrinking $K$ slightly.
These estimates in turn imply that the function
$$\Delta_{\omega_0} \vp(t)=\tr{\omega_0}{\omega(t)}-\tr{\omega_0}{\hat{\omega}_t},$$
is uniformly bounded in $C^k(K,\omega_0)$ for all $k\geq 0$. But \eqref{tt1}, \eqref{tt} and \eqref{tt2} imply that $\vp(t)$ is uniformly bounded on $K$ (by a constant that depends on $K$ but is independent of $t$) and elliptic estimates (as in the proof of Theorem \ref{estimates}) give
\begin{equation}\label{esti1}
\|\vp(t)\|_{C^k(K,\omega_0)}\leq C_{K,k},
\end{equation}
for all $t\geq 0, k\geq 0$, up to shrinking $K$ again.
Now for $t\geq 1$, \eqref{tt} gives
$$\dot{\vp}(t)\leq Cte^{-t},$$
and so
$$\frac{\de}{\de t}(\vp(t)+Ce^{-t}(1+t))=\dot{\vp}(t)-Cte^{-t}\leq 0.$$
The function $\vp(t)+Ce^{-t}(1+t)$ is thus nonincreasing and uniformly bounded from below on compact subsets of $X\backslash \Null(-c_1(X))$, and so as
$t\to\infty$ the functions $\vp(t)$ converge pointwise on $X\backslash \Null(-c_1(X))$ to a function $\vp_\infty$, which thanks to \eqref{esti1} is smooth and in fact
$\vp(t)\to\vp_\infty$ in $C^\infty_{\rm loc}(X\backslash \Null(-c_1(X)))$.
Also \eqref{tt4} shows that $\omega_\infty:=\eta+\ddbar\vp_\infty$ is a smooth K\"ahler metric
on $X\backslash \Null(-c_1(X))$. The flow equation \eqref{ma2a} implies that $\dot{\vp}(t)$ also converges smoothly to some limit function.
Now, since $\vp(t)$ converge smoothly to $\vp_\infty$ on compact subsets of $X\backslash \Null(-c_1(X))$ it follows that given any $x\in X\backslash \Null(-c_1(X))$ there is a sequence $t_i\to\infty$ such that $\dot{\vp}(x,t_i)\to 0$. But since $\dot{\vp}(t)$ converges smoothly on compact sets to some limit function, it follows that
$\dot{\vp}(t)\to 0$ in $C^\infty_{\rm loc}(X\backslash \Null(-c_1(X)))$. Taking then the limit as $t\to\infty$ in \eqref{ma2a} we obtain
$$0=\log\frac{\omega_\infty^n}{\Omega}-\vp_\infty,$$
on $X\backslash \Null(-c_1(X))$. Taking $\ddbar$ of this, we finally obtain
$$\Ric(\omega_\infty)=-\eta-\ddbar\vp_\infty=-\omega_\infty.$$
Lastly we show that the limit $\omega_\infty$ is independent of the initial metric $\omega_\infty$, following \cite{TiZ}. The first observation is that since the functions
$e^{\vp(t)+\dot{\vp}(t)}$ are uniformly bounded (thanks to \eqref{tt1}, \eqref{tt}) and converge to $e^{\vp_\infty}$ pointwise a.e. on $X$, the dominated convergence theorem implies that
$$\lim_{t\to\infty}\int_X e^{\vp(t)+\dot{\vp}(t)}\Omega=\int_Xe^{\vp_\infty}\Omega,$$
where we extend $\vp_\infty$ by zero on $\Null(-c_1(X))$,
but at the same time
$$\lim_{t\to\infty}\int_X e^{\vp(t)+\dot{\vp}(t)}\Omega=\lim_{t\to\infty}\int_X \omega(t)^n=\int_X (-2\pi c_1(X))^n>0,$$
and so $\int_Xe^{\vp_\infty}\Omega$ is independent of the initial metric $\omega_0$.
If $\omega_0'$ is another K\"ahler metric on $X$, consider the flow \eqref{krf2a} starting at $\omega_0+\omega_0'$, which is equivalent to the parabolic complex Monge-Amp\`ere equation
\begin{equation}\label{ma2b}
\left\{
                \begin{aligned}
                  &\frac{\de}{\de t}\vp'(t)=\log\frac{(\hat{\omega}'_t+\ddbar\vp'(t))^n}{\Omega}-\vp'(t)\\
                  &\vp'(0)=0\\
                  &\hat{\omega}'_t+\ddbar\vp'(t)>0,
                \end{aligned}
              \right.
\end{equation}
where the reference forms are now
$$\hat{\omega}'_t=e^{-t}(\omega_0+\omega_0')+(1-e^{-t})\eta=\hat{\omega}_t+e^{-t}\omega_0'.$$
Therefore the difference $\vp(t)-\vp'(t)$ satisfies
\begin{equation*}
\left\{
                \begin{aligned}
                  &\frac{\de}{\de t}(\vp(t)-\vp'(t))=\log\frac{(\hat{\omega}'_t-e^{-t}\omega'_0+\ddbar\vp'(t)+\ddbar(\vp(t)-\vp'(t)))^n}{(\hat{\omega}'_t+\ddbar\vp'(t))^n}\\
                  &\quad\quad\quad\quad\quad\quad\quad\quad-(\vp(t)-\vp'(t))\\
                  &(\vp-\vp')(0)=0\\
                  &\hat{\omega}'_t+\ddbar(\vp(t)-\vp'(t))>0,
                \end{aligned}
              \right.
\end{equation*}
and at a maximum of $\vp(t)-\vp'(t)$, achieved at time $t>0$, we obtain
$$\vp(t)-\vp'(t)\leq 0,$$
and so $\vp(t)\leq \vp'(t)$ holds for all $t\geq 0$. Passing to the limit we obtain
$$\vp_\infty\leq\vp'_\infty,$$
on $X\backslash \Null(-c_1(X))$ and since, as remarked earlier,
$$\int_X e^{\vp_\infty}\Omega=\int_X e^{\vp'_\infty}\Omega,$$
this implies that $\vp_\infty=\vp'_\infty$ a.e. on $X$, and therefore everywhere on $X\backslash \Null(-c_1(X))$ where these functions are smooth.

This shows that the limits of the flow starting at $\omega_0$ and $\omega_0+\omega_0'$ are the same, and by symmetry we obtain the same statement for $\omega_0$ and $\omega_0'$.
\end{proof}

\subsection{Semiample canonical bundle}

Combining Theorems \ref{ena}, \ref{dva} and \ref{tri}, we see that the only case left to study (when $T=\infty$) is when
$-c_1(X)\in\de\mathcal{C}_X, \int_X(-c_1(X))^n=0$, and $-c_1(X)$ is not the zero class. This is the hardest case, and in general not much is known. However, a widely-believed conjecture in algebraic geometry (or rather, its direct generalization to K\"ahler manifolds), called the Abundance Conjecture, predicts that if $X$ is a compact K\"ahler manifold with $K_X$ nef, then $K_X$ is semiample. This means that there exists $\ell\geq 1$ such that for every given point $x\in X$ we can find a section $s\in H^0(X,K_X^{\otimes\ell})$ such that $s(x)\neq 0$ (i.e. $K_X^{\otimes\ell}$ is base-point free).

From now on, inspired by the Abundance Conjecture, we will make the assumption that $K_X$ is semiample (which automatically implies $K_X$ nef, see below). Then it turns out that one can say a lot about the behavior of the flow. The reason is that using sections of $K_X^{\otimes \ell}$ we may define a holomorphic map $f:X\to \mathbb{CP}^N$, where $N=\dim H^0(X,K_X^{\otimes\ell})-1$, by fixing a basis $\{s_0,\dots,s_N\}$ of $H^0(X,K_X^{\otimes\ell})$ and mapping a point $x\in X$ to the point
$[s_0(x):\dots:s_N(x)]$, which is a well-defined point in $\mathbb{CP}^N$ because these sections have empty common zero locus, by assumption. Also by definition of $f$ we have that $f^*\mathcal{O}_{\mathbb{CP}^N}(1)\cong K_X^{\otimes\ell}$.
In particular, if $\omega_{\mathrm{FS}}$ denotes the Fubini-Study metric on $\mathbb{CP}^N$, then $\frac{1}{\ell}f^*\omega_{\mathrm{FS}}$ is a smooth semipositive $(1,1)$ form which represents $-c_1(X)$. We conclude that $-c_1(X)\in\ov{\mathcal{C}_X}$, i.e. that $K_X$ is nef.

By the Proper Mapping Theorem (see \cite[p.34]{GH}), the image $f(X)$ is an irreducible analytic subvariety $Y$ of $\mathbb{CP}^N$, i.e. an irreducible algebraic variety. Provided we replace $\ell$ by a suitably high multiple of it, we have that the map $f:X\to Y$ has connected fibers, $Y$ is normal (see \cite[Theorem 2.1.27, Example 2.1.15]{Laz}) and the dimension of $Y$ equals the Kodaira dimension $\kappa(X)$ of $X$ (see \cite[Theorem 2.1.33]{Laz}).

We now split into cases depending on the Kodaira dimension $\kappa(X)$.
\subsection{The case $\kappa(X)=0$}
The first case is $\kappa(X)=0$, where we need the following well-known lemma.

\begin{lemma}\label{cycy}
Let $X$ be a compact K\"ahler manifold with $K_X$ semiample. Then the following are equivalent:
\begin{itemize}
\item[(a)] $\kappa(X)=0$
\item[(b)] $c_1(X)=0$ in $H^2(X,\mathbb{R})$
\item[(c)] There exists $k\geq 1$ such that $K_X^{\otimes k}\cong\mathcal{O}_X$ is holomorphically trivial.
\end{itemize}
\end{lemma}
In fact, without the assumption that $K_X$ be semiample, it remains true that $(b)\Leftrightarrow (c)\Rightarrow(a)$, while the implication $(a)\Rightarrow(b)$ is false. In this case, the only hard implication is $(b)\Rightarrow (c)$, and we refer the interested reader to \cite{To2}, for example.

\begin{proof}
The implication $(c)\Rightarrow (b)$ is trivial. First we show that $(a)\Rightarrow (c)$.
The assumption that $\kappa(X)=0$ is equivalent to the fact that $\dim H^0(X,K_X^{\otimes k})\leq 1$ for all $k\geq 1$, and is equal to $1$ for at least one value of $k$. Choose $k$ large enough so that $K_X^{\otimes k}$ is base-point free. Then we must have $\dim H^0(X,K_X^{\otimes k})=1$, and if $s\in H^0(X,K_X^{\otimes k})$ is a nontrivial section then necessarily $s$ is never vanishing. This means that $K_X^{\otimes k}\cong\mathcal{O}_X$ is holomorphically trivial.

Next, we show that $(b)\Rightarrow (a)$.
Fix a smooth metric $h$ on $K_X$ and a K\"ahler metric $\omega$ on $X$. The curvature $R_h$ of $h$ is a closed real $(1,1)$ form cohomologous to $c_1(K_X)=-c_1(X)=0$, so
$$\int_X \tr{\omega}{R_h}\omega^n=n\int_X\omega^{n-1}\wedge R_h=0,$$
and so we can find a smooth function $u$ such that $\Delta_{\omega}u=\tr{\omega}{R_h}$. Therefore the smooth metric $\ti{h}=e^u h$ on $K_X$ has curvature $R_{\ti{h}}=R_h-\ddbar u$ which satisfies $\tr{\omega}{R_{\ti{h}}}=0.$ Given any $k\geq 1$ and any $s\in H^0(X,K_X^{\otimes k})$, let $|s|^2$ be its pointwise length squared with respect to the metric $\ti{h}^{k}$. Then a straightforward calculation gives
$$\Delta_{\omega}|s|^2=|\nabla s|^2-k|s|^2\tr{\omega}{R_{\ti{h}}}=|\nabla s|^2\geq 0,$$
where $\nabla$ is the Chern connection of the metric $\ti{h}^k$ on $K_X^{\otimes k}$. By the strong maximum principle this implies that $|s|^2$ is constant, and so $|\nabla s|^2$ is identically zero, i.e. the section $s$ is parallel. This implies that $\dim H^0(X,K_X^{\otimes k})\leq 1$, because if we have two global sections $s_1,s_2$, given a point $x\in X$ there exists $\lambda\in\mathbb{C}$ such that $s_1(x)=\lambda s_2(x)$ (up to switching $s_1$ and $s_2$), and since they are both globally parallel we must have $s_1=\lambda s_2$ globally.
We have therefore shown that $\kappa(X)\leq 0$ (without using that $K_X$ is semiample). Since $K_X$ is semiample, we have $H^0(X,K_X^{\otimes\ell})\neq 0$ for some $\ell\geq 1$, and so $\kappa(X)=0$.
\end{proof}

So, under our assumption that $K_X$ is semiample, if $\kappa(X)=0$ then Theorem \ref{ena} applies.

\subsection{The case $\kappa(X)=n$}
The second case is when $\kappa(X)=n=\dim X$. Recall that since $K_X$ is semiample, we have a holomorphic map $f:X\to\mathbb{CP}^N$ such that $f^*\mathcal{O}_{\mathbb{CP}^N}(1)\cong K_X^{\otimes\ell}$, for some $\ell\geq 1$.

If, as before, we let $Y=f(X)$, which is an irreducible algebraic variety of dimension $n$, then the map $f:X\to Y$ has connected fibers and the generic fiber has dimension $0$, i.e. it is a bimeromorphic morphism. It follows that we have
\[\begin{split}
\int_X(-c_1(X))^n&=\int_Xc_1(K_X)^n=\ell^{-n}\int_X c_1(f^*\mathcal{O}_{\mathbb{CP}^N}(1))^n\\
&=\ell^{-n}\int_Y c_1(\mathcal{O}_{\mathbb{CP}^N}(1)|_{Y})^n>0,
\end{split}\]
since the last term is (up to a positive constant) equal to the volume of the regular part of $Y$ with respect to the restriction of $\omega_{\mathrm{FS}}$, the Fubini-Study metric on $\mathbb{CP}^N$.

Therefore, either we have $-c_1(X)\in\mathcal{C}_X$, in which case Theorem \ref{dva} applies, or otherwise we have $-c_1(X)\in\de\mathcal{C}_X$ and $\int_X(-c_1(X))^n>0$, and Theorem \ref{tri} applies.

\subsection{The case $0<\kappa(X)<n$}
The third and last case to study is thus $0<\kappa(X)<n$. Let $Y_{sing}$ be the singular locus of $Y$, which is a proper analytic subvariety of $Y$, and $Y_{reg}=Y\backslash Y_{sing}$ its regular locus, so $Y_{reg}$ is a connected complex manifold of dimension $\kappa(X)$. Also, $f^{-1}(Y_{sing})$ is a proper analytic subvariety of $X$, and so $f:X\backslash f^{-1}(Y_{sing})\to Y_{reg}$ is a surjective holomorphic map between complex manifolds, with compact connected fibers. Let $S'\subset Y$ be the union of $Y_{sing}$ together with the critical values of $f:X\backslash f^{-1}(Y_{sing})\to Y_{reg}$ (i.e. the images of all points $x\in X\backslash f^{-1}(Y_{sing})$ such that $df_x$ is not surjective). Then $S'$ is a proper analytic subvariety of $Y$, $S=f^{-1}(S')$ is a proper analytic subvariety of $X$, and $f:X\backslash S\to Y\backslash S'$ is a (surjective) holomorphic submersion between complex manifolds, and all the fibers $X_y=f^{-1}(y), y\in Y\backslash S'$ are connected compact complex manifolds of dimension equal to $n-\kappa(X)$. Informally, we will refer to $S$ as the set of singular fibers of $f$, and to the fibers $X_y=f^{-1}(y), y\in Y\backslash S'$ as the smooth fibers, although this is not strictly speaking correct.

Recall that the map $f$ has the property that $K_X^{\otimes\ell}\cong f^*\mathcal{O}_{\mathbb{CP}^N}(1)$, which implies that for every $y\in Y\backslash S'$ we have $K_X^{\otimes \ell}|_{X_y}\cong \mathcal{O}_{X_y}$. However, since $f$ is a submersion in a neighborhood of $X_y$, we have the adjunction-type relation $K_{X_y}\cong K_X|_{X_y}$ (see Lemma \ref{adj} below), and so it follows that $K_{X_y}^{\otimes \ell}\cong\mathcal{O}_{X_y}$, and so in particular $c_1(X_y)=0$ in $H^2(X_y,\mathbb{R})$. This means that the smooth fibers are Calabi-Yau manifolds, and so $X$ is a fiber space over $Y$ with generic fiber a Calabi-Yau $(n-\kappa(X))$-fold.

So we have seen that assuming that $K_X$ is semiample has provided us with a fibration structure on $X$ (and in fact, one can also view the existence of this fibration as being an equivalent statement to the Abundance Conjecture). This is a major advantage over the ``general'' case when one only assumes that $K_X$ is nef (i.e. $T=\infty$).

Our goal is the following result, which generalizes earlier work of Song-Tian \cite{SoT, SoT2} and Fong-Zhang \cite{FZ} (see also \cite{To}):

\begin{theorem}[T.-Weinkove-Yang \cite{TWY}, T.-Zhang \cite{TZ}]\label{colla}
Let $(X,\omega_0)$ be a compact K\"ahler manifold with $K_X$ semiample and $0<\kappa(X)<n$, and let $f:X\to Y$ be the fibration we just described. Let $\omega(t),t\in[0,\infty)$ be the solution of the K\"ahler-Ricci flow \eqref{krf} starting at $\omega_0$. Then as $t\to\infty$ we have
\begin{equation}\label{conve1}
\frac{\omega(t)}{t}\to f^*\omega_Y,
\end{equation}
in $C^0_{\mathrm{loc}}(X\backslash S),$ where $\omega_Y$ is a K\"ahler metric on $Y\backslash S'$ which satisfies
\begin{equation}\label{wp}
\Ric(\omega_Y)=-\omega_Y+\omega_{\mathrm{WP}},
\end{equation}
and $\omega_{\mathrm{WP}}$ is a smooth semipositive $(1,1)$ form on $Y\backslash S'$. Furthermore, for any given $y\in Y\backslash S'$ we have
\begin{equation}\label{conve2}
\omega(t)|_{X_y}\to \omega_y,
\end{equation}
in $C^\infty(X_y)$, where $\omega_y$ is the unique Ricci-flat K\"ahler metric on $X_y$ in the class $[\omega_0]|_{X_y}$.

Lastly, if $S=\emptyset$ (i.e. $Y$ is smooth and $f$ is a submersion) then $(X,\frac{\omega(t)}{t})$ converge to $(Y,\omega_Y)$ in the Gromov-Hausdorff topology, as $t\to\infty$.
\end{theorem}

The Weil-Petersson form $\omega_{\mathrm{WP}}$ measures the variation of the complex structures of the smooth Calabi-Yau fibers, and it is identically zero whenever all the fibers $X_y$ are biholomorphic to each other (see Proposition \ref{gr}).

In the setting of Theorem \ref{colla}, Song-Tian \cite{SoT, SoT2} had earlier proved that \eqref{conve1} holds in the weak topology of currents, in the $C^0_{\rm loc}$ topology of K\"ahler potentials, and when $n=2$ also in the $C^{1,\alpha}_{\rm loc}$ topology of K\"ahler potentials (for $0<\alpha<1$). This was then extended to all $n$ in \cite{FZ} (cf. \cite{To}), but note that this convergence falls short of the one obtained in Theorem \ref{colla}. As we will see in Theorem \ref{smooth}, if the smooth fibers $X_y$ are tori (or finite quotients of tori) then in fact \eqref{conve1} holds in the $C^\infty_{\mathrm{loc}}(X\backslash S)$ topology thanks to \cite{FZ,Gi,GTZ,HT,TZ}, and this is expected to hold in general.

We also mention that in the setting of Theorem \ref{colla}, it was proved in \cite{SoT4} that the scalar curvature of $\frac{\omega(t)}{t}$ remains uniformly bounded.
It is also conjectured that in this same setting, assuming $S\not =\emptyset,$ then $(X,\frac{\omega(t)}{t})$ converge to the metric completion of $(Y\backslash S',\omega_Y)$ in the Gromov-Hausdorff topology, as $t\to\infty$. This is completely open even in the simplest case when $n=2,\kappa(X)=1$, and in fact we do not even know whether these metrics
have uniformly bounded diameter, see Section \ref{sectopen}.

\subsection{General facts about holomorphic submersions}
Before we begin the proof of Theorem \ref{colla} we need to discuss a few results about holomorphic submersions.  For simplicity of notation we will write $m=\kappa(X)$. To avoid excessive technicalities, we will assume that $S$ is empty, or in other words that $Y$ is a smooth projective manifold and the map $f:X\to Y$ is a submersion. In the general case one argues along the same lines, but with the extra complication of having to introduce a suitably chosen cutoff function in essentially all the estimates (see \cite{TWY} for details). The only estimates which are substantially harder to obtain are the uniform $C^0$ bounds for $\vp$ and $\dot{\vp}$ (which in general are weaker than those obtained in Lemma \ref{c0}). Also, by assuming that $S=\emptyset,$ we will in fact be able to conclude that the convergence in \eqref{conve1} and \eqref{conve2} is exponentially fast.

Note that the fibers $X_y$ (which are now all smooth) are all diffeomorphic to each other (by Ehresmann's Theorem \cite[Theorem 2.4]{Ko}, which implies that $f$ is a smooth fiber bundle), but in general are different as complex manifolds, so the term $\omega_{\mathrm{WP}}$ will not be zero in general. In other words, $f$ is in general not a holomorphic fiber bundle (by the Fischer-Grauert theorem \cite{FG}, $f$ is a holomorphic fiber bundle if and only if all fibers $X_y$ are biholomorphic to each other). However, if $\dim X_y=1$, so that the fibers are elliptic curves, then necessarily $f$ is a holomorphic fiber bundle, since elliptic curves are classified by their $j$-invariant, which in our case defines a holomorphic map $j:Y\to\mathbb{C}$ which must be constant since $Y$ is compact.

A useful fact, which we will use extensively, is that on the total space of a holomorphic submersion we can always find local holomorphic product coordinates.

\begin{lemma}\label{adj}
Let $f:X^n\to Y^m$ be a holomorphic submersion between complex manifolds. Then given any point $x\in X$ we can find an open set $U\ni x$ and local holomorphic coordinates $(z_1,\dots,z_n)$ on $U$ and $(y_1,\dots,y_m)$ on $f(U)$ such that in these coordinates the map $f$ is given by $(z_1,\dots,z_n)\mapsto (z_1,\dots,z_m)$. If $f$ is proper with connected fibers, then the canonical bundle of every fiber $X_y=f^{-1}(y)$ satisfies $K_{X_y}\cong K_X|_{X_y}$.
\end{lemma}
\begin{proof}
The existence of local holomorphic product coordinates is a simple consequence of the implicit function theorem for holomorphic maps (see e.g. \cite[p.60]{Ko}).

If $f$ is proper with connected fibers $X_y=f^{-1}(y)$, then the adjunction formula (\cite[Proposition 2.2.17]{Hu}) gives
$$K_{X_y}\cong K_X|_{X_y}\otimes \det(N_{X_y/X}),$$
where $N_{X_y/X}$ is the normal bundle of $X_y$ inside $X$. However this normal bundle is trivial, because its dual is globally trivialized  by
$$f^*(dy_1\wedge\dots\wedge dy_m),$$
where $(y_1,\dots,y_m)$ are local holomorphic coordinates on $Y$ near $y$.
\end{proof}

In particular, using these local coordinates, we can view $(z_1,\dots,z_m)$ as ``base directions'' and $(z_{m+1},\dots,z_n)$ as ``fiber directions'', a fact that we will use very often.

First, we define the Weil-Petersson form $\omega_{\mathrm{WP}}$. Recall that by construction of the map $f$ we have $K_{X_y}^{\otimes \ell}\cong\mathcal{O}_{X_y}$. Let $\Psi$ be a local nonvanishing holomorphic section of $f_*(K_{X/Y}^{\otimes\ell})$, i.e. a family $\Psi_y\in H^0(X_y, K_{X_y}^{\otimes \ell}),$ for $y$ in some small open set $U$ in $Y$, such that each $\Psi_y$ is never vanishing on $X_y$, and the forms $\Psi_y$ vary holomorphically in $y$. Here $K_{X/Y}=K_X\otimes(f^*K_Y)^*$ denotes the relative canonical bundle. On $U$ we then define
$$\omega_{\mathrm{WP}}=-\ddbar \log\left( (\sqrt{-1})^{(n-m)^2}\int_{X_y}(\Psi_y\wedge\ov{\Psi_y})^{\frac{1}{\ell}}\right),$$
where $(\sqrt{-1})^{(n-m)^2}(\Psi_y\wedge\ov{\Psi_y})^{\frac{1}{\ell}}$ is a smooth positive volume form on $X_y$, defined as follows:
in local holomorphic coordinates $(z_1,\dots,z_{n-m})$ on $X_y$ we can write
$$\Psi_y=F(y,z) (dz_1\wedge\dots\wedge dz_{n-m})^{\otimes \ell},$$
where $F$ is a nonvanishing holomorphic function, and we have
$$\Psi_y\wedge\ov{\Psi_y}=|F(y,z)|^2 (dz_1\wedge \dots\wedge dz_{n-m}\wedge d\ov{z}_1\wedge\dots\wedge d\ov{z}_{n-m})^{\otimes \ell},$$
$$(\Psi_y\wedge\ov{\Psi_y})^{\frac{1}{\ell}}=|F(y,z)|^{\frac{2}{\ell}} dz_1\wedge \dots\wedge dz_{n-m}\wedge d\ov{z}_1\wedge\dots\wedge d\ov{z}_{n-m},$$
$$(\sqrt{-1})^{(n-m)^2}(\Psi_y\wedge\ov{\Psi_y})^{\frac{1}{\ell}}=|F(y,z)|^\frac{2}{\ell} (\mn)^{n-m}dz_1\wedge d\ov{z}_1\wedge\dots\wedge dz_{n-m}\wedge d\ov{z}_{n-m},$$
and this is well-defined independent of the choice of coordinates.
Note also that the volume form $(\sqrt{-1})^{(n-m)^2}(\Psi_y\wedge\ov{\Psi_y})^{\frac{1}{\ell}}$ on $X_y$ is Ricci-flat, in the sense that
\begin{equation}\label{rff}
\ddbar\log\left((\sqrt{-1})^{(n-m)^2}(\Psi_y\wedge\ov{\Psi_y})^{\frac{1}{\ell}}\right)=\frac{1}{\ell}\ddbar\log|F|^2=0,
\end{equation}
because $F$ is a never-vanishing holomorphic function.
Furthermore, the Weil-Petersson form $\omega_{\mathrm{WP}}$ is well-defined globally on $Y$, because if
$\ti{\Psi}$ is another local nonvanishing holomorphic section of $f_*(K_{X/Y}^{\otimes\ell})$, defined on an open set $\ti{U}\subset Y$, then for all $y\in U\cap \ti{U}$ (assuming this is nonempty) we have that
$\Psi_y$ and $\ti{\Psi}_y$ are both nonvanishing sections of the trivial bundle $K_{X_y}^{\otimes \ell}$, and so there is a nonzero constant $h_y$ such that $\ti{\Psi}_y=h_y \Psi_y$ on $X_y$. Since $\Psi_y$ and $\ti{\Psi}_y$ vary holomorphically in $y$, then so does $h_y$, i.e. it defines a local nonvanishing holomorphic function $h$ on $U\cap\ti{U}$.
But we have $\ddbar\log|h|^2=0,$ and so $\omega_{\mathrm{WP}}$ is well-defined globally on $Y$. Also, we may take $\ell$ to be the smallest positive integer such that $K_{X_y}^{\otimes \ell}\cong\mathcal{O}_{X_y}$ holds (since if we use multiples of this $\ell$, we obtain the same Weil-Petersson form).

Although we will not use the following proposition, it is a useful fact to keep in mind.

\begin{proposition}\label{gr}
Let $f:X\to Y$ be a holomorphic submersion between compact K\"ahler manifolds, with connected fibers, such that $K_X^{\otimes \ell}\cong f^*L$ for some $\ell\geq 1$, where $L\to Y$ is a holomorphic line bundle.
Then the Weil-Petersson form $\omega_{\mathrm{WP}}$ on $Y$ is semipositive definite, and identically equal to zero if and only if $f$ is a holomorphic fiber bundle.
\end{proposition}

\begin{proof}
The statements we need to prove are local on $Y$, so we may assume that $Y$ is a ball in $\mathbb{C}^m$, where $L$ is trivial and so $K_X^{\otimes \ell}$ is also trivial. We may also assume that $\ell$ is the smallest positive integer such that this holds. We can then find an $\ell$-fold unramified connected covering $\tau:\ti{X}\to X$ such that $K_{\ti{X}}=\tau^*K_X^{\otimes\ell}$ is trivial, see e.g. \cite[Lemma 4.6]{Fu} (connectedness follows from the fact that we took $\ell$ minimal). Then
composing the map $\tau$ with $f$ we obtain a holomorphic submersion  $\ti{f}:\ti{X}\to Y$. Its Stein factorization is $\ti{X}\overset{p}{\to} \ti{Y}\overset{q}{\to} Y$ where $\ti{Y}$ is a connected complex manifold (since $\ti{X}$ is connected), $p$ is a holomorphic submersion with connected fibers, and $q$ is a finite unramified covering of $Y$ (see e.g. \cite[Lemma 2.4]{FG2}). Since $Y$ is a ball and $\ti{Y}$ is connected, we conclude that $q$ is a biholomorphism, and so we may assume that $\ti{f}$ has connected fibers $\ti{X}_y$ which satisfy
$K_{\ti{X}_y}\cong\mathcal{O}_{\ti{X}_y}$. The maps $\ti{X}_y\to X_y$ are also $\ell$-fold unramified coverings, and the Weil-Petersson form for $\ti{f}$ equals the one for $f$.
Furthermore, $\ti{f}$ is a holomorphic fiber bundle if and only if $f$ is (\cite[Lemma 4.5]{Fu}).

Therefore we may assume that $K_{X_y}\cong\mathcal{O}_{X_y}$.
For every $y\in Y$ there is a Kodaira-Spencer linear map $\rho_y:T_yY\to H^1(X_y,T^{1,0}X_y)$ (see \cite{Ko}), and the Weil-Petersson form at $y$ is equal to the pullback under $\rho_y$ of the $L^2$ inner product on
$H^1(X_y,T^{1,0}X_y)$ defined using harmonic forms with respect to the Ricci-flat metric on $X_y$ in the class $[\omega_0]|_{X_y}$, thanks to \cite[Theorem 2]{Ti2} or \cite{FS}. Therefore $\omega_{\mathrm{WP}}$ is semipositive definite (see also \cite[Lemma 1.8]{Fu} for a direct proof of this semipositivity), and identically equal to zero if and only if all the Kodaira-Spencer maps $\rho_y$ are zero. But Serre duality, together with $K_{X_y}\cong\mathcal{O}_{X_y}$, gives $H^1(X_y,T^{1,0}X_y)\cong H^{n-1}(X_y,\Omega^1_{X_y})\cong H^{1,n-1}(X_y),$ and so this vector space has dimension independent of $y$. A theorem of Kodaira-Spencer \cite[Theorem 4.6]{Ko} then implies that the Kodaira-Spencer maps are all zero if and only if $f$ is a holomorphic fiber bundle.
\end{proof}

It is instructive to see directly that if the map $f$ is a holomorphic fiber bundle then the Weil-Petersson form is identically zero. Indeed, in this case all the fibers are biholomorphic to a fixed Calabi-Yau manifold $F$ and we can find local trivializing biholomorphisms $f^{-1}(U)\to U\times F$, over all sufficiently small open sets $U\subset Y$, and using these we can then choose the forms $\Psi_y$ as above to be independent of $y\in U$, equal to the pullback to $U\times F$ of a fixed never vanishing section of $K_F^{\otimes \ell}$. This way the integrals $\int_{X_y}(\Psi_y\wedge\ov{\Psi_y})^{\frac{1}{\ell}}$ do not depend on $y$, and so $\omega_{\mathrm{WP}}=0$.

We also have the following useful fact.

\begin{proposition}\label{c1neg}
Let $f:X\to Y$ be a holomorphic submersion between compact K\"ahler manifolds, with connected fibers, such that $K_X^{\otimes \ell}\cong f^*L$ for some $\ell\geq 1$, where $L\to Y$ is an ample line bundle. Then the class
$$-2\pi c_1(Y)+[\omega_{\mathrm{WP}}],$$
is a K\"ahler class on $Y$.
\end{proposition}
\begin{proof}
The assumption that $f$ has connected fibers is equivalent to $f_*\mathcal{O}_X\cong\mathcal{O}_Y,$ and so the projection formula gives
\begin{equation}\label{proj}
f_* (K_X^{\otimes\ell})\cong(f_* (K_{X/Y}^{\otimes\ell}))\otimes K_Y^{\otimes\ell}.
\end{equation}
But the assumption $K_X^{\otimes \ell}\cong f^*L$ implies
$$\mathcal{O}_{X_y}\cong K_{X}^{\otimes\ell}|_{X_y}\cong K_{X/Y}^{\otimes\ell}|_{X_y},$$
and together with Lemma \ref{adj} we obtain $K_{X_y}^{\otimes\ell}\cong \mathcal{O}_{X_y}$.

Therefore $\dim H^0(X_y, K_{X/Y}^{\otimes\ell}|_{X_y})=1$ is independent of $y\in Y$,
and Grauert's theorem on direct images \cite[Theorem I.8.5]{bhpv} shows that
\begin{equation}\label{gra}
f_* (K_{X/Y}^{\otimes\ell})=:L',
\end{equation}
is a line bundle on $Y$ . Since all the fibers of $f$ have trivial $K_{X_y}^{\otimes\ell}$, it follows
that
\begin{equation}\label{push}
K_X^{\otimes \ell}\cong f^*f_*(K_X^{\otimes\ell})
\end{equation}
(see \cite[Theorem V.12.1]{bhpv}). Indeed, note that
$$f_*f^*f_*(K_X^{\otimes\ell})\cong f_*K_X^{\otimes \ell},$$
thanks to the projection formula. If we denote by $E=K_X^{\otimes \ell}\otimes ( f^*f_*(K_X^{\otimes\ell}))^*$ the ``error term'', then we have that
$f_*E\cong\mathcal{O}_Y$, and so $H^0(X,E)\cong H^0(Y,f_*E)=\mathbb{C}$. Let $e$ be a global trivializing section of $f_*E$, and let $s\in H^0(X,E)$ be the section
which corresponds to $e$ under this isomorphism. If $s$ vanishes at a point $x\in X$, then the restriction of $s$ to the fiber $X_{f(x)}$ is a holomorphic section
of $E|_{X_y}\cong K_{X_y}^{\otimes\ell}\cong\mathcal{O}_{X_y}$, the trivial bundle, so $s|_{X_y}$ is a holomorphic function which vanishes somewhere, and hence it is zero since the fiber $X_{f(x)}$ is compact. Therefore $e$ vanishes at the
point $f(x)$, which is absurd. This shows that $s$ is never vanishing, and so $E$ is the trivial bundle, and this proves \eqref{push}.

From \eqref{proj}, \eqref{gra} and \eqref{push} we conclude that
\begin{equation}\label{voglio}
K_X^{\otimes\ell}\cong f^*((f_* (K_{X/Y}^{\otimes\ell}))\otimes K_Y^{\otimes\ell})\cong f^*(L'\otimes K_Y^{\otimes\ell}).
\end{equation}
But we also have by assumption that $K_X^{\otimes \ell}\cong f^*L$, and so
$$f^*(L'\otimes K_Y^{\otimes\ell}\otimes L^*)\cong\mathcal{O}_X,$$
and pushing forward and using the projection formula we see that
$$L'\otimes K_Y^{\otimes\ell}\cong L,$$
which is ample, and so $c_1(L)\in\mathcal{C}_Y$.
By definition, the smooth form $\ell\omega_{\mathrm{WP}}$ is the curvature of a singular metric on $L'$, and so $[\omega_{\mathrm{WP}}]=\frac{2\pi}{\ell}c_1(L')$.
We obtain that
$$-2\pi c_1(Y)+[\omega_{\mathrm{WP}}]=2\pi c_1(K_Y)+[\omega_{\mathrm{WP}}]=\frac{2\pi}{\ell}c_1(L)\in\mathcal{C}_Y,$$
as claimed.
\end{proof}

\subsection{Reduction to a parabolic complex Monge-Amp\`ere equation}

Since in Theorem \ref{colla} the collapsing is for the rescaled metric $\frac{\omega(t)}{t}$, we again consider the normalized flow
\begin{equation}\label{krf2}
\left\{
                \begin{aligned}
                  &\frac{\de}{\de t}\omega(t)=-\Ric(\omega(t))-\omega(t)\\
                  &\omega(0)=\omega_0
                \end{aligned}
              \right.
\end{equation}
The flow \eqref{krf2} is also solvable on $[0,\infty),$ and \eqref{conve1} is equivalent to showing that the solution $\omega(t)$ of \eqref{krf2} satisfies
\begin{equation}\label{conve1p}
\omega(t)\to f^*\omega_Y,
\end{equation}
in $C^0(X)$ as $t\to\infty$, and \eqref{conve2} is equivalent to the statement that
\begin{equation}\label{conve2p}
e^t\omega(t)|_{X_y}\to \omega_y,
\end{equation}
in $C^\infty(X_y)$, where $\omega_y$ is the unique Ricci-flat K\"ahler metric on $X_y$ in the class $[\omega_0]|_{X_y}$. In fact, since we assume that $S=\emptyset$, we will be able to show that convergence in \eqref{conve1p} and \eqref{conve2p} is exponentially fast.

As usual, we would like to rewrite \eqref{krf2} as an equivalent parabolic complex Monge-Amp\`ere equation, but in order to obtain the convergence results in Theorem \ref{colla}, we have to make a very careful choice of reference metrics, and we have to first derive several preliminary results. The K\"ahler class of the evolving metric $\omega(t)$ is now
$$[\omega(t)]=e^{-t}[\omega_0]-(1-e^{-t})2\pi c_1(X).$$
Since the fibers $X_y$ are Calabi-Yau, thanks to Yau's Theorem \cite{Ya} for every $y\in Y$ there exists a unique smooth function $\rho_y$ on $X_y$ with $\int_{X_y}\rho_y \omega_0^{n-m}=0$, and such that $\omega_0|_{X_y}+\ddbar\rho_y=\omega_y$ is the unique Ricci-flat K\"ahler metric on $X_y$ in the class $[\omega_0|_{X_y}]$. Thanks to Yau's a priori estimates for $\rho_y$ in \cite{Ya}, we see that the functions $\rho_y$ depend smoothly on $y$, and so they define a global smooth function $\rho$ on $X$ (see also \cite[Lemma 2.1]{Fi}). We define
$$\omega_{\mathrm{SRF}}=\omega_0+\ddbar\rho.$$
This is a closed real $(1,1)$ form on $X$, which restricts to a Ricci-flat K\"ahler metric on all fibers $X_y$ of $f$.
It was first introduced by Greene-Shapere-Vafa-Yau \cite{GSVY} in the context of elliptically fibered $K3$ surfaces and ``stringy cosmic strings''.
For every $\eta$ K\"ahler form on $Y$, we clearly have that $f^*\eta^m\wedge \omega_{\mathrm{SRF}}^{n-m}$ is a smooth positive volume form on $X$. As a side remark, it would be interesting to know whether
$\omega_{\rm SRF}$ is semipositive definite everywhere on $X$.

The following two propositions are due to Song-Tian \cite{SoT, SoT2} (see also \cite{To}).
\begin{proposition}\label{wp1}
Given a K\"ahler form $\eta$ on $Y$, then on $X$ we have
$$\ddbar\log (f^*\eta^m\wedge \omega_{\mathrm{SRF}}^{n-m})=-f^*\Ric(\eta)+f^*\omega_{\mathrm{WP}}.$$
\end{proposition}
\begin{proof}
We choose local product coordinates as in Lemma \ref{adj}, which we call $(z_1,\dots, z_n)$ on $U\subset X$ and $(z_1,\dots,z_m)$ on $f(U)\subset Y$.
In these coordinates we write
$$\eta=\mn\sum_{i,j=1}^m\eta_{i\ov{j}}dz_i\wedge d\ov{z}_j.$$
We choose a local nonvanishing holomorphic section $\Psi_y$ of $f_*(K_{X/Y}^{\otimes\ell})$ as before, with $y\in f(U)$, and
define a smooth positive function on $f(U)$ by
$$u(y)=\frac{(\sqrt{-1})^{(n-m)^2}(\Psi_y\wedge\ov{\Psi}_y)^{\frac{1}{\ell}}}{ \omega_{\mathrm{SRF}}^{n-m}|_{X_y}}.$$
This is well-defined because both $(\sqrt{-1})^{(n-m)^2}(\Psi_y\wedge\ov{\Psi}_y)^{\frac{1}{\ell}}$ and $\omega_{\mathrm{SRF}}^{n-m}|_{X_y}$
are Ricci-flat volume forms on $X_y$ (recalling \eqref{rff}) and so their ratio is a constant on $X_y$.
Then integrating $u(y)\omega_{\mathrm{SRF}}^{n-m}|_{X_y}$ over $X_y$ we see that
$$u(y)=\frac{(\sqrt{-1})^{(n-m)^2}\int_{X_y}(\Psi_y\wedge\ov{\Psi}_y)^{\frac{1}{\ell}}}{\int_{X_y}\omega_{\mathrm{SRF}}^{n-m}},$$
and so
$$-\ddbar \log u = \omega_{\mathrm{WP}}+\ddbar \log \int_{X_y}\omega_{\mathrm{SRF}}^{n-m}.$$
But the function $y\mapsto \int_{X_y}\omega_{\mathrm{SRF}}^{n-m}$ is constant on $Y$, because it equals the pushforward $f_*\omega_{\mathrm{SRF}}^{n-m}$ and we have
\begin{equation}\label{const}
df_*\omega_{\mathrm{SRF}}^{n-m} = f_*d\omega_{\mathrm{SRF}}^{n-m}=0.
\end{equation}
Therefore
\begin{equation}\label{crux}
-\ddbar \log u = \omega_{\mathrm{WP}}.
\end{equation}
Writing as before
$$\Psi_y=F(y,z) (dz_1\wedge\dots\wedge dz_{n-m})^{\otimes \ell},$$
with $F$ holomorphic and nonzero, then we have
\begin{equation}\label{calc}
f^*\eta^m\wedge \omega_{\mathrm{SRF}}^{n-m}=f^*\eta^m\wedge(\omega_{\mathrm{SRF}}^{n-m}|_{X_y})=\frac{1}{f^*u}
(\sqrt{-1})^{(n-m)^2}f^*\eta^m\wedge(\Psi_y\wedge\ov{\Psi}_y)^{\frac{1}{\ell}},
\end{equation}
and so
\[\begin{split}
\ddbar\log(f^*\eta^m\wedge \omega_{\mathrm{SRF}}^{n-m})&=\ddbar\log\left(|F|^{\frac{2}{\ell}}\det(\eta_{i\ov{j}})\right)-f^*\ddbar\log u\\
&=-f^*\Ric(\eta)+f^*\omega_{\mathrm{WP}},
\end{split}\]
thanks to \eqref{crux}.
\end{proof}

\begin{proposition}\label{wp2}
There is a unique K\"ahler metric $\omega_Y$ on $Y$ which satisfies
\begin{equation}\label{twke}
\Ric(\omega_Y)=-\omega_Y+\omega_{\mathrm{WP}}.
\end{equation}
\end{proposition}
\begin{proof}
Thanks to Proposition \ref{c1neg} we know that
$-2\pi c_1(Y)+[\omega_{\mathrm{WP}}]\in\mathcal{C}_Y$, and so we can choose a K\"ahler metric $\eta$ in this class.
Thanks to \eqref{voglio}, we have that
$$2\pi c_1(X)=f^*(2\pi c_1(Y)-[\omega_{\mathrm{WP}}]),$$
and so we can find a smooth positive volume form $\Omega'$ on $X$ with $\Ric(\Omega')=-f^*\eta$. Consider then the smooth positive function on $X$ given by
$$G=\frac{\Omega'}{f^*\eta^m\wedge \omega_{\mathrm{SRF}}^{n-m}}.$$
We claim that $G$ is constant when restricted to every fiber $X_y$ of $f$. Indeed we can choose local product coordinates as in Lemma \ref{adj}, and write
$$\eta=\mn\sum_{i,j=1}^m\eta_{i\ov{j}}dz_i\wedge d\ov{z}_j,$$
$$\omega_{\mathrm{SRF}}|_{X_y}=\mn\sum_{i,j=m+1}^n g_{i\ov{j}}dz_i\wedge d\ov{z}_j,$$
$$\Omega'=H (\mn)^ndz_1\wedge d\ov{z}_1\wedge\dots\wedge dz_n\wedge d\ov{z}_n,$$
so that in these coordinates we have
$$G=\frac{H}{\det(\eta_{i\ov{j}}) \det(g_{i\ov{j}})},$$
and so if we differentiate only along $X_y$ we have
$$\ddbar\log G=\Ric(\omega_{\mathrm{SRF}}|_{X_y})=0,$$
because $f^*\eta$ and $\Ric(\Omega')$ are pulled back from $Y$, and $\omega_{\mathrm{SRF}}|_{X_y}$ is Ricci-flat. Therefore $G$ is the pullback of a smooth positive function on $Y$,
still denoted by $G$.

Thanks to Aubin \cite{Au} and Yau \cite{Ya} there is a unique smooth function $\psi$ on $Y$ such that $\eta+\ddbar\psi>0$ and
\begin{equation}\label{maa}
(\eta+\ddbar\psi)^m=G e^\psi \eta^m.
\end{equation}
If we let $\omega_Y=\eta+\ddbar\psi$, then we can use Proposition \ref{wp1} to compute
\begin{equation}\label{long}\begin{split}
\Ric(\omega_Y)&=-\ddbar\log G -\ddbar\psi+\Ric(\eta)\\
&=\Ric(\Omega')+\ddbar\log(f^*\eta^m\wedge \omega_{\mathrm{SRF}}^{n-m})-\ddbar\psi+\Ric(\eta)\\
&=-\eta-\Ric(\eta)+\omega_{{\rm WP}}-\ddbar\psi+\Ric(\eta)\\
&=-\omega_Y+\omega_{{\rm WP}},
\end{split}\end{equation}
which is \eqref{twke}. Note here that the $(1,1)$ forms $\Ric(\Omega')$ and $\ddbar\log(f^*\eta^m\wedge \omega_{\mathrm{SRF}}^{n-m})$ which as written are defined on $X$, are in fact pullbacks of forms on $Y$ (the latter thanks to Proposition \ref{wp1}).

Conversely if $\omega_Y$ solves \eqref{twke}, then we obtain
$$[\omega_Y]=-2\pi c_1(Y)+[\omega_{{\rm WP}}]=[\eta],$$
and so $\omega_Y=\eta+\ddbar\psi$ for some smooth function $\psi$. We have
$$\ddbar\log\frac{\omega_Y^m}{Ge^\psi\eta^m}=-\Ric(\omega_Y)-\omega_Y+\omega_{{\rm WP}}=0,$$
using the same argument as in \eqref{long}, and so
$$\frac{\omega_Y^m}{Ge^\psi\eta^m}=c,$$
a positive constant on $Y$. Replacing $\psi$ with $\psi+\log c$ we may assume that $c=1$, and so $\psi$ solves \eqref{maa}.
But \eqref{maa} has a unique solution, as follows easily from the maximum principle, and so $\omega_Y$ is also unique.
\end{proof}

Let now
$$\Omega=\binom{n}{m} f^*\omega_Y^m\wedge \omega_{\mathrm{SRF}}^{n-m},$$ which is a smooth positive volume form on $X$.
Combining Propositions \ref{wp1} and \ref{wp2} we obtain that
$$\Ric(\Omega)=-f^*\omega_Y.$$
We define now reference forms on $X$
$$\hat{\omega}_t=e^{-t}\omega_{\mathrm{SRF}}+(1-e^{-t})f^*\omega_Y,$$
which are cohomologous to $\omega(t)$, and are positive definite for all $t\geq T_0$ (because $f^*\omega_Y$ is positive in the base directions and zero in the others, and $\omega_{\mathrm{SRF}}$ is positive in the fiber directions).
In fact, there is a uniform constant $C>0$ such that
\begin{equation}\label{use}
\hat{\omega}_t\geq C^{-1}e^{-t}\omega_0,
\end{equation}
for all $t\geq T_0$. Note also that
\begin{equation}\label{use2}
\hat{\omega}_t\geq \frac{1}{2}f^*\omega_Y,
\end{equation}
for all $t\geq T_0$.
Then \eqref{krf2} is equivalent to
\begin{equation}\label{ma2}
\left\{
                \begin{aligned}
                  &\frac{\de}{\de t}\vp(t)=\log\frac{e^{(n-m)t}(\hat{\omega}_t+\ddbar\vp(t))^n}{\Omega}-\vp(t)\\
                  &\vp(0)=-\rho\\
                  &\hat{\omega}_t+\ddbar\vp(t)>0
                \end{aligned}
              \right.
\end{equation}
Indeed, if $\vp(t)$ solves \eqref{ma2} and we define $\omega(t)=\hat{\omega}_t+\ddbar\vp(t)$, then
\[\begin{split}
\frac{\de}{\de t}\omega(t)&=\frac{\de}{\de t}(\hat{\omega}_t+\ddbar\vp(t))\\
&=-\hat{\omega}_t+f^*\omega_Y-\Ric(\omega(t))+\Ric(\Omega)-\ddbar\vp(t)\\
&=-\Ric(\omega(t))-\omega(t),
\end{split}\]
and \eqref{krf2} holds. Conversely, if $\omega(t)$ solves \eqref{krf2}, we define $\vp(t)$ by solving the ODE
$$\frac{\de}{\de t}\vp(t)=\log\frac{e^{(n-m)t}\omega(t)^n}{\Omega}-\vp(t),\quad \vp(0)=-\rho,$$
and compute
\[\begin{split}
\frac{\de}{\de t}(e^t(\omega(t)-\hat{\omega}_t-\ddbar\vp(t)))=e^t(-\Ric(\omega(t))+\Ric(\omega(t)))=0,
\end{split}\]
and since $(e^t(\omega(t)-\hat{\omega}_t-\ddbar\vp(t)))|_{t=0}=0$, we conclude that $\omega(t)=\hat{\omega}_t+\ddbar\vp(t)$ for all $t$ (such that the solution exists), and \eqref{ma2} holds.

Note that the factor of $e^{(n-m)t}$ in \eqref{ma2} did not play any role in this derivation, and indeed it could be omitted at this moment, but it becomes crucial when discussing the long time convergence properties of the flow.

As we mentioned earlier, the flow \eqref{ma2} has a solution defined on $[0,+\infty)$.

\subsection{$C^0$ estimates for the potential and its time derivative}

\begin{lemma}\label{c0}
There is a uniform constant $C>0$ such that for all $t\geq 0$ we have
\begin{equation}\label{phi}
|\vp(t)|\leq C(1+t)e^{-t},
\end{equation}
\begin{equation}\label{dphi}
|\dot{\vp}(t)|\leq Ce^{-\frac{t}{4}}.
\end{equation}
\end{lemma}
\begin{proof}
First, we observe that for $t\geq T_0$ we have
\begin{equation}\label{sharp}
\left|e^t\log\frac{e^{(n-m)t}\hat{\omega}_t^n}{\Omega}\right|\leq C.
\end{equation}
Indeed, we have
\[\begin{split}
e^t&\log\frac{e^{(n-m)t}\hat{\omega}_t^n}{\Omega}=e^t\log\frac{e^{(n-m)t}(e^{-t}\omega_{\mathrm{SRF}}+(1-e^{-t})f^*\omega_Y)^n}{\Omega}\\
&=e^t\log\frac{e^{(n-m)t}(\binom{n}{m}e^{-(n-m)t}(1-e^{-t})^mf^*\omega_Y^m\wedge \omega_{\mathrm{SRF}}^{n-m}+\cdots+e^{-nt}\omega_{\mathrm{SRF}}^{n})}{\binom{n}{m} f^*\omega_Y^m\wedge \omega_{\mathrm{SRF}}^{n-m}}\\
&=e^t\log(1+O(e^{-t})),
\end{split}
\]
which is bounded. We can then apply the maximum principle to $e^t\vp(t)-At$, for some constant $A>0$ to be determined. At a maximum point, assuming it is achieved at $t\geq T_0$, we have
\[\begin{split}
0&\leq \frac{\de}{\de t}(e^t\vp(t)-At)=e^t \log\frac{e^{(n-m)t}(\hat{\omega}_t+\ddbar\vp(t))^n}{\Omega}-A\\
&\leq e^t \log\frac{e^{(n-m)t}\hat{\omega}_t^n}{\Omega}-A\leq C-A<0,
\end{split}
\]
as long as we choose $A>C$, where we used \eqref{sharp}. Therefore we obtain a uniform upper bound for $e^t\vp(t)-At$, which proves that $\vp(t)\leq C(1+t)e^{-t}$, the upper bound in \eqref{phi}. The lower bound is similar.

In order to establish \eqref{dphi} we first show that
\begin{equation}\label{dphi2}
|\dot{\vp}(t)|\leq C.
\end{equation}
We apply the maximum principle to $\dot{\vp}(t)-A\vp(t)$, for some constant $A>0$ to be determined. At a maximum point, assuming it is achieved at $t\geq T_0$, we have
\[\begin{split}
0&\leq \left(\frac{\de}{\de t}-\Delta\right)(\dot{\vp}(t)-A\vp(t))\\
&=\tr{\omega(t)}{(f^*\omega_Y-\hat{\omega}_t)}+n-m-\dot{\vp}(t)-A\dot{\vp}(t)+An-A\tr{\omega(t)}{\hat{\omega}_t}\\
&\leq -(A+1)\dot{\vp}(t)+C,
\end{split}
\]
as long as we choose $A$ large enough so that $A\hat{\omega}_t\geq f^*\omega_Y$ for all $t\geq T_0$, using \eqref{use2}. Since $\vp(t)$ is bounded by \eqref{phi}, we conclude from this that $\dot{\vp}(t)\leq C$.

For the lower bound on $\dot{\vp}$, observe that
$$\tr{\omega(t)}{\hat{\omega}_t}\geq n\left(\frac{\hat{\omega}_t^n}{\omega(t)^n}\right)^{\frac{1}{n}}=n\left(e^{-\vp(t)-\dot{\vp}(t)}\frac{e^{(n-m)t}\hat{\omega}_t^n}{\Omega}\right)^{\frac{1}{n}}
\geq C^{-1}e^{-\frac{\dot{\vp}(t)}{n}},$$
using the arithmetic-geometric mean inequality, and the estimates \eqref{phi} and \eqref{sharp}. We can now apply the minimum principle to $\dot{\vp}(t)+2\vp(t)$.
At a minimum point, assuming it is achieved at $t\geq T_0$, we have
\[\begin{split}
0&\geq \left(\frac{\de}{\de t}-\Delta\right)(\dot{\vp}(t)+2\vp(t))\\
&=\tr{\omega(t)}{(f^*\omega_Y-\hat{\omega}_t)}+n-m-\dot{\vp}(t)+2\dot{\vp}(t)-2n+2\tr{\omega(t)}{\hat{\omega}_t}\\
&\geq \tr{\omega(t)}{\hat{\omega}_t}+\dot{\vp}(t)-C\\
&\geq C^{-1}e^{-\frac{\dot{\vp}(t)}{n}}+\dot{\vp}(t)-C,
\end{split}
\]
and so at this point we obtain
$$e^{-\frac{\dot{\vp}(t)}{n}}\leq C(1-\dot{\vp}(t)),$$
which gives a uniform lower bound for $\dot{\vp}(t)$ at this point, and hence everywhere (remembering \eqref{phi}). This proves \eqref{dphi2}.

We now prove \eqref{dphi}. Differentiating \eqref{ma2} we obtain
$$\frac{\de}{\de t}\dot{\vp}(t)=-R(t)-m-\dot{\vp}(t),$$
and using $R(t)\geq -C$ and \eqref{dphi2} we obtain
\begin{equation}\label{ddphi}
\frac{\de}{\de t}\dot{\vp}(t)\leq C_0.
\end{equation}
First we show the bound
$$\dot{\vp}(t)\leq Ce^{-\frac{t}{4}}.$$
If this fails, then we can find a sequence $(x_k,t_k)\in X\times[0,+\infty)$ such that $t_k\to\infty$ and $\dot{\vp}(x_k,t_k)\geq ke^{-\frac{t_k}{4}}$. If we let $\gamma_k=\frac{k}{2C_0}e^{-\frac{t_k}{4}}$ then it follows from \eqref{ddphi} that
$$\dot{\vp}(x_k,t)\geq \frac{k}{2}e^{-\frac{t_k}{4}},$$
for $t\in [t_k-\gamma_k,t_k]$. Integrating in $t$ we get
$$\vp(x_k,t_k)-\vp(x_k,t_k-\gamma_k)=\int_{t_k-\gamma_k}^{t_k}\dot{\vp}(x_k,t)dt\geq \gamma_k \frac{k}{2}e^{-\frac{t_k}{4}}=\frac{k^2}{4C_0}e^{-\frac{t_k}{2}}.$$
If for some value of $k$ we have $\gamma_k\leq 1$, then we can use \eqref{phi} to bound
$$\vp(x_k,t_k)-\vp(x_k,t_k-\gamma_k)\leq C(1+t_k)e^{-t_k}+C(1+t_k-\gamma_k)e^{-t_k+\gamma_k}\leq C(1+t_k)e^{-t_k},$$
and so we obtain
\begin{equation}\label{ein}
\frac{k^2}{4C_0}e^{-\frac{t_k}{2}}\leq C(1+t_k)e^{-t_k}.
\end{equation}
If on the other hand for some $k$ we have $\gamma_k\geq 1$ then we integrate in $t$ again
$$\vp(x_k,t_k)-\vp(x_k,t_k-1)=\int_{t_k-1}^{t_k}\dot{\vp}(x_k,t)dt\geq  \frac{k}{2}e^{-\frac{t_k}{4}},$$
and using \eqref{phi} again we obtain
\begin{equation}\label{zwei}
\frac{k}{2}e^{-\frac{t_k}{4}}\leq C(1+t_k)e^{-t_k}.
\end{equation}
One of the two cases must occur for infinitely many values of $k$, and so letting $k\to\infty$ in \eqref{ein} or \eqref{zwei} we obtain contradiction.

Finally, to prove the lower bound
$$\dot{\vp}(t)\geq -Ce^{-\frac{t}{4}},$$
we use the same argument with the interval $[t_k-\gamma_k,t_k]$ replaced by $[t_k,t_k+\gamma_k]$.
\end{proof}

\subsection{The parabolic Schwarz Lemma}

We have the following parabolic Schwarz Lemma, as in \cite{SoT} (see \cite{Ya} for the original Yau-Schwarz Lemma).
\begin{lemma}\label{schwarz}
There is a uniform constant $C>0$ such that for all $t\geq 0$ we have
\begin{equation}\label{schw}
\omega(t)\geq C^{-1}f^*\omega_Y.
\end{equation}
\end{lemma}
\begin{proof}
Given any point $x\in X$ we choose local coordinates $\{z_i\}$ on $X$ centered at $x$ which are normal for $\omega(t)$, and coordinates $\{y_\alpha\}$ on $Y$ near $f(x)$,
which are normal for $\omega_Y$.
In these coordinates we can represent
the map $f$ as an $m$-tuple of local holomorphic functions $\{f^\alpha\}$. We will use subscripts like $f^\alpha_i, f^\alpha_{ij},...$
to indicate partial derivatives. We will also write $g_{i\ov{j}}$ for the entries of $\omega(t)$ in these coordinates, and $h_{\alpha\ov{\beta}}$ for those of $\omega_Y$.
In these coordinates we have $\tr{\omega(t)}{(f^*\omega_Y)}=g^{k\ov{\ell}}f^\alpha_k \ov{f^\beta_\ell}h_{\alpha\ov{\beta}}$.

Then we have
\begin{equation}\label{scc}\begin{split}
 \left(\frac{\de}{\de t}-\Delta\right)\tr{\omega(t)}{(f^*\omega_Y)}&=g^{k\ov{j}} g^{i\ov{\ell}}R_{i\ov{j}}f^\alpha_k \ov{f^\beta_\ell}h_{\alpha\ov{\beta}}+\tr{\omega(t)}{(f^*\omega_Y)} -g^{i\ov{j}}\de_i \de_{\ov{j}}\left(g^{k\ov{\ell}}f^\alpha_k \ov{f^\beta_\ell}h_{\alpha\ov{\beta}}\right)\\
 &=g^{k\ov{j}} g^{i\ov{\ell}}R_{i\ov{j}}f^\alpha_k \ov{f^\beta_\ell}h_{\alpha\ov{\beta}}+\tr{\omega(t)}{(f^*\omega_Y)}-g^{i\ov{j}}g^{k\ov{\ell}} f^\alpha_{ki} \ov{f^\beta_{\ell j} } h_{\alpha\ov{\beta}}\\
& -g^{i\ov{j}} g^{p\ov{\ell}}g^{k\ov{q}} f^\alpha_k \ov{f^\beta_\ell} h_{\alpha\ov{\beta}}   R_{i\ov{j}p\ov{q}}
+g^{i\ov{j}} g^{k\ov{\ell}} f^\alpha_k \ov{f^\beta_\ell} f^\gamma_i \ov{f^\delta_j} (R_Y)_{\alpha\ov{\beta}\gamma\ov{\delta}}\\
&\leq \tr{\omega(t)}{(f^*\omega_Y)}-g^{i\ov{j}}g^{k\ov{\ell}} f^\alpha_{ki} \ov{f^\beta_{\ell j} } h_{\alpha\ov{\beta}}+C(\tr{\omega(t)}{(f^*\omega_Y)})^2,
\end{split}
\end{equation}
where in the last line we used the following argument: if we set $\xi_i=df(\frac{\de}{\de z_i})=\sum_\alpha f^\alpha_i \frac{\de}{\de y_\alpha}$ then at our point $x$ we have
\[\begin{split}
g^{i\ov{j}} g^{k\ov{\ell}} f^\alpha_k \ov{f^\beta_\ell} f^\gamma_i \ov{f^\delta_j} (R_Y)_{\alpha\ov{\beta}\gamma\ov{\delta}}&=\sum_{i,k}
f^\alpha_k \ov{f^\beta_k} f^\gamma_i \ov{f^\delta_i} (R_Y)_{\alpha\ov{\beta}\gamma\ov{\delta}}=\sum_{i,k}\mathrm{Rm}_Y(\xi_i,\ov{\xi_i},\xi_k,\ov{\xi_k})\\
&\leq C\sum_{i,k} |\xi_i|^2_{\omega_Y} |\xi_k|^2_{\omega_Y}=C(\tr{\omega(t)}{(f^*\omega_Y)})^2,
\end{split}\]
where the constant $C$ is an upper bound for the bisectional curvature of $\omega_Y$ among all $\omega_Y$-unit vectors.
Now at $x$ we have $\de_i (\tr{\omega(t)}{(f^*\omega_Y)})=\sum_{k,\alpha} f^\alpha_{ki}\ov{f^\alpha_k},$ and using the Cauchy-Schwarz inequality we have
\begin{equation}\label{cs0} \begin{split}
|\de \tr{\omega(t)}{(f^*\omega_Y)}|^2_{\omega(t)} = & \sum_{i,k,p,\alpha,\beta} f^\alpha_{ki}\ov{f^\beta_{pi}}f^\beta_{p}\ov{f^\alpha_k}\\
\leq& \sum_{k,p,\alpha,\beta}|f^\alpha_k||f^\beta_p| \left( \sum_i |f^\alpha_{ki}|^2\right)^{1/2} \left( \sum_j |f^\beta_{pj}|^2\right)^{1/2}\\
= &\left(\sum_{k,\alpha} |f^\alpha_k| \left(\sum_i |f^\alpha_{ki}|^2 \right)^{1/2}\right)^2\leq \left(\sum_{\ell,\beta} |f^\beta_\ell|^2\right)
\left(\sum_{i, k,\alpha}|f^\alpha_{ki}|^2\right)\\
=  & (\tr{\omega(t)}{(f^*\omega_Y)})g^{i\ov{j}}g^{k\ov{\ell}} f^\alpha_{ki} \ov{f^\beta_{\ell j} } h_{\alpha\ov{\beta}},
\end{split}\end{equation}
and combining \eqref{scc} and \eqref{cs0} we obtain
\begin{equation}\label{sccc}
\left(\frac{\de}{\de t}-\Delta\right)\log\tr{\omega(t)}{(f^*\omega_Y)}\leq C\tr{\omega(t)}{(f^*\omega_Y)}+1,
\end{equation}
at every point where $\tr{\omega(t)}{(f^*\omega_Y)}>0$.
On the other hand we also have
$$\left(\frac{\de}{\de t}-\Delta\right)\vp(t)=\dot{\vp}(t)-n+\tr{\omega(t)}{\hat{\omega}_t}\geq\frac{1}{2}\tr{\omega(t)}{(f^*\omega_Y)}-C,$$
for $t\geq T_0$, thanks to \eqref{use2} and Lemma \ref{c0}.
Therefore, if we choose $A$ large enough, we have that
\[\begin{split} \left(\frac{\de}{\de t}-\Delta\right)(\log\tr{\omega(t)}{(f^*\omega_Y)}-A\vp(t))&\leq -\tr{\omega(t)}{(f^*\omega_Y)}+C,
\end{split}\]
and from this we conclude easily that $\tr{\omega(t)}{(f^*\omega_Y)}\leq C$ on $X\times[0,\infty)$ (note that at a maximum point of $\log\tr{\omega(t)}{(f^*\omega_Y)}-A\vp(t)$ we must have $\tr{\omega(t)}{(f^*\omega_Y)}>0$).
\end{proof}

\subsection{An optimal $C^0$ estimate for the evolving metric}
Define a smooth function $\underline{\vp}(t)$ on $Y$ by
$$\underline{\vp}(t)(y)=\frac{\int_{X_y}\vp(t)\omega_0^{n-m}}{\int_{X_y}\omega_0^{n-m}},$$
which is just the fiberwise average of $\vp(t)$. We will also denote its pullback $f^*\underline{\vp}(t)$ to $X$ by $\underline{\vp}(t)$.
\begin{lemma}\label{improve}
There is a uniform constant $C>0$ such that for all $t\geq 0$ we have
\begin{equation}\label{c02}
\sup_X|\vp(t)-\underline{\vp}(t)|\leq Ce^{-t}.
\end{equation}
\end{lemma}
\begin{proof}
Let $\psi(t)=e^t(\vp(t)-\underline{\vp}(t))$. When we restrict to a fiber $X_y$ we have
$$e^t\omega(t)|_{X_y}=\omega_{\mathrm{SRF}}|_{X_y}+\ddbar (\psi(t)|_{X_y}),$$
and
\[\begin{split}
\frac{(\omega_{\mathrm{SRF}}|_{X_y}+\ddbar (\psi(t)|_{X_y}))^{n-m}}{(\omega_{0}|_{X_y})^{n-n}}&=\frac{e^{(n-m)t}\omega(t)^{n-m}\wedge f^*\omega_Y^m}{\omega_0^{n-m}\wedge f^*\omega_Y^m}\\
&=\frac{\omega(t)^{n-m}\wedge f^*\omega_Y^m}{\omega(t)^n} \frac{e^{\vp(t)+\dot{\vp}(t)}\Omega}{\omega_0^{n-m}\wedge f^*\omega_Y^m}\\
&\leq C(\tr{\omega(t)}{(f^*\omega_Y)})^{n-m}\\
&\leq C,
\end{split}\]
using Lemmas \ref{c0} and \ref{schwarz}, and the elementary inequality
$$\frac{\omega(t)^{n-m}\wedge f^*\omega_Y^m}{\omega(t)^n}\leq \left(\frac{\omega(t)^{n-1}\wedge f^*\omega_Y}{\omega(t)^n}\right)^{n-m},$$
which follows for example from the Maclaurin inequality
between elementary symmetric functions.
Therefore Yau's $C^0$ estimate \cite{Ya} applies, and using also that $\int_{X_y}\psi(t)\omega_0^{n-m}=0$, we obtain
$$\sup_{X_y} \left|\psi(t)|_{X_y}\right|\leq C,$$
independent of $t$. Furthermore, this constant is uniform in $y\in Y$, since it depends only on geometric quantities on the manifold $(X_y,\omega_{\rm SRF}|_{X_y})$ (specifically its Sobolev and Poincar\'e constants) and these are uniformly bounded in $y$. This proves \eqref{c02}.
\end{proof}

\begin{lemma}\label{ddt}
There is a uniform constant $C>0$ such that for all $t\geq 0$ we have
\begin{equation}\label{dddt}
\left(\frac{\de}{\de t}-\Delta\right)\underline{\vp}(t)\leq C.
\end{equation}
\end{lemma}
\begin{proof}
We have
$$\frac{\de}{\de t}\underline{\vp}(t)=\frac{\int_{X_y}\dot{\vp}(t)\omega_0^{n-m}}{\int_{X_y}\omega_0^{n-m}}\leq C,$$
by Lemma \ref{c0}. Next, recall from \eqref{const} that $\int_{X_y}\omega_0^{n-m}$ does not depend on $y$, so it is enough to estimate
$\Delta\left(\int_{X_y}\vp(t)\omega_0^{n-m}\right)$. To compute this, it is convenient to write the integral $\int_{X_y}\vp(t)\omega_0^{n-m}$ using fiber integration as
$$\int_{X_y}\vp(t)\omega_0^{n-m}=f_*(\vp(t)\omega_0^{n-m})(y),$$
where the fiber integration map $f_*$ is defined for every proper submersion, it commutes with $d$, and since $f$ is holomorphic it preserves the $(p,q)$ types of forms, and therefore it also commutes with $\de$ and $\db$.
Then we have
$$\ddbar \left(\int_{X_y}\vp(t)\omega_0^{n-m}\right)=\ddbar f_*(\vp(t)\omega_0^{n-m})=f_*(\ddbar\vp(t)\wedge\omega_0^{n-m}),$$
and so
\[\begin{split}
\Delta\left(\int_{X_y}\vp(t)\omega_0^{n-m}\right)&=\tr{\omega(t)}{f^*(f_*(\ddbar\vp(t)\wedge\omega_0^{n-m}))}\\
&=\tr{\omega(t)}{f^*(f_*((\omega(t)-\hat{\omega}_t)\wedge\omega_0^{n-m}))}\\
&\geq -\tr{\omega(t)}{f^*(f_*(\hat{\omega}_t\wedge\omega_0^{n-m}))},
\end{split}\]
but $f_*(\hat{\omega}_t\wedge\omega_0^{n-m}))$ is a smooth $(1,1)$ form on $Y$ which satisfies
$$f_*(\hat{\omega}_t\wedge\omega_0^{n-m}))\leq C\omega_Y,$$
for all $t\geq 0$. The Schwarz Lemma estimate \eqref{schw} then implies that
$$\Delta\left(\int_{X_y}\vp(t)\omega_0^{n-m}\right)\geq -C,$$
and \eqref{dddt} follows.
\end{proof}

\begin{proposition}\label{equiv2}
There is a uniform constant $C>0$ such that for all $t\geq T_0$ we have
\begin{equation}\label{equiv3}
C^{-1}\hat{\omega}_t\leq \omega(t)\leq C\hat{\omega}_t.
\end{equation}
\end{proposition}
\begin{proof}
We apply the maximum principle to $$\log(e^{-t}\tr{\omega(t)}{\omega_0})-Ae^t(\vp(t)-\underline{\vp}(t)),$$
for a constant $A$ to be determined. To compute the evolution of
$\log(e^{-t}\tr{\omega(t)}{\omega_0})$ we just use the Schwarz Lemma calculation in \eqref{sccc} to the identity map from $(X,\omega(t))$ to $(X,\omega_0)$, which gives
$$\left(\frac{\de}{\de t}-\Delta\right)\log(e^{-t}\tr{\omega(t)}{\omega_0})\leq C \tr{\omega(t)}{\omega_0}.$$
At a maximum point of our quantity, assuming it is achieved at $t>0$, we have
\[\begin{split}
0&\leq \left(\frac{\de}{\de t}-\Delta\right)(\log(e^{-t}\tr{\omega(t)}{\omega_0})-Ae^t(\vp(t)-\underline{\vp}(t)))\\
&\leq C\tr{\omega(t)}{\omega_0}-A e^t(\vp(t)-\underline{\vp}(t))-Ae^t\dot{\vp}(t)+Ane^t-Ae^t\tr{\omega(t)}{\hat{\omega}_t}+CAe^t\\
&\leq CAe^t-\tr{\omega(t)}{\omega_0},
\end{split}
\]
as long as we choose $A$ sufficiently large, using \eqref{use}, \eqref{dphi2}, \eqref{c02} and \eqref{dddt}. Therefore we conclude that
$$e^{-t}\tr{\omega(t)}{\omega_0}\leq C,$$
on $X\times [0,\infty)$, which implies
$$\omega(t)\geq C^{-1}e^{-t}\omega_0\geq C^{-1}e^{-t}\omega_{\mathrm{SRF}},$$
and adding this to \eqref{schw} we obtain
$$\omega(t)\geq C^{-1}\hat{\omega}_t,$$
which is half of \eqref{equiv3}. For the other half, it is enough to observe that
$$\frac{\omega(t)^n}{\hat{\omega}_t^n}=e^{\vp(t)+\dot{\vp}(t)}\frac{\Omega}{e^{(n-m)t}\hat{\omega}_t^n}\leq C,$$
thanks to Lemma \ref{c0} and \eqref{sharp}, and so the upper bound
$$\omega(t)\leq C\hat{\omega}_t,$$
follows.
\end{proof}

\subsection{$C^0$ convergence of the evolving metric}
\begin{lemma}\label{base}
There is a uniform constant $C>0$ such that for all $t\geq 0$ we have
\begin{equation}\label{base2}
\tr{\omega(t)}{(f^*\omega_Y)}\leq m+Ce^{-\frac{t}{8}}.
\end{equation}
\end{lemma}
\begin{proof}
We apply the maximum principle to $$e^{\frac{t}{8}}(\tr{\omega(t)}{(f^*\omega_Y)}-m)-e^{\frac{t}{4}}(\vp(t)+\dot{\vp}(t)).$$
To compute the evolution of this quantity, we first calculate
$$\left(\frac{\de}{\de t}-\Delta\right)\vp(t)=\dot{\vp}(t)-n+\tr{\omega(t)}{\hat{\omega}_t},$$
$$\left(\frac{\de}{\de t}-\Delta\right)\dot{\vp}(t)=\tr{\omega(t)}{(f^*\omega_Y-\hat{\omega}_t)}+n-m-\dot{\vp}(t),$$
$$\left(\frac{\de}{\de t}-\Delta\right)(\vp(t)+\dot{\vp}(t))=\tr{\omega(t)}{(f^*\omega_Y)}-m,$$
and from the Schwarz Lemma calculation \eqref{scc}, together with \eqref{equiv3},
$$\left(\frac{\de}{\de t}-\Delta\right)\tr{\omega(t)}{(f^*\omega_Y)}\leq \tr{\omega(t)}{(f^*\omega_Y)}+C(\tr{\omega(t)}{(f^*\omega_Y)})^2\leq C.$$
At a maximum point of our quantity, assuming it is achieved at $t>0$, we have
\[\begin{split}
0&\leq \left(\frac{\de}{\de t}-\Delta\right)(e^{\frac{t}{8}}(\tr{\omega(t)}{(f^*\omega_Y)}-m)-e^{\frac{t}{4}}(\vp(t)+\dot{\vp}(t)))\\
&\leq \frac{e^{\frac{t}{8}}}{8}(\tr{\omega(t)}{(f^*\omega_Y)}-m)+Ce^{\frac{t}{8}}-\frac{e^{\frac{t}{4}}}{4}(\vp(t)+\dot{\vp}(t))-e^{\frac{t}{4}}(\tr{\omega(t)}{(f^*\omega_Y)}-m)\\
&\le Ce^{\frac{t}{8}}-\frac{e^{\frac{t}{4}}}{2}(\tr{\omega(t)}{(f^*\omega_Y)}-m),
\end{split}
\]
using Lemma \ref{c0}. Therefore we get a uniform upper bound for this quantity, and hence for $e^{\frac{t}{8}}(\tr{\omega(t)}{(f^*\omega_Y)}-m)$.
\end{proof}

\begin{theorem}\label{baseb}
There are uniform constants $C,\eta>0$ such that for all $t\geq T_0$ we have
\begin{equation}\label{baseb2}
\tr{\omega(t)}{\hat{\omega}_t}\leq n+Ce^{-\eta t}.
\end{equation}
\end{theorem}
This result may seem similar to the one obtained in Lemma \ref{base}, but it is much more powerful and its proof is considerably harder.
This was originally proved when $n=2$ in \cite{TWY0} (for a more general flow of Hermitian metrics, which specializes to the K\"ahler-Ricci flow when the initial metric is K\"ahler). The method of proof used there is special to this dimension, because in this case the reference metrics $\hat{\omega}_t$ satisfy $|\widehat{\mathrm{Rm}}(t)|_{\hat{\omega}_t}\leq Ce^{\frac{t}{2}},$ while in general dimensions this is $O(e^t)$. In these notes we present the proof obtained in \cite{TWY}, which works in all dimensions. This will require us to first
prove strong estimates for the metric along the fibers, including proving \eqref{conve2p}, and then we will be able to prove \eqref{baseb2}.

Before proving Theorem \ref{baseb}, we use it to complete the proof of \eqref{conve1p}.
\begin{proof}[Proof of \eqref{conve1p}]
We observe that for $t\geq T_0$ we have,
\begin{equation}\label{baseb3}
\frac{\hat{\omega}_t^n}{\omega(t)^n}=e^{-\vp(t)-\dot{\vp}(t)}\frac{e^{(n-m)t}\hat{\omega}_t^n}{\binom{n}{m}f^*\omega_Y^m\wedge \omega_{\mathrm{SRF}}^{n-m}}\geq
e^{(-C(1+t)e^{-t}-Ce^{-\frac{t}{4}}-Ce^{-t})}\geq 1-Ce^{-\frac{t}{4}},
\end{equation}
using Lemma \ref{c0}.
If now at any given point we choose local holomorphic coordinates so that $\omega(t)$ is the identity and $\hat{\omega}_t$ is given by a positive definite $n\times n$ Hermitian matrix $A$, then \eqref{baseb2} and \eqref{baseb3} give
$$\mathrm{tr} A\leq n + Ce^{-\eta t}, \quad \det A\geq 1-Ce^{-\frac{t}{4}},$$
and so Lemma \ref{matrix} below gives
$$\|A-\mathrm{Id}\|\leq Ce^{-\frac{\eta}{2} t},$$
which means
$$\|\hat{\omega}_t-\omega(t)\|_{C^0(X,\omega(t))}\leq Ce^{-\frac{\eta}{2} t},$$
and since $\omega(t)\leq C\omega_0$ (by Proposition \ref{equiv2}), this gives
$$\|\hat{\omega}_t-\omega(t)\|_{C^0(X,\omega_0)}\leq Ce^{-\frac{\eta}{2} t},$$
and remembering that $\hat{\omega}_t=f^*\omega_Y+e^{-t}(\omega_{\mathrm{SRF}}-f^*\omega_Y)$, this gives \eqref{conve1p}.
\end{proof}

In the proof we have used the following elementary result:
\begin{lemma} \label{matrix}
Let $A$ be an $n\times n$ positive definite Hermitian matrix such that
$$  \mathrm{tr} A \leq n + \ve, \quad   \det A   \geq 1-\ve,$$
for some $0<\ve<1$.
Then there is a constant $C$ which depends only on $n$ such that
$$\| A- \mathrm{Id} \|^2 \le C \ve,$$
where $\|\cdot\|$ is the Hilbert-Schmidt norm, and $\mathrm{Id}$ is the $n\times n$ identity matrix.
\end{lemma}
\begin{proof}  The lemma is trivial for $n=1$ so we may assume that $n\geq 2$.
Let $\lambda_1, \ldots, \lambda_n>0$ be the eigenvalues of $A$.  Define the normalized elementary symmetric polynomials $S_k$ by
$$S_k = {\binom{n}{k}}^{-1} \sum_{1 \le i_1 < \cdots < i_k \le n} \lambda_{i_1} \cdots \lambda_{i_k}, \quad \textrm{for } k=1, \dots, n.$$
By assumption we have that $S_1 \le 1+ \frac{\ve}{n}$ and $S_n \geq 1-\ve$.
Together with the Maclaurin inequalities we obtain
$$1+\frac{\ve}{n}\geq S_1 \geq \sqrt{S_2} \geq S_3^{\frac{1}{3}}\geq\dots\geq S_n^{\frac{1}{n}}\geq 1-\ve,$$
which implies that $|S_1-1|+|S_2-1| \le C \ve$ for $C$ depending only on $n$.
A direct calculation gives
$$\| A-\mathrm{Id} \|^2 = \sum_{j=1}^n (\lambda_j-1)^2 = n^2 S_1^2 -2nS_1- n(n-1) S_2 +n \le C \ve,$$
for $C$ depending only on $n$.
\end{proof}

\subsection{Estimates for the metric along the fibers}
Our goal now is to prove \eqref{conve2p}, which we will then use to prove Theorem \ref{baseb}. The first step is the following:
\begin{theorem}\label{fiberc3}
There is a constant $C>0$ such that for every $y\in Y$ and all $t\geq 0$ we have
\begin{equation}\label{fibc3}
\|e^t\omega(t)|_{X_y}\|_{C^1(X_y,\omega_0|_{X_y})}\leq C.
\end{equation}
\end{theorem}
In fact, we will reprove this result in Theorem \ref{fibercinf} below, but we decided to still present this proof in detail since it is self-contained.
\begin{proof}
Given a point $x\in X$, let $y=f(x)$ and choose local product coordinates on an open set $U\ni x$ and on $f(U)\ni y$ as in Lemma \ref{adj}, and let $\omega_E$ be the Euclidean metric on $U$
in these coordinates. We may also assume that in these coordinates $U$ and $f(U)$ are identified with unit balls in $\mathbb{C}^n$ and $\mathbb{C}^m$ respectively, with $x$ and $y$ being the origin. We claim that on the half-ball $B_{\frac{1}{2}}(0)\subset U$ we have
\begin{equation}\label{fibc31}
|\nabla^E\omega(t)|^2_{\omega(t)}\leq Ce^t,
\end{equation}
for all $t\geq 0$. Assuming this holds, then restricting \eqref{equiv3} to $X_y$ we obtain
\begin{equation}\label{c30}
C^{-1}e^{-t}\omega_E|_{X_y}\leq \omega(t)|_{X_y}\leq Ce^{-t}\omega_E|_{X_y},
\end{equation}
and so on $B_{\frac{1}{2}}(0)$ we obtain
\begin{equation}\label{c31}\begin{split}
|\nabla^E (e^t\omega(t)|_{X_y})|^2_{\omega_E}&=e^{-t}|\nabla^E (\omega(t)|_{X_y})|^2_{e^{-t}\omega_E}\leq Ce^{-t}|\nabla^E (\omega(t)|_{X_y})|^2_{\omega(t)}\\
&\leq Ce^{-t}|\nabla^E \omega(t)|^2_{\omega(t)}\leq C,
\end{split}\end{equation}
using \eqref{fibc31}. Then \eqref{c30} and \eqref{c31} together prove \eqref{fibc3} on $B_{\frac{1}{2}}(0)$, and a simple covering argument gives \eqref{fibc3} everywhere.

We are left with proving \eqref{fibc31}. Following Yau \cite{Ya} we define a smooth nonnegative function on $U$ by
$$S=|\nabla^E\omega(t)|^2_{\omega(t)},$$
which in fact equals $|\Gamma|^2_{\omega(t)}$ where $\Gamma_{ij}^k$ are the Christoffel symbols of $\omega(t)$.
We calculate
$$\frac{\de}{\de t}\Gamma_{ij}^k=\frac{\de}{\de t}\left(g^{k\ov{\ell}}\de_i g_{j\ov{\ell}}\right)=-g^{k\ov{\ell}}\de_i R_{j\ov{\ell}}+g^{k\ov{q}}g^{p\ov{\ell}}R_{p\ov{q}}\de_i g_{j\ov{\ell}}=-g^{k\ov{\ell}}\nabla_i R_{j\ov{\ell}},$$
where $\nabla$ is the covariant derivative of $\omega(t)$.
We also have
$$g^{p\ov{q}}\nabla_p\nabla_{\ov{q}}\Gamma_{ij}^k=g^{p\ov{q}}\nabla_p(\de_{\ov{q}}\Gamma_{ij}^k)=-g^{p\ov{q}}\nabla_p R_{ji\ov{q}}^k=-g^{p\ov{q}}\nabla_i R_{jp\ov{q}}^k
=-g^{k\ov{\ell}}\nabla_i R_{j\ov{\ell}},$$
using the second Bianchi identity, and
$$g^{p\ov{q}}\nabla_{\ov{q}}\nabla_p\Gamma_{ij}^k-g^{p\ov{q}}\nabla_p\nabla_{\ov{q}}\Gamma_{ij}^k=g^{p\ov{q}}R_{i\ov{q}}\Gamma_{pj}^k+g^{p\ov{q}}R_{j\ov{q}}\Gamma_{ip}^k
-g^{k\ov{q}}R_{p\ov{q}}\Gamma_{ij}^p,$$
\[\begin{split}
\Delta S&=g^{p\ov{q}}\nabla_p\nabla_{\ov{q}}\left(g^{i\ov{a}}g^{j\ov{b}}g_{k\ov{c}}\Gamma_{ij}^k \ov{\Gamma_{ab}^c}\right)\\
&=
|\nabla \Gamma|^2_{\omega(t)}+|\ov{\nabla} \Gamma|^2_{\omega(t)}
-2\mathrm{Re}\left(g^{i\ov{a}}g^{j\ov{b}}g_{k\ov{c}}g^{k\ov{\ell}}\nabla_i R_{j\ov{\ell}} \ov{\Gamma_{ab}^c}\right)\\
&+g^{i\ov{a}}g^{j\ov{b}}g_{k\ov{c}}g^{p\ov{q}}\Gamma_{ij}^k R_{p\ov{a}}\ov{\Gamma_{qb}^c}
+g^{i\ov{a}}g^{j\ov{b}}g_{k\ov{c}}g^{p\ov{q}}\Gamma_{ij}^kR_{p\ov{b}}\ov{\Gamma_{aq}^c}
-g^{i\ov{a}}g^{j\ov{b}}\Gamma_{ij}^kR_{k\ov{q}}\ov{\Gamma_{ab}^q},
\end{split}\]
\[\begin{split}
\frac{\de}{\de t}S&=S-2\mathrm{Re}\left(g^{i\ov{a}}g^{j\ov{b}}g_{k\ov{c}}g^{k\ov{\ell}}\nabla_i R_{j\ov{\ell}} \ov{\Gamma_{ab}^c}\right)\\
&+g^{i\ov{a}}g^{j\ov{b}}g_{k\ov{c}}g^{p\ov{q}}\Gamma_{ij}^k R_{p\ov{a}}\ov{\Gamma_{qb}^c}
+g^{i\ov{a}}g^{j\ov{b}}g_{k\ov{c}}g^{p\ov{q}}\Gamma_{ij}^kR_{p\ov{b}}\ov{\Gamma_{aq}^c}
-g^{i\ov{a}}g^{j\ov{b}}\Gamma_{ij}^kR_{k\ov{q}}\ov{\Gamma_{ab}^q},
\end{split}\]
and so
\[\begin{split}
\left(\frac{\de}{\de t}-\Delta\right)S=S-|\nabla \Gamma|^2_{\omega(t)}-|\ov{\nabla} \Gamma|^2_{\omega(t)}.
\end{split}\]
Let now $\rho$ be a smooth nonnegative cutoff function, which is supported in $B_1(0)$ and is identically $1$ on $B_{\frac{1}{2}}(0)$, and with
$$\mn\de\rho\wedge\db\rho\leq C\omega_E,\quad -C\omega_E\leq \ddbar(\rho^2)\leq C\omega_E,$$
where $C$ is a dimensional constant. Recalling \eqref{use} and \eqref{equiv3}, we obtain that
$\omega(t)\geq C^{-1}e^{-t}\omega_E$, and so
$$|\nabla\rho|^2_{\omega(t)}\leq Ce^t\quad \Delta(\rho^2)\geq -Ce^t,$$
on $U$.

We can then compute
\[\begin{split}
\left(\frac{\de}{\de t}-\Delta \right)(\rho^2 S)&\leq \rho^2\left(\frac{\de}{\de t}-\Delta \right)S+C Se^t+2|\langle \nabla \rho^2,\nabla S\rangle_{\omega(t)}|\\
&\leq \rho^2S-\rho^2\left(|\nabla \Gamma|^2_{\omega(t)}+|\ov{\nabla} \Gamma|^2_{\omega(t)}\right)+C Se^t+2|\langle \nabla \rho^2,\nabla S\rangle_{\omega(t)}|.
\end{split}\]
On the other hand, using the Young inequality
\[\begin{split}
2|\langle \nabla \rho^2,\nabla S\rangle_{\omega(t)}|
&=4\rho|\langle \nabla\rho, \nabla |\Gamma|^2_{\omega(t)}\rangle_{\omega(t)}|\leq 4\rho |\nabla \rho|_{\omega(t)}\cdot | \nabla |\Gamma|^2_{\omega(t)}|_{\omega(t)}\\
&\leq 4\rho |\nabla \rho|_{\omega(t)}|\Gamma|_{\omega(t)}\left(|\nabla \Gamma|_{\omega(t)}+|\ov{\nabla} \Gamma|_{\omega(t)}\right)\\
&\leq \rho^2\left(|\nabla \Gamma|^2_{\omega(t)}+|\ov{\nabla} \Gamma|^2_{\omega(t)}\right)+CS|\nabla \rho|^2_{\omega(t)}\\
&\leq \rho^2\left(|\nabla \Gamma|^2_{\omega(t)}+|\ov{\nabla} \Gamma|^2_{\omega(t)}\right)+CSe^t,
\end{split}\]
and so
$$\left(\frac{\de}{\de t}-\Delta\right) (\rho^2 S)\leq CSe^t,$$
$$\left(\frac{\de}{\de t}-\Delta\right) (e^{-t}\rho^2 S)\leq CS.$$
Next, on $B_1(0)$ we define
$$\omega_t=\omega_E^{(m)}+e^{-t}\omega_E^{(n-m)},$$
where $\omega_E^{(m)}$ and $\omega_E^{(n-m)}$ denote the Euclidean metrics on the two factors of $\mathbb{C}^n=\mathbb{C}^m\times\mathbb{C}^{n-m}$. Thanks to \eqref{equiv3} we have that
\begin{equation}\label{nice}
C^{-1}\omega_t\leq \omega(t)\leq C\omega_t,
\end{equation}
on $U$ for all $t\geq 0$. Note that the covariant derivative of $\omega_t$ just equals $\nabla^E$, independent of $t$, and that $\omega_t$ is flat.
Then, as in \eqref{gettt}, we can compute
\[
\begin{split}
\left(\frac{\de}{\de t}-\Delta\right) \tr{\omega_{t}}{\omega(t)}&=  -\tr{\omega_{t}}{\omega}+e^{-t}g_{t}^{i\ov{q}} \, g_{t}^{p\ov{j}}\, (g^{(n-m)}_E)_{p\ov{q}} \, g_{i\ov{j}} -g_{t}^{i\ov{\ell}}\, g^{p\ov{j}} g^{k\ov{q}} \nabla^E_i g_{k\ov{j}} \nabla^E_{\ov{\ell}}g_{p\ov{q}}\\
&\leq -g_{t}^{i\ov{\ell}}\, g^{p\ov{j}} g^{k\ov{q}} \nabla^E_i g_{k\ov{j}} \nabla^E_{\ov{\ell}}g_{p\ov{q}}\\
& \leq  -C^{-1}S,
\end{split}
\]
using \eqref{nice} and the fact that $e^{-t}\omega_E^{(n-m)}\leq \omega_t$. It follows that if we take $C_0$ large enough, then we have
$$ \left(\frac{\de}{\de t}-\Delta \right)\left(e^{-t}\rho^2 S+C_0\tr{\omega_{t}}{\omega(t)}\right)  \leq 0.$$
Note that we have $\tr{\omega_{t}}{\omega(t)}\leq C$, thanks to \eqref{nice}. Since $\rho=0$ on the boundary of $B_{1}(0)$, the maximum principle then gives
that $e^{-t}\rho^2 S+C_0\tr{\omega_{t}}{\omega(t)}\leq C$ on $B_1(0)\times[0,\infty)$, and so
$\sup_{B_{1/2}(0)}S\leq Ce^t,$ as required.
\end{proof}

The following improvement is due to Zhang and the author \cite{TZ} (and in fact it also gives another proof of Theorem \ref{fiberc3}):
\begin{theorem}\label{fibercinf}
For every $k\geq 1$ there is a constant $C_k>0$ such that for every $y\in Y$ and all $t\geq 0$ we have
\begin{equation}\label{fibcinf}
\|e^t\omega(t)|_{X_y}\|_{C^k(X_y,\omega_0|_{X_y})}\leq C_k.
\end{equation}
\end{theorem}
\begin{proof}
Given a point $x_0\in X$, let $y_0=f(x_0)$.
To prove \eqref{fibcinf} we choose local product coordinates on an open set $U\ni x_0$ and on $f(U)\ni y_0$, centered at these points as in the proof of Theorem \ref{fiberc3}, and let $\omega_E$ be the Euclidean metric on $U$ in these coordinates. We may assume that $f(U)$ is the unit ball in $\mathbb{C}^m$, and $U$ is the product of the unit balls in $\mathbb{C}^m$ and $\mathbb{C}^{n-m}$.

For each $t\geq 0$ let $B_t=B_{e^{t/2}}(0)\subset\mathbb{C}^m$, let $B=B_1(0)\subset \mathbb{C}^{n-m}$, and define rescaling holomorphic maps
$$F_t:B_t\times B\to U=B_0\times B,\quad F_t(y,z)=(ye^{-t/2},z).$$
These maps are all equal to the identity when restricted to $\{0\}\times B$, which is a ``vertical'' chart contained in the fiber $X_{y_0}$.
Thanks to \eqref{equiv3} we have
$$C^{-1}(f^*\omega_Y+e^{-t}\omega_0)\leq \omega(t)\leq C(f^*\omega_Y+e^{-t}\omega_0),$$
on $U$, and so the metrics
$$\omega_t(s):=e^t F_t^*\omega(se^{-t}+t),\quad -1\leq s\leq 0,$$
on $B_t\times B$ satisfy
$$C^{-1}F_t^*(e^tf^*\omega_Y+\omega_0)\leq \omega_t(s)\leq CF_t^*(e^tf^*\omega_Y+\omega_0),$$
and
$$\frac{\de}{\de s}\omega_t(s)=-\Ric(\omega_t(s))-e^{-t}\omega_t(s), \quad -1\leq s\leq 0.$$
It is readily verified, using product coordinates as above, that the metrics
$F_t^*(e^tf^*\omega_Y+\omega_0)$ converge smoothly on compact subsets of $\mathbb{C}^{m}\times B$ to a limiting K\"ahler metric.
Indeed, if we write
$$f^*\omega_Y(y,z)=\mn\sum_{\alpha,\beta=1}^m (g_Y)_{\alpha\ov{\beta}}(y)dy_\alpha\wedge d\ov{y}_\beta,$$
\[\begin{split}
\omega_0(y,z)&=\mn\sum_{\alpha,\beta=1}^m (g_0)_{\alpha\ov{\beta}}(y,z)dy_\alpha\wedge d\ov{y}_\beta+2\mathrm{Re}\left(\mn\sum_{\alpha=1}^m\sum_{i=1}^{n-m} (g_0)_{\alpha\ov{i}}(y,z)dy_\alpha\wedge d\ov{z}_i\right)\\
&+\mn\sum_{i,j=1}^{n-m} (g_0)_{i\ov{j}}(y,z)dz_i\wedge d\ov{z}_j,
\end{split}\]
then we have
\[\begin{split}
F_t^*(e^tf^*\omega_Y+\omega_0)(y,z)&=\mn\sum_{\alpha,\beta=1}^m ((g_Y)_{\alpha\ov{\beta}}(ye^{-t/2})+e^{-t}(g_0)_{\alpha\ov{\beta}}(ye^{-t/2},z))dy_\alpha\wedge d\ov{y}_\beta\\
&+2e^{-t/2}\mathrm{Re}\left(\mn\sum_{\alpha=1}^m\sum_{i=1}^{n-m} (g_0)_{\alpha\ov{i}}(ye^{-t/2},z)dy_\alpha\wedge d\ov{z}_i\right)\\
&+\mn\sum_{i,j=1}^{n-m} (g_0)_{i\ov{j}}(ye^{-t/2},z)dz_i\wedge d\ov{z}_j
\end{split}\]
which converges smoothly on compact subsets of $\mathbb{C}^{m}\times B$ to
$$\mn\sum_{\alpha,\beta=1}^m (g_Y)_{\alpha\ov{\beta}}(0)dy_\alpha\wedge d\ov{y}_\beta+\mn\sum_{i,j=1}^{n-m} (g_0)_{i\ov{j}}(0,z)dz_i\wedge d\ov{z}_j,$$
which is a smooth K\"ahler metric.
This implies that
$$C^{-1}\omega_E\leq \omega_t(s)\leq C\omega_E,$$
for all $t\geq 0, -1\leq s\leq 0$, where $\omega_E$ is a Euclidean metric on $\mathbb{C}^{m}\times B$. We can therefore apply the local higher order estimates in Theorem \ref{higher} (note that the coefficient $e^{-t}$ of $e^{-t}\omega_t(s)$ in the evolution of $\omega_t(s)$ is uniformly bounded) and obtain that for every compact set $K\subset \mathbb{C}^m\times B$ there are constants $C_k$ such that
$$\|\omega_t(s)\|_{C^k(K,g_E)}\leq C_k,$$
for all $t\geq 0, -\frac{1}{2}\leq s\leq 0$. Setting $s=0$ we obtain
$$\|e^tF_t^*\omega(t)\|_{C^k(K,g_E)}\leq C_k,$$
and since $F_t$ is the identity when restricted to $\{0\}\times B$, which is identified with $X_{y_0}\cap U$, we obtain \eqref{fibcinf} after a simple covering argument.
\end{proof}

We will also need the following elementary result:

\begin{lemma}\label{triv}
Let $F:X\times [0,\infty)\to\mathbb{R}$ be a smooth function such that
\begin{equation}\label{un}
|\nabla (F|_{X_y})|_{g_0|_{X_y}}\leq C,
\end{equation}
for all $y\in Y, t\geq 0$, such that
\begin{equation}\label{du}
\int_{X_y} (F|_{X_y}) \omega_{\rm SRF}^{n-m}=0,
\end{equation}
for all $y\in Y, t\geq 0$, and such that
\begin{equation}\label{tr}
\sup_X F(x,t)\leq h(t),
\end{equation}
for all $t\geq 0,$ where $h(t)$ is a positive function with $h(t)\to 0$ as $t\to \infty$. Then we have
\begin{equation}\label{qu}
\sup_X |F(x,t)|\leq Ch(t)^{\frac{1}{2n+1}},
\end{equation}
for all $t$ sufficiently large.
\end{lemma}
\begin{proof}
Thanks to \eqref{tr}, it is enough to show that
$$\inf_X F(x,t)\geq -Ch(t)^{\frac{1}{2n+1}}.$$
If this fails, then we can find $t_k\to\infty$ and $x_k\in X$ such that
$$F(x_k,t_k)\leq -kh(t_k)^{\frac{1}{2n+1}}.$$
If we let $y_k=f(x_k)$, then thanks to \eqref{un} we have that for all $x$ in the $g_0|_{X_{y_k}}$-geodesic ball $B_r(x_k)$ in $X_{y_k}$ centered at $x_k$, of radius
$$r=\min\left(\frac{kh(t_k)^{\frac{1}{2n+1}}}{2C},\frac{1}{2C}\right)\leq \frac{1}{2C},$$
we have
$$F(x,t_k)\leq -\frac{kh(t_k)^{\frac{1}{2n+1}}}{2},$$
and so using \eqref{du}, \eqref{tr} we get
$$0=\int_{X_{y_k}} F(x,t_k) \omega_{\rm SRF}^{n-m}(x)\leq -\frac{kh(t_k)^{\frac{1}{2n+1}}}{2}\int_{B_r(x_k)}\omega_{\rm SRF}^{n-m}+Ch(t_k).$$
But the metrics $\omega_{\rm SRF}|_{X_{y_k}}$ are all uniformly equivalent to each other, and since $r\leq \frac{1}{2C}$ we have
$$\int_{B_r(x_k)}\omega_{\rm SRF}^{n-m}\geq C^{-1}r^{2n}\geq \min\left(C^{-1}k^{2n}h(t_k)^{\frac{2n}{2n+1}},C^{-1}\right),$$
and so we obtain that either
$k^{2n+1}\leq C,$ or $k\leq Ch(t_k)^{\frac{2n}{2n+1}}$, both of which are impossible for $k$ large.
\end{proof}

We can now prove \eqref{conve2p}.
\begin{theorem}\label{trk}
For any given $y\in Y$ we have
\begin{equation}\label{conve2p1}
e^t\omega(t)|_{X_y}\to \omega_y,
\end{equation}
in $C^\infty(X_y)$, where $\omega_y$ is the unique Ricci-flat K\"ahler metric on $X_y$ in the class $[\omega_0]|_{X_y}$. The convergence in the $C^0$ norm is exponentially fast.
\end{theorem}
\begin{proof}
We compute
\[\begin{split}
\frac{(e^t\omega(t)|_{X_y})^{n-m}}{(\omega_{\rm SRF}|_{X_y})^{n-m}}&=e^{(n-m)t}\frac{(\omega(t)|_{X_y})^{n-m}}{(\omega_{\rm SRF}|_{X_y})^{n-m}}\\
&=e^{(n-m)t}\frac{\omega(t)^{n-m}\wedge f^*\omega_Y^m}{\omega_{\rm SRF}^{n-m}\wedge f^*\omega_Y^m}\\
&=e^{(n-m)t}\binom{n}{m}\frac{\omega(t)^{n-m}\wedge f^*\omega_Y^m}{\Omega}\\
&=e^{\vp(t)+\dot{\vp}(t)}\binom{n}{m}\frac{\omega(t)^{n-m}\wedge f^*\omega_Y^m}{\omega(t)^n},
\end{split}\]
and so the function $F:X\times[0,\infty)\to\mathbb{R}$ defined by
$$F=e^{\vp(t)+\dot{\vp}(t)}\binom{n}{m}\frac{\omega(t)^{n-m}\wedge f^*\omega_Y^m}{\omega(t)^n},$$
satisfies
$$F|_{X_y}=\frac{(e^t\omega(t)|_{X_y})^{n-m}}{(\omega_{\rm SRF}|_{X_y})^{n-m}},$$
and so
$$\int_{X_y}(F|_{X_y}) \omega_{\rm SRF}^{n-m}=\int_{X_y}(e^t\omega(t)|_{X_y})^{n-m}=\int_{X_y} \omega_{\rm SRF}^{n-m},$$
so $F-1$ satisfies \eqref{du}. It also satisfies \eqref{un} thanks to \eqref{fibc3}. Now, thanks to Lemma \ref{c0} we have that
\begin{equation}\label{dec}
|e^{\vp(t)+\dot{\vp}(t)}-1|\leq Ce^{-\frac{t}{4}}.
\end{equation}
Choosing local coordinates at a point $x\in X_y$, so that at that point $\omega(t)$ is the identity and $f^*\omega_Y$ is diagonal with eigenvalues $(\lambda_1,\dots,\lambda_m,0,\dots,0),$ then at this point
$$\binom{n}{m}\frac{\omega(t)^{n-m}\wedge f^*\omega_Y^m}{\omega(t)^n}=\prod_{j=1}^m\lambda_j\leq \left(\frac{\sum_{j=1}^m\lambda_j}{m}\right)^m=\left(\frac{\tr{\omega(t)}{(f^*\omega_Y)}}{m}\right)^m,$$
and so \eqref{base2} gives
$$\binom{n}{m}\frac{\omega(t)^{n-m}\wedge f^*\omega_Y^m}{\omega(t)^n}\leq\left(\frac{\tr{\omega(t)}{(f^*\omega_Y)}}{m}\right)^m\leq 1+Ce^{-\eta t},$$
for some uniform $\eta>0$. Combining this with \eqref{dec} gives that
$$F\leq 1+Ce^{-\eta t},$$
everywhere on $X\times[0,\infty)$, which verifies \eqref{tr}. Therefore Lemma \ref{triv} gives us
$$|F-1|\leq Ce^{-\eta t},$$
for some smaller $\eta>0$, i.e.
\begin{equation}\label{dov}
\|(e^t\omega(t)|_{X_y})^{n-m}-(\omega_{\rm SRF}|_{X_y})^{n-m}\|_{C^0(X_y,\omega_0|_{X_y})}\leq Ce^{-\eta t},
\end{equation}
for all $y\in Y$. Next, we compute
\[\begin{split}
(e^t\omega(t)|_{X_y})&\wedge(\omega_{\rm SRF}|_{X_y})^{n-m-1} =e^{t}\frac{\omega(t)\wedge \omega_{\rm SRF}^{n-m-1}\wedge f^*\omega_Y^m}{\omega_{\rm SRF}^{n-m}\wedge f^*\omega_Y^m}(\omega_{\rm SRF}|_{X_y})^{n-m}\\
&=e^{\vp(t)+\dot{\vp}(t)}\binom{n}{m}\frac{\omega(t)\wedge (e^{-t}\omega_{\rm SRF})^{n-m-1}\wedge f^*\omega_Y^m}{\omega(t)^n}(\omega_{\rm SRF}|_{X_y})^{n-m},
\end{split}\]
so the smooth function
$G:X\times[0,\infty)\to\mathbb{R}$ defined by
$$G=e^{\vp(t)+\dot{\vp}(t)}\binom{n}{m}\frac{\omega(t)\wedge (e^{-t}\omega_{\rm SRF})^{n-m-1}\wedge f^*\omega_Y^m}{\omega(t)^n},$$
satisfies
$$G|_{X_y}=\frac{(e^t\omega(t)|_{X_y})\wedge(\omega_{\rm SRF}|_{X_y})^{n-m-1}}{(\omega_{\rm SRF}|_{X_y})^{n-m}},$$
and the arithmetic-geometric mean inequality gives
$$(G|_{X_y})^{n-m}\geq \frac{(e^t\omega(t)|_{X_y})^{n-m}}{(\omega_{\rm SRF}|_{X_y})^{n-m}},$$
and the RHS converges to $1$ exponentially fast thanks to \eqref{dov}. Therefore $1-G$ satisfies \eqref{tr}, and it also satisfies
\eqref{du} (as is simple to verify) and \eqref{un}, thanks to \eqref{fibc3}. Another application of
Lemma \ref{triv} gives us
$$|1-G|\leq Ce^{-\eta t},$$
for some $\eta>0$, i.e.
\begin{equation}\label{dov2}
\|(e^t\omega(t)|_{X_y})\wedge(\omega_{\rm SRF}|_{X_y})^{n-m-1} -(\omega_{\rm SRF}|_{X_y})^{n-m}\|_{C^0(X_y,\omega_0|_{X_y})}\leq Ce^{-\eta t},
\end{equation}
for all $y\in Y$. Therefore if we choose local coordinates along a fiber $X_y$ such that at a given point $\omega_{\rm SRF}|_{X_y}$ is the identity and $e^t\omega(t)|_{X_y}$ is a positive-definite Hermitian matrix $A$, then \eqref{dov} and \eqref{dov2} imply that
$${\rm tr}A\leq n+Ce^{-\eta t},\quad \det A \geq 1-Ce^{-\eta t},$$
and so Lemma \ref{matrix} gives
$$\|A-\mathrm{Id}\|\leq Ce^{-\frac{\eta}{2} t},$$
which implies
$$\|e^t\omega(t)|_{X_y}-\omega_{\rm SRF}|_{X_y}\|_{C^0(X_y,\omega_0|_{X_y})}\leq Ce^{-\frac{\eta}{2} t},$$
for all $y\in Y$ and $t\geq 0$, so the metrics $e^t\omega(t)|_{X_y}$ converge to $\omega_{\rm SRF}|_{X_y}$ exponentially fast. The convergence is smooth thanks to Theorem \ref{fibercinf}.
\end{proof}
\subsection{Completion of the proof of Theorem \ref{colla}}
As we showed earlier, to complete the proof of \eqref{conve1p} it is enough to prove Theorem \ref{baseb}, which we can now do:
\begin{proof}[Proof of Theorem \ref{baseb}]
Recall that by definition
$$\hat{\omega}_t=e^{-t}\omega_{\rm SRF}+(1-e^{-t})f^*\omega_Y,$$
and that thanks to Lemma \ref{base} we have
$$\tr{\omega(t)}{(f^*\omega_Y)}\leq m+Ce^{-\frac{t}{8}}.$$
It follows that to prove \eqref{baseb2} it is enough to show that
\begin{equation}\label{trx}
\tr{\omega(t)}{(e^{-t}\omega_{\rm SRF})}\leq n-m+Ce^{-\eta t}.
\end{equation}
To this end, fix a point $x\in X$ and let $y=f(x)$, and choose local product coordinates near these points.
At the point $x$ we can then consider the $(1,1)$ form $\omega_{\rm SRF}|_{X_y}$ as defined for all tangent vectors to $X$ at $x$ (not just those tangent to the fiber $X_y$) by using the obvious projection in these coordinates, so it makes sense to estimate
$$\tr{\omega(t)}{(e^{-t}\omega_{\rm SRF}|_{X_y})}=\tr{(e^{t}\omega(t)|_{X_y})}{(\omega_{\rm SRF}|_{X_y})}\leq n-m+Ce^{-\eta t},$$
thanks to Theorem \ref{trk}.
Lastly, we need to estimate the difference
$$\tr{\omega(t)}{(e^{-t}\omega_{\rm SRF}-e^{-t}\omega_{\rm SRF}|_{X_y})},$$
and to do this we write in local product coordinates at $x$
$$\omega_{\rm SRF}-\omega_{\rm SRF}|_{X_y}=\mn\sum_{\alpha,\beta=1}^m h_{\alpha\ov{\beta}}dy_\alpha\wedge d\ov{y}_\beta+2\mathrm{Re}\left(
\mn\sum_{\alpha=1}^m\sum_{j=1}^{n-m} h_{\alpha\ov{j}}dy_\alpha\wedge d\ov{z}_j\right),$$
where we use greek indices for the base coordinates and latin indices for the fiber coordinates.
The term involving $h_{j\ov{k}}$ is not present because $\omega_{\rm SRF}-\omega_{\rm SRF}|_{X_y}$ vanishes when restricted to $X_y$.
Therefore
$$\tr{\omega(t)}{(e^{-t}\omega_{\rm SRF}-e^{-t}\omega_{\rm SRF}|_{X_y})}\leq \left|g^{\alpha\ov{\beta}}h_{\alpha\ov{\beta}}\right|+2
\left|g^{\alpha\ov{j}}h_{\alpha\ov{j}}\right|\leq Ce^{\frac{t}{2}},$$
because thanks to \eqref{equiv3} the terms $g^{\alpha\ov{\beta}}$ are uniformly bounded, and the terms $g^{\alpha\ov{j}}$ are bounded by $Ce^{\frac{t}{2}}$ by Cauchy-Schwarz (since $g^{j\ov{k}}$ is of the order of $e^t$). This completes the proof of \eqref{trx}.
\end{proof}

Lastly, to complete the proof of Theorem \ref{colla}, we need to show that $(X,\frac{\omega(t)}{t})$ converge to $(Y,\omega_Y)$ in the Gromov-Hausdorff topology as $t\to\infty$.
Recall that since $f$ is a submersion everywhere, Ehresmann's Theorem \cite[Theorem 2.4]{Ko} implies that $f$ is a smooth fiber bundle. Then
the Gromov-Hausdorff convergence follows from \eqref{conve1p} and the following (cf. \cite[Lemma 9.1]{TWY0}):

\begin{theorem}\label{gh} Let $\pi:M\to B$ be a smooth fiber bundle, where $(M,g_M)$ and $(B,g_B)$ are closed Riemannian manifolds.
If $g(t), t\geq 0,$ is a family of Riemannian metrics on $M$ with $\|g(t)-\pi^*g_B\|_{C^0(M,g_M)}\to 0$ as $t\to\infty$, then
$(M,g(t))$ converges to $(B,g_B)$ in the Gromov-Hausdorff sense as $t\to\infty$.
\end{theorem}
\begin{proof}
For any $y\in B$ we denote by $E_y=\pi^{-1}(y)$ the fiber over $y$. Fix $\ve>0$, denote by $L_t$ the length of a curve in $M$ measured with respect to $g(t)$,
and by $d_t$ the induced distance function on $M$. Similarly we have $L_B, d_B$ on $B$.
Let $F=\pi:M\to B$ and define a map $G:B\to M$ by sending every point $y\in B$ to some chosen point in $M$ on the fiber $E_y$. The map $G$ will in general be discontinuous, and it
satisfies $F\circ G=\mathrm{Id}$, so
\begin{equation}\label{gh1}
d_B(y,F(G(y)))=0.
\end{equation}
On the other hand since $g(t)|_{E_y}$ goes to zero, we have that for any $t$ large and for any $x\in M$
\begin{equation}\label{gh2}
d_t(x, G(F(x)))\leq \ve.
\end{equation}
Next, given two points $x_1,x_2\in M$ let $\gamma:[0,L]\to B$ be a unit-speed minimizing geodesic in $B$ joining $F(x_1)$ and $F(x_2)$.
Since the bundle $\pi$ is locally trivial, we can cover the image of $\gamma$ by finitely many open sets $U_j, 1\leq j\leq N,$  such that
$\pi^{-1}(U_j)$ is diffeomorphic to $U_j\times E$ (where $E$ is the fiber of the bundle) and there is a subdivision $0=t_0<t_1<\dots<t_N=L$ of $[0,L]$ such that $\gamma([t_{j-1},t_j])\subset U_j$.
Fix a point $e\in E$, and use the trivializations to define $\ti{\gamma}_j(s)=(\gamma(s),e)$, for $s\in [t_{j-1},t_j]$, which are curves in $M$ with the
property that
$$|L_t(\ti{\gamma}_j)-L_B(\gamma|_{[t_{j-1},t_j]})|\leq \ve/N,$$
as long as $t$ is sufficiently large (because $g(t)\to\pi^*g_B$).
The points $\ti{\gamma}_{j}(t_j)$ and $\ti{\gamma}_{j+1}(t_j)$ lie in the same fiber of $\pi$, so we can join them by a curve contained in
this fiber with $L_t$-length at most $\ve/2N$ (for $t$ large). We also join $x_1$ with $\ti{\gamma}_1(0)$ and
$x_2$ with $\ti{\gamma}_N(L)$ in the same fashion. Concatenating these ``vertical'' curves and the curves $\ti{\gamma}_j$, we obtain a piecewise
smooth curve $\ti{\gamma}$ in $M$ joining $x_1$ and $x_2$, with $\pi(\ti{\gamma})=\gamma$ and
$|L_t(\ti{\gamma})-d_B(F(x_1),F(x_2))|\leq 2\ve.$ Therefore,
\begin{equation}\label{gh3}
d_t(x_1,x_2)\leq L_t(\ti{\gamma})\leq d_B(F(x_1),F(x_2))+2\ve.
\end{equation}
Since $F\circ G=\mathrm{Id},$ we also have that for all $t$ large and for all $y_1, y_2\in B$,
\begin{equation}\label{gh4}
d_t(G(y_1),G(y_2))\leq d_B(y_1,y_2)+2\ve.
\end{equation}
Given now two points $x_1,x_2\in M$, let $\gamma$ be a unit-speed minimizing $g(t)$-geodesic joining them. If we denote by $L_{\pi^*g_B}(\gamma)$ the length of $\gamma$
using the degenerate metric $\pi^*g_B$, then we have for $t$ large,
\begin{equation}\label{gh5}
d_B(F(x_1),F(x_2))\leq L_B(F(\gamma))=L_{\pi^*g_B}(\gamma)\leq L_t(\gamma)+\ve =d_t(x_1,x_2)+\ve,
\end{equation}
where we used again that $g(t)\to\pi^*g_B$. Obviously this also implies that for all $t$ large and for all $y_1, y_2\in B$,
\begin{equation}\label{gh6}
d_B(y_1,y_2)\leq d_t(G(y_1),G(y_2))+\ve.
\end{equation}
Combining \eqref{gh1}, \eqref{gh2}, \eqref{gh3}, \eqref{gh4}, \eqref{gh5} and \eqref{gh6} we get the required Gromov-Hausdorff convergence.
\end{proof}

\subsection{Smooth collapsing when the general fibers are tori}
Having completed the proof of Theorem \ref{colla}, we now show under the same assumptions that if we assume that the generic fiber $X_y$ of $f$ is biholomorphic to the quotient of a complex torus by a holomorphic free action of a finite group, then the collapsing in \eqref{conve1} is in the smooth topology. More precisely, we show:

\begin{theorem}[\cite{FZ,Gi,GTZ,HT,TZ}]\label{smooth}
Let $(X,\omega_0)$ be a compact K\"ahler manifold with $K_X$ semiample and $0<\kappa(X)<n$, and let $f:X\to Y$ be the fibration as in Theorem \ref{colla}, and assume that for some $y\in Y\backslash S'$ the fiber $X_y=f^{-1}(y)$ is biholomorphic to a finite quotient of a torus. Let $\omega(t),t\in[0,\infty)$ be the solution of the K\"ahler-Ricci flow \eqref{krf} starting at $\omega_0$. Then as $t\to\infty$ we have
\begin{equation}\label{conve3}
\frac{\omega(t)}{t}\to f^*\omega_Y,
\end{equation}
in $C^\infty_{\mathrm{loc}}(X\backslash S),$ where $\omega_Y$ is the same K\"ahler metric on $Y\backslash S'$ as in Theorem \ref{colla}. Furthermore, the metrics $\frac{\omega(t)}{t}$ have locally uniformly bounded curvature tensor on compact sets of $X\backslash S$.
\end{theorem}

This theorem was proved in \cite{FZ} under the assumption that $X_y$ is biholomorphic to a torus, that $X$ is projective, and the initial class $[\omega_0]$ is in $H^2(X,\mathbb{Q})$, by adapting to this parabolic setting the proof of a similar result for the elliptic complex Monge-Amp\`ere equation in \cite{GTZ}. In the case when $X=Y\times F$ where $c_1(Y)<0$ and $F$ is a finite quotient of a torus, this theorem was proved in \cite{Gi}. The projectivity and rationality assumptions in \cite{FZ} were removed in \cite{HT}, and finally the case when $X_y$ is a finite quotient of a torus was dealt with in \cite{TZ}. We will give a unified treatment of these results, following \cite{GTZ,HT,TZ}.

It is natural to conjecture that in the general setting of Theorem \ref{colla} (i.e. when the fibers $X_y$ are general Calabi-Yau manifolds) the smooth convergence in \eqref{conve3} still holds. On the other hand, the local uniform boundedness of the curvature of $\frac{\omega(t)}{t}$ is false when $X_y$ is not a quotient of a torus. Indeed thanks to \eqref{conve2} the metrics $\omega(t)|_{X_y}$ converge smoothly to $\omega_{\rm SRF}|_{X_y}$, the unique Ricci-flat K\"ahler metric on $X_y$ in the class $[\omega_0|_{X_y}].$ But the metric $\omega_{\rm SRF}|_{X_y}$ is not flat, since otherwise $X_y$ would be a finite quotient of a torus by \cite[Corollary V.4.3]{KN} and \cite[Theorem IX.7.9]{KN2}. It follows easily that the largest bisectional curvature of $\omega_{\rm SRF}|_{X_y}$ (among unit vectors) is strictly positive, and so the same is true for $\omega(t)|_{X_y}$ for all $t$ large. Since the bisectional curvature decreases in submanifolds, the same is also true for $\omega(t)$ (at points on $X_y$), and so the maximum of the curvature of $\frac{\omega(t)}{t}$ on $X_y$ blows up to infinity as $t\to\infty.$

\begin{proof}
As in the proof of Theorem \ref{colla} we assume that $\omega(t)$ satisfies instead the normalized flow \eqref{krf2}.
The statements that we need to prove are local on the base $Y\backslash S'$, so it is enough to prove that for every sufficiently small open subset $B\subset Y\backslash S'$, given any $k\geq 0$ there are constants $C_k>0$ such that on the preimage $U=f^{-1}(B)$ we have
\begin{equation}\label{cinfty}
\|\omega(t)\|_{C^k(U,g_0)}\leq C_{k},
\end{equation}
and
\begin{equation}\label{curvv}
\sup_U |\mathrm{Rm}(\omega(t))|_{\omega(t)}\leq C_0,
\end{equation}
for all $t\geq 0$.
Let us first give the proof of these in the case when $X_y$ is in fact biholomorphic to a torus for some $y\in Y\backslash S'$. Then, using Ehresmann's Theorem \cite[Theorem 2.4]{Ko} (which gives that $f$ is a locally trivial smooth fiber bundle over $Y\backslash S'$) and the fact $Y\backslash S'$ is connected, we immediately conclude that all fibers $X_y,y\in Y\backslash S'$ are diffeomorphic to a torus. But a compact K\"ahler manifold which is diffeomorphic to a torus must be in fact biholomorphic to a torus, as follows easily using the Albanese map, and we conclude that all fibers $X_y,y\in Y\backslash S'$ are biholomorphic to tori, say $X_y\cong\mathbb{C}^{n-m}/\Lambda_y$, where $\Lambda_y$ is a lattice in $\mathbb{C}^{n-m}$. Since $f$ is a holomorphic submersion over $Y\backslash S'$, we may choose a basis $v_1(y),\dots, v_{2n-2m}(y)$ of the lattice $\Lambda_y$ which varies holomorphically in $y\in B$, for any sufficiently small $B\subset Y\backslash S'$. We can then construct another family $f'$ of tori over $B$, by taking the quotient of $B\times\mathbb{C}^{n-m}$ by the holomorphic free $\mathbb{Z}^{2n-2m}$ action given by
$$(\ell_1,\dots,\ell_{2n-2m})\cdot (y,z)=\left(y,z+\sum_{i=1}^{2n-2m} \ell_i v_i(y)\right),$$
where $y\in B, z\in \mathbb{C}^{n-m}$ and $\ell_i\in\mathbb{Z}$. Note that while the choice of the generating vectors $v_i(y)$ is not unique, the quotient does not depend on this choice. This gives us a holomorphic submersion $f':U'\to B$ with fiber $f'^{-1}(y)$ biholomorphic to $X_y$, for all $y\in B$. A theorem of Wehler \cite[Satz 3.6]{Weh} then shows that the families $f$ and $f'$ are locally isomorphic, so up to shrinking $B$ there is a biholomorphism $U'\to U$, which is compatible with the projections to $B$. Composing the quotient map
$B\times\mathbb{C}^{n-m}$ with this biholomorphism, we obtain a local biholomorphism $p:B\times\mathbb{C}^{n-m}\to U$ such that $f\circ p (y,z)=y$ for all $(y,z)$. The map $p$ is thus the universal covering of $U$.

The following is the key tool we need:

\begin{proposition}[\cite{GSVY, GTZ,HT}]\label{semif}
Up to shrinking $B$, on $U=f^{-1}(B)$ there is a closed semipositive definite real $(1,1)$ form $\omega_{\rm SF}$ which is semi-flat in the sense that $\omega_{\rm SF}|_{X_y}$ a flat K\"ahler metric on $X_y$ for all $y\in B$, and such that $p^*\omega_{\rm SF}=\ddbar\eta$ where $\eta\in C^\infty(B\times\mathbb{C}^{n-m},\mathbb{R})$ satisfies
\begin{equation}\label{scaling}
\eta(y,\lambda z)=\lambda^2\eta(y,z),
\end{equation}
for all $(y,z)\in B\times\mathbb{C}^{n-m}$ and $\lambda\in\mathbb{R}$.
\end{proposition}

This was proved in \cite[Section 3]{GTZ} when $X$ is projective and $[\omega_0]$ is rational, following the recipe in \cite{GSVY}, and was then proved in \cite{HT} in general. We will not prove this here, but just say that the function $\eta$ is given explicitly by
$$\eta(y,z)=-\frac{1}{4}\sum_{i,j=1}^{n-m} \left(\mathrm{Im}Z(y)\right)^{-1}_{ij} (z_i-\ov{z}_i)(z_j-\ov{z}_j),$$
where $Z$ is a holomorphic period map from $B$ to the Siegel upper half space $\mathfrak{H}_{n-m}$ of symmetric $(n-m)\times (n-m)$ complex matrices with positive definite imaginary part (so $\left(\mathrm{Im}Z(y)\right)^{-1}$ in this formula is well-defined). The key reason why this can be done is that $\mathfrak{H}_{n-m}$ classifies complex tori which are polarized by a K\"ahler class. We refer the reader to \cite{GTZ,HT} for the details of the construction of $Z$ (which is easier under the rationality assumption) and of why this $\eta$ satisfies our requirements.

Now, recall than thanks to Proposition \ref{equiv2} (or rather its generalization to the case when $S\not =\emptyset$) we have that
$$C^{-1}(e^{-t}\omega_{\rm SF}+f^*\omega_Y)\leq \omega(t)\leq C(e^{-t}\omega_{\rm SF}+f^*\omega_Y),$$
where we used that on $U$ we have that $\omega_{\rm SF}+f^*\omega_Y$ is a K\"ahler metric. For $t\geq 0$ let
$\lambda_t:B\times\mathbb{C}^{n-m}\to B\times\mathbb{C}^{n-m}$ be given by
$$\lambda_t(y,z)=(y,ze^{t/2}),$$
which is a ``stretching in the fiber directions'' (compare this with the maps $F_t$ in Theorem \ref{fibercinf} which were instead shrinking the base directions).
Then the metrics
$$\omega_t(s):=\lambda_t^*p^*\omega(s+t),\quad -1\leq s\leq 0,$$
on $B\times\mathbb{C}^{n-m}$ satisfy
$$C^{-1}(e^{-t}\lambda_t^*p^*\omega_{\rm SF}+\lambda_t^*p^*f^*\omega_Y)\leq \omega_t(s)\leq C(e^{-t}\lambda_t^*p^*\omega_{\rm SF}+\lambda_t^*p^*f^*\omega_Y),$$
for all $t\geq 0, -1\leq s\leq 0$, and
$$\frac{\de}{\de s}\omega_t(s)=-\Ric(\omega_t(s))-\omega_t(s), \quad -1\leq s\leq 0.$$
But we have that $f\circ p\circ\lambda_t=f\circ p$, so $\lambda_t^*p^*f^*\omega_Y=p^*f^*\omega_Y$, and
$$\lambda_t^*p^*\omega_{\rm SF}=\lambda_t^*\ddbar\eta=\ddbar(\eta\circ\lambda_t)=e^t\ddbar\eta=e^t p^*\omega_{\rm SF},$$
since
$$\eta\circ\lambda_t(y,z)=\eta(y,ze^{t/2})=e^t\eta(y,z),$$
thanks to \eqref{scaling}. Therefore we conclude that on $B\times\mathbb{C}^{n-m}$ we have
$$C^{-1}p^*(\omega_{\rm SF}+f^*\omega_Y)\leq \omega_t(s)\leq Cp^*(\omega_{\rm SF}+f^*\omega_Y),$$
and so for each given compact set $K\subset B\times\mathbb{C}^{n-m}$ there is a constant $C_K$ such that on $K$ we have
$$C^{-1}_K\omega_E\leq \omega_t(s)\leq C_K\omega_E,$$
for all $t\geq 0, -1\leq s\leq 0$, where $\omega_E$ is a Euclidean metric on $B\times\mathbb{C}^{n-m}$.
We can therefore apply the local higher order estimates in Theorem \ref{higher} and obtain that for every compact set $K\subset B\times\mathbb{C}^{n-m}$ there are constants $C_{K,k}$ such that
$$\|\omega_t(s)\|_{C^k(K,g_E)}\leq C_{K,k},$$
for all $t\geq 0, -\frac{1}{2}\leq s\leq 0$. Setting $s=0$ we obtain
\begin{equation}\label{cinfty2}
\|\lambda_t^*p^*\omega(t)\|_{C^k(K,g_E)}\leq C_{K,k},
\end{equation}
and we still have
$$\lambda_t^*p^*\omega(t)\geq C^{-1}\omega_E,$$
on $K\times[0,\infty)$.
In particular, this gives
\begin{equation}\label{curvv2}
\sup_K |\mathrm{Rm}(\lambda_t^*p^*\omega(t))|_{\lambda_t^*p^*\omega(t)}\leq C,
\end{equation}
for all $t\geq 0$. If now $K'\subset U\subset X\backslash S$ is a compact set which is small enough so that $K=p^{-1}(K')\subset B\times\mathbb{C}^{n-m}$ is compact and $p$ is a biholomorphism on $K$ (note that such compact sets $K'$ cover $U$) then we have
$$\sup_{K'}|\mathrm{Rm}(\omega(t))|_{\omega(t)}=\sup_K |\mathrm{Rm}(p^*\omega(t))|_{p^*\omega(t)}=\sup_{\lambda_{1/t}(K)} |\mathrm{Rm}(\lambda_t^*p^*\omega(t))|_{\lambda_t^*p^*\omega(t)},$$
where $\lambda_{1/t}$ is the inverse map of $\lambda_t$. But the compact sets $\lambda_{1/t}(K)$ are all contained in a fixed compact set of $B\times\mathbb{C}^{n-m}$, and so
from \eqref{curvv2} and a covering argument we easily obtain \eqref{curvv}. Also, \eqref{cinfty2} easily implies that
$$\|p^*\omega(t)\|_{C^k(K,g_E)}\leq C_{K,k},$$
for any given compact set $K\subset B\times\mathbb{C}^{n-m}$ (in fact, \eqref{cinfty2} is a much stronger bound). Since $p$ is a local biholomorphism, this (and another covering argument) proves \eqref{cinfty}, and completes the proof of Theorem \ref{smooth} when $X_y$ is a torus.

If now the fiber $X_y$ is just biholomorphic to a finite quotient of a torus, for some $y\in Y\backslash S'$ (and therefore for all such $y$, by the same argument as before using Ehresmann's Theorem), then we choose again a sufficiently small open set $B\subset Y\backslash S'$ such that $f$ is a locally trivial smooth fiber bundle over $B$, and so there is a diffeomorphism $\Psi:B\times F\to f^{-1}(B)$, compatible with the projections to $B$, where $F$ is the smooth manifold underlying $X_y$. We use $\Psi$ to pull back the complex structure on $f^{-1}(B)$ to a complex structure $J$ on $B\times F$, which is in general different from the product complex structure on $B\times X_y$. This way, the map $\Psi$ becomes a biholomorphism (where here and in the following we always use the complex structure $J$ on $B\times F$). If we now let $\ti{F}\to F$ be a smooth finite covering map with $\ti{F}$ a torus, then the map
$p:B\times\ti{F}\to B\times F$ is a smooth finite covering (hence a local diffeomorphism), and so we can use it to pull back the complex structure $J$ to a complex structure $\ti{J}$ on $B\times\ti{F}$. This way $p$ is also a local biholomorphism, and so pulling back a K\"ahler metric on $f^{-1}(B)$ via $\Psi\circ p$ we obtain a compatible K\"ahler metric on $B\times \ti{F}$. Then the projection $\pi:B\times \ti{F}$ is by construction equal to $f\circ \Psi\circ p$, and so it is holomorphic, and clearly a proper submersion. This implies that its fibers
$\ti{X}_y=\pi^{-1}(y)$ are all compact complex submanifolds of $B\times \ti{F}$ (with the complex structure $\ti{J}$), and so they are also K\"ahler, and each $\ti{X}_y$ is diffeomorphic to the torus $\ti{F}$. As remarked earlier, this implies that all fibers $\ti{X}_y$ are in fact biholomorphic to complex tori $\mathbb{C}^{n-m}/\Lambda_y$. Therefore the family $\pi$ over $B$ has torus fibers, and pulling back the solution $\omega(t)$ of the K\"ahler-Ricci flow via the holomorphic finite covering map $\Psi\circ p$ we obtain a solution
$p^*\Psi^*\omega(t)$ of the K\"ahler-Ricci flow on $B\times\ti{F}$. We can then apply Proposition \ref{semif} to $B\times \ti{F}$ and get a semi-flat form $\omega_{\rm SF}$ with the same properties, and from Proposition \ref{equiv2} we again have
$$C^{-1}(e^{-t}\omega_{\rm SF}+\pi^*\omega_Y)\leq p^*\Psi^*\omega(t)\leq C(e^{-t}\omega_{\rm SF}+\pi^*\omega_Y),$$
on $B\times\ti{F}$ for all $t\geq 0$. Then the rest of the argument above goes through, and we obtain \eqref{cinfty} and \eqref{curvv} on $B\times\ti{F}$ for the metrics
$p^*\Psi^*\omega(t)$. Since $\Psi\circ p$ is a holomorphic finite covering, these estimates immediately imply those for $\omega(t)$ on $f^{-1}(B)$.
\end{proof}

\section{Some open problems}\label{sectopen}

In this closing section, we collect some well-known open problems on the K\"ahler-Ricci flow (in addition to the conjectures that we have already discussed in Section \ref{sectfin}),
related to the material discussed in these notes.\\

\subsection{Diameter bounds}

Diameter bounds for solutions of the K\"ahler-Ricci flow as we approach a singularity are not easy to get. In general we expect:

\begin{conjecture}\label{diam}
Let $X$ be a compact K\"ahler manifold and $\omega(t)$ a solution of the K\"ahler-Ricci flow \eqref{krf} defined on a maximal time interval $[0,T)$ with $T<\infty$. Then there is a constant $C>0$ such that
$$\diam(X,\omega(t))\leq C,$$
for all $t\in [0,T)$.
\end{conjecture}

This conjecture is known in the case when the limiting class $[\alpha]=[\omega_0]-2\pi Tc_1(X)$ is equal to $\pi^*[\omega_Y]$ where $\pi:X\to Y$ is the blowup of a compact K\"ahler manifold $Y$ at finitely many distinct point and $\omega_Y$ is a K\"ahler metric on $Y$, thanks to \cite{SW}, and this is in fact the general case when $n=2$ and the singularity is noncollapsed. The conjecture is also known when $X$ is Fano and $[\omega_0]=\lambda c_1(X)$ for some $\lambda>0$, since as we mentioned earlier Perelman proved that in this case $\diam(X,\omega(t))\leq C(T-t)^{\frac{1}{2}}$ (see \cite{ST}), and it is also proved in \cite{SSW} for some special Fano fibrations (also discussed earlier).

In the case of infinite time solutions, we expect:
\begin{conjecture}\label{diam2}
Let $X$ be a compact K\"ahler manifold and $\omega(t)$ a solution of the K\"ahler-Ricci flow \eqref{krf} defined on $[0,\infty)$. Then there is a constant $C>0$ such that
$$\diam\left(X,\frac{\omega(t)}{t}\right)\leq C,$$
for all $t\geq 0$.
\end{conjecture}

Recall that the existence of an infinite time solution is equivalent to $K_X$ being nef. As mentioned earlier, the Abundance Conjecture for K\"ahler manifolds would imply that $K_X$ is semiample, so in particular $\kappa(X)\geq 0$. Assuming $K_X$ is semiample, if $\kappa(X)=0$ then $X$ is Calabi-Yau (by Lemma \ref{cycy}) and we even have that  $\diam(X,\omega(t))\leq C,$ thanks to Theorem \ref{ena}. If $\kappa(X)=n$ then $K_X$ is nef and big, and in this case Conjecture \ref{diam2} is proved in \cite{GSW} when $n=2$ and in \cite{TiZ2} when $n\leq 3$ (see also \cite{Gu} for further progress).
Lastly, when $0<\kappa(X)<n$ (this is the setup of Theorem \ref{colla}), Conjecture \ref{diam2} seems to be open even in the case when $n=2,\kappa(X)=1$.

\subsection{Volume growth}

The growth of the total volume of $X$ as $t\to\infty$ for an infinite time solution is a delicate issue as well. Indeed, the following conjecture is equivalent to the Abundance Conjecture in the general K\"ahler case:

\begin{conjecture}\label{vol2}
Let $X$ be a compact K\"ahler manifold and $\omega(t)$ a solution of the K\"ahler-Ricci flow \eqref{krf} defined on $[0,\infty)$. Then $\kappa(X)\geq 0$ and
there is a constant $C>0$ such that
\begin{equation}\label{vol4}
C^{-1}t^{\kappa(X)}\leq \vol(X,\omega(t))\leq Ct^{\kappa(X)},
\end{equation}
for all $t\geq 0$.
\end{conjecture}

Since the Abundance Conjecture in the K\"ahler case is now known for $n\leq 3$ by \cite{CHP}, so is this conjecture. Indeed, by the Abundance Conjecture we have that $K_X$ is semiample, and then as explained in Section \ref{sectinf} we get a fiber space $f:X\to Y$ onto a normal projective variety of dimension $\kappa(X)\geq 0$, such that $K_X^{\otimes \ell}=f^*L$ for an ample line bundle $L$ on $Y$. This implies that $c_1(K_X)^p=0$ for all $p>\kappa(X)$, and so
\[\begin{split}
\vol(X,\omega(t))&=\int_X (\omega_0+2\pi tc_1(K_X))^n\\
&=c t^{\kappa(X)}\int_X\omega_0^{n-\kappa(X)}\wedge c_1(K_X)^{\kappa(X)}+O(t^{\kappa(X)-1}),
\end{split}\]
where $c>0$ and
$$\int_X\omega_0^{n-\kappa(X)}\wedge c_1(K_X)^{\kappa(X)}=\frac{1}{\ell^{\kappa(X)}}\int_X \omega_0^{n-\kappa(X)}\wedge f^*c_1(L)^{\kappa(X)}>0.$$
The fact that conversely Conjecture \ref{vol2} implies the Abundance Conjecture follows easily from \cite[Theorem 5.5]{Nak} (which is the extension of \cite[Corollary 6.1.13]{KMM} to the K\"ahler case), since \eqref{vol4} implies that $K_X$ is abundant (i.e. its numerical dimension is equal to $\kappa(X)$).

We also have the following simple observation, related to Conjecture \ref{vol2}:
\begin{proposition}
Let $X$ be a compact K\"ahler manifold and $\omega(t)$ a solution of the K\"ahler-Ricci flow \eqref{krf} defined on $[0,\infty)$.
Then there is a constant $C>0$ such that
\begin{equation}\label{vol3}
\vol(X,\omega(t))\leq C,
\end{equation}
if and only if $X$ is Calabi-Yau.
\end{proposition}
\begin{proof}
If $X$ is Calabi-Yau then $\vol(X,\omega(t))$ is clearly constant. Conversely, if \eqref{vol3} holds then expanding
$$\vol(X,\omega(t))=\int_X (\omega_0+2\pi tc_1(K_X))^n,$$
we see that we must have
$$\int_X\omega_0^{n-1}\wedge c_1(K_X)=0.$$
Since $c_1(K_X)$ is nef, the Khovanskii-Teissier inequality for nef classes (see e.g. \cite{FX})
$$\int_X\omega_0^{n-1}\wedge c_1(K_X)\geq\left(\int_X\omega_0^{n-2}\wedge c_1(K_X)^2\right)^{\frac{1}{2}}\left(\int_X\omega_0^n\right)^{\frac{1}{2}},$$
implies that $\int_X\omega_0^{n-2}\wedge c_1(K_X)^2=0.$ The result now follows from the Hodge-Riemann bilinear relations on K\"ahler manifolds, proved in \cite{DN}. Indeed, following their notation, we set
$\omega_1=\dots=\omega_{n-1}:=\omega_0$, so that the condition $\int_X\omega_0^{n-1}\wedge c_1(K_X)=0$ says that $c_1(K_X)\in P^{1,1}(X)$, while the condition
$\int_X\omega_0^{n-2}\wedge c_1(K_X)^2=0$ says that $Q(c_1(K_X),c_1(K_X))=0$. Since by \cite[Theorem A]{DN} the bilinear form $Q$ is positive definite on $P^{1,1}(X)$, this
implies that $c_1(K_X)=0$, and so $X$ is Calabi-Yau.
\end{proof}

\subsection{Singularity types} Following \cite{Ha2} we say that a solution $\omega(t)$ of the K\"ahler-Ricci flow \eqref{krf} on a compact K\"ahler manifold $X$, defined on a maximal time interval $[0,T), T<\infty$, develops a {\em type I} singularity at time $T$ if we have
$$\sup_{X\times[0,T)} (T-t)|{\rm Rm}(\omega(t))|_{\omega(t)}<+\infty,$$
and a {\em type IIa} singularity if
$$\sup_{X\times[0,T)} (T-t)|{\rm Rm}(\omega(t))|_{\omega(t)}=+\infty.$$
While type I singularities are easy to construct, this is not the case for type IIa singularities. The first compact examples, for the Ricci flow on Riemannian manifolds, were constructed in \cite{GZ}. Since these examples are not K\"ahler, this leaves open the following:

\begin{problem}
Construct a type IIa finite time singularity of the K\"ahler-Ricci flow on a compact K\"ahler manifold.
\end{problem}

For example, when $X$ is Fano and $[\omega_0]=\lambda c_1(X)$ for some $\lambda>0$, then the singularity being type I is equivalent to the curvature remaining uniformly bounded for all $t\geq 0$, after we renormalize the flow to have constant volume (the normalized flow exists for all $t \geq 0$). It seems very likely that there exist Fano manifolds where the normalized flow does not have uniformly bounded curvature, but no examples have been found yet.

On the other hand, in the Fano case Perelman has proved a uniform scalar curvature bound (see \cite{ST}), which in the unnormalized flow translates to the estimate
\begin{equation}\label{scal2}
R(t)\leq \frac{C}{T-t},
\end{equation}
on $X\times[0,T)$. It is not known whether \eqref{scal2} holds for all finite time singularities of the K\"ahler-Ricci flow, but see \cite{Zh} for partial results.

We now discuss infinite time solutions, and their singularity types ``at infinity''. Again following \cite{Ha2}
we say that a solution $\omega(t)$ of the K\"ahler-Ricci flow \eqref{krf} on a compact K\"ahler manifold $X$, defined for all $t\geq 0$, develops a {\em type III} singularity at infinity if we have
$$\sup_{X\times[0,\infty)} t|{\rm Rm}(\omega(t))|_{\omega(t)}<+\infty,$$
and a {\em type IIb} singularity if
$$\sup_{X\times[0,\infty)} t|{\rm Rm}(\omega(t))|_{\omega(t)}=+\infty.$$
A simple scaling argument shows that type III is equivalent to the solution of the normalized flow \eqref{krf2a} having uniformly bounded curvature for all $t\geq 0$, and type IIb to its negation. When $n=1$ it follows from work of Hamilton \cite{Ha3} that all infinite time solutions are type III. In the case of the Ricci flow on real $3$-dimensional compact Riemannian manifolds, all infinite time solutions are type III thanks to \cite{Ba}. However, in the K\"ahler case when $n=2$ there are type IIb solutions. It is enough to take $X$ a $K3$ surface, and $\omega$ a Ricci-flat K\"ahler metric on $X$, which exists thanks to Yau \cite{Ya}. Then $\omega$ cannot be flat since $\chi(X)=24\neq 0$, so $\sup_X|{\rm Rm}(\omega)|_{\omega}=c>0$.
Then $\omega(t)\equiv\omega$ is a static solution of the K\"ahler-Ricci flow \eqref{krf}, and
$$\sup_{X\times[0,\infty)} t|{\rm Rm}(\omega(t))|_{\omega(t)}=\sup_{t\in[0,\infty)} ct=+\infty,$$
so this solution is type IIb.

The following theorem was proved in \cite{TZ}:
\begin{theorem}\label{singt}
Let $X^n$ be a compact K\"ahler manifold with $K_X$ semiample, and consider a solution of the K\"ahler-Ricci flow \eqref{krf} (which necessarily exists for all positive time).
\begin{enumerate}
\item Assume $\kappa(X)=0$. Then the solution is type III if and only if $X$ is a finite unramified quotient of a torus
\item Assume $\kappa(X)=n$. Then the solution is type III if and only if $K_X$ is ample
\item Assume $0<\kappa(X)<n$, and let $X_y$ be any smooth fiber of the fibration $f:X\to Y$ defined by sections of $K_X^{\otimes\ell}$, for $\ell$ large.
If $X_y$ is not a finite unramified quotient of a torus then the solution is type IIb, while if $X_y$ is a finite unramified quotient of a torus and $S=\emptyset$ (i.e. $Y$ is smooth and $f$ is a submersion)
then the solution is type III.
\end{enumerate}
In particular, in all these cases the singularity type does not depend on the initial metric.
\end{theorem}

Another proof of (2) was obtained in \cite{Gu}. This theorem leaves open the case when the generic fiber $X_y$ is a finite unramified quotient of a torus, but $f$ is not a submersion everywhere. In this case the solution can be either type IIb or type III, depending on the singularities and multiplicities of the fibers contained in $S$. A complete classification when $n=2$ is obtained in \cite{TZ}, where it is also shown that in general dimensions if any component of singular fiber is uniruled then the solution is of type IIb. It remains to understand what happens when no such component is uniruled.

Considering Theorem \ref{singt}, it is then natural to conjecture:

\begin{conjecture}\label{ssing}
Let $X$ be a compact K\"ahler manifold with $K_X$ nef, so every solution of the K\"ahler-Ricci flow \eqref{krf} exists for all positive time. Then the singularity type at infinity does not depend on the choice of the initial metric $\omega_0$.
\end{conjecture}

As mentioned above, this conjecture is only known when $n\leq 2$, thanks to \cite{TZ}.

\end{document}